\documentclass[dvipsnames]{amsart}
\usepackage{pict2e}
 
\usepackage{amsmath,amssymb,amsthm}
\usepackage{extarrows}
\usepackage{fullpage}
\usepackage{changes}
\usepackage{soul}
\usepackage{todonotes}

\allowdisplaybreaks

\usepackage{color}
\usepackage{hyperref}
\hypersetup{
	colorlinks=true,
	linkcolor=blue,
	citecolor=blue,
	urlcolor=blue
}

\usepackage{graphicx}
\usepackage{mathrsfs}
\usepackage{tikz-cd}
\usepackage{array}
\usepackage{multirow}
\usepackage{soul}
\usepackage{quiver}
\usepackage{thmtools}
\usepackage{thm-restate}

\usepackage{xcolor}

\usepackage{tikz}

\newtheorem{theoremAlph}{Theorem}
\newtheorem{definAlph}{Definition}
  
\newtheorem{theorem}{Theorem}[section]
\newtheorem{lemma}[theorem]{Lemma}	
\newtheorem{proposition}[theorem]{Proposition}
\newtheorem{corollary}[theorem]{Corollary}
\theoremstyle{definition}
\newtheorem{definition}[theorem]{Definition} 
\newtheorem{remark}[theorem]{Remark}	
\newtheorem{example}[theorem]{Example}

\newcommand{\longdownarrow}{\lower 1.4ex\hbox{\begin{picture}(18,18)(0,0)
\thicklines
\put(0,18){\vector(0,-1){18}}
\end{picture}}}
\newcommand{\longsearrow}{\lower 1.4ex\hbox{\begin{picture}(18,18)(0,0)
\thicklines
\put(0,18){\vector(1,-1){18}}
\end{picture}}}
\newcommand{\longssearrow}{\lower 1.4ex\hbox{\begin{picture}(18,18)(0,0)
\thicklines
\put(0,18){\vector(1,-2){18}}
\end{picture}}}
\newcommand{\longeearrow}{\lower 1.4ex\hbox{\begin{picture}(0,0)(9,9)
\thicklines
\put(0,18){\vector(1,0){18}}
\end{picture}}}
 \usepackage{relsize}
\theoremstyle{definition}

\numberwithin{equation}{section}

\newcommand{\W}{\mathbb{W}}
\newcommand{\V}{\mathbb{V}}
\newcommand{\R}{\mathbb{R}}
\newcommand{\N}{\mathbb{N}}
\newcommand{\Z}{\mathbb{Z}}
\newcommand{\G}{\mathbb{G}}

\renewcommand{\H}{\mathbb{H}}
\title[Stokes' Theorem on positively graded groups]{Stokes' Theorem on positively graded groups} 

\thanks{F.T. and V.M. were supported by the MIUR Excellence Department Project awarded to the Department of Mathematics, University of Pisa, CUP I57G22000700001.
F.T. would like to thank the Centro di Ricerca Matematica Ennio De Giorgi and the Scuola Normale Superiore for their hospitality and support. She also acknowledges the support and hospitality of the Department of Mathematics of Pisa University during her visit.
V.M. acknowledges the support of PRIN 2022PJ9EFL {\em Geometric Measure Theory: Structure of Singular Measures, Regularity Theory and Applications in the Calculus of Variations}, funded by the European Union--NextGenerationEU, CUP E53D23005860006}

\author[V.~Magnani]{Valentino Magnani}
\author[F.~Tripaldi]{Francesca Tripaldi}

\address[V. Magnani]{ Dipartimento di Matematica\\ Universit\`a di Pisa\\ Largo Bruno Pontecorvo 5\\ 56127, Pisa, Italy}
\email{valentino.magnani@unipi.it}

\address[F. Tripaldi]%
{Department of Pure Mathematics, University of Leeds, Woodhouse, LS2 9JT Leeds, UK} 
\email{f.tripaldi@leeds.ac.uk}

\date{}

\begin{document}
\maketitle
\begin{abstract}
  This paper studies the validity of Stokes’ theorem for differential subcomplexes naturally adapted to the noncommutative geometry of positively graded Lie groups, with particular emphasis on Carnot groups. We introduce geometric conditions under which Stokes-type formulae hold for the Rumin complex and for a new family of spectral complexes associated with the homogeneous weight filtration of the de Rham complex. In particular, the spectral complexes allow us to recover the validity of Stokes’ theorem on locally smooth intrinsic graphs. This is achieved by showing that the corresponding Stokes' formulae are governed entirely by the degree of the underlying submanifolds.
  
  Our approach also reveals that both the Rumin complex and the spectral complexes can be interpreted directly in terms of the classical de Rham complex through the Leibniz rule and integration over suitable classes of submanifolds, namely $R$-manifolds and spectral manifolds, respectively. 
  
  Finally, motivated by this interaction between homogeneous weights and degrees of submanifolds, we propose a notion of currents naturally adapted to these subcomplexes.
\end{abstract}

\tableofcontents

\section{Introduction}

Stokes' theorem represents one of the most important results in integration theory on oriented manifolds. It expresses a fundamental principle in differential geometry: global geometric information can be recovered from local differential data via boundary interactions. By relating the integral of a differential form over a manifold to the integral of the form itself over its boundary, it provides a unifying framework that generalises classical results such as the fundamental theorem of calculus, Green’s theorem, and the divergence theorem, and represents one of the central bridges between analysis, geometry, and topology.

The importance of Stokes' theorem goes far beyond mathematics, with applications in several areas of theoretical physics. It underlies the passage from macroscopic integral laws to local differential equations. It allows global balance principles, formulated as fluxes across the boundary of a region, to be expressed in terms of pointwise differential relations, thus linking conservation laws with their differential counterparts, as in the classical formulation of Maxwell’s equations and fluid dynamics.

In the broad context of differential geometry, Stokes' theorem underpins the topological properties of the de Rham cohomology of smooth manifolds. Indeed, it implies that the integration of differential forms over singular chains vanishes on boundaries and depends only on cohomology classes, thereby inducing a natural isomorphism between the de Rham cohomology of a smooth manifold and its singular homology with real coefficients. Stokes’ theorem also implies the homotopy invariance of de Rham cohomology: homotopies produce chains whose boundary measures the difference between the endpoints, so the pullbacks of a closed form along homotopic maps define the same cohomology class and therefore have the same integrals over corresponding cycles.

In the context of de Rham currents, Stokes' theorem is the basic principle for the definition of the boundary operator. Viewed as a generalised $k$-dimensional manifold, a $k$-current $T$, or de Rham current, is a continuous linear functional defined on any smooth, compactly supported $k$-form $\alpha$. Its boundary is precisely introduced assuming the validity of Stokes' theorem $\partial T(\alpha)=T(d\alpha)$. Any smooth, oriented, $k$-dimensional submanifold $\Sigma$ with boundary induces a $k$-current $[[\Sigma]]$, and Stokes' theorem shows that the boundary current is defined by the geometric boundary $\partial\Sigma$ of $\Sigma$ along with its induced orientation: 
\begin{equation*}
\partial [[\Sigma]](\alpha)=\int_{\Sigma}d\alpha=\int_{\partial\Sigma}\alpha = [[\partial \Sigma]](\alpha)\,.
\end{equation*}
The abstract notion of boundary allows one to formulate the celebrated Plateau problem, namely the existence of a mass minimising current among all currents with prescribed boundary. The solution of this problem is a foundational result of geometric measure theory in Euclidean spaces; see, for instance, \cite{Federer69,Morgan2016-bk}.

In the context of geometric measure theory in Carnot groups the situation is considerably subtler. Finding the proper geometric integration requires the
right domains of integration as intrinsic smooth objects with respect to the
geometry of the group. Since the first developments of geometric measure theory in 
Heisenberg groups, it appeared that the right smooth objects are 
intrinsic graphs, as the building blocks for the notion of intrinsic rectifiability 
\cite{FSSC01,FSSC6}. An intrinsic graph is a set associated with a pair $(\W,\V)$ of complementary subgroups (Definition~\ref{def: complementary subgroups}). It is a set of the form $\{w\phi(w):w\in A\}$, where $A\subset\W$ is an open set and $\phi:A\to\V$ is the defining mapping.
Intrinsic graphs, introduced and studied in \cite{FMS14,FranchiSerapioni}, are intrinsically regular sets when their defining mapping is continuously intrinsically differentiable. They also appear as regular sets arising from an intrinsic differential calculus in Carnot groups, using the notion of Pansu differentiability, \cite{Mag14}. In sum, when the ambient space is a Carnot group, or more generally a positively graded group, the geometry of regular subsets is much more rigid, and intrinsic graphs with intrinsically differentiable defining mapping become the proper class of intrinsically $C^1$ smooth submanifolds. 
Crucially, not every differential form of the de Rham complex interacts well with smooth intrinsic graphs from the point of view of integration. By contrast, the forms that appear in the Rumin complex $(E_0^\bullet,d_c)$ are compatible with the underlying sub-Riemannian structure, given their fundamental role in the definition of intrinsic graphs, as shown in \cite{FranchiSerapioni}. 

This observation naturally leads to the question of whether the Rumin differential $d_c$ admits a Stokes-type theorem analogous to the classical one. Beyond the fact that simple Rumin forms are in a one-to-one correspondence with intrinsic graphs (as observed explicitly in Section \ref{section intrinsic graphs}), the Rumin complex plays a central role in the sub-Riemannian setting, by providing a differential complex whose operators reflect the grading of the underlying Lie algebra while retaining many of the structural properties of the de Rham complex. Our approach to establishing conditions under which a Stokes' formula holds for the Rumin complex on an orientable $C^1$ manifold $\Sigma$ with boundary $\partial\Sigma$ is based on relating the operator $d_c$
 to the exterior derivative $d$, thereby allowing us to exploit the validity of the classical Stokes' theorem for differential forms.

Since the Rumin differential is defined as $d_c=\Pi_0d\Pi_E\alpha$, one of the main difficulties in deriving a Stokes-type formula from the classical one lies in the presence of the projection operators $\Pi_0$ and $\Pi_E$. A natural strategy is therefore to identify a class of manifolds on which the contribution of this projection becomes unnecessary or can be controlled directly.

The key ingredient underlying our approach is the characterisation of the space of smooth differential forms $\Omega^\bullet(\G)$ on a positively graded group $\G$ though the algebraic differential $d_0$ and its adjoint $\delta_0$. The latter is defined once a suitable scalar product on $\Omega^\bullet(\G)$ is fixed (see Subsection \ref{subsection: introducing a scalar product}). In particular, the associated algebraic Laplacian $\Box_0:=d_0\delta_0+\delta_0d_0$ yields a Hodge-type orthogonal decomposition of $\Omega^\bullet(\G)$ (as shown in Proposition \ref{prop: Hodge decomposition for Box0}), so that
\begin{align*}
    \Omega^\bullet(\G)=\operatorname{Im}d_0\oplus\ker\Box_0\oplus\operatorname{Im}\delta_0\,.
\end{align*}
This decomposition allows us to isolate the component of a differential form that effectively contributes to the Rumin differential and, crucially, to identify those whose vanishing renders the projection operators irrelevant in the context of integration. More precisely,
\begin{itemize}
    \item \(E_0^\bullet=\ker\Box_0\cap\Omega^\bullet(\G)\);
    \item for every \(\alpha\in E_0^\bullet\), one has $\Pi_E\alpha=\alpha+\operatorname{Im}\delta_0$;
    \item for every \(\alpha\in E_0^\bullet\), $d\Pi_E\alpha
        \in
        \big(\operatorname{Im}d_0\big)^\perp\cap\Omega^\bullet(\G)
        =
        \big(\ker\Box_0\oplus\operatorname{Im}\delta_0\big)\cap\Omega^\bullet(\G)$.
\end{itemize}
These observations suggest introducing a class of submanifolds on which forms belonging to \(\operatorname{Im}\delta_0\) integrate to zero, thereby eliminating the contribution of the projection operators in Stokes-type formulas for the Rumin complex.
Indeed, if this is the case, we obtain Stokes' theorem for the Rumin differential $d_c$ directly from the classical Stoke's formula, since 
\begin{align*}
        \int_{\partial\Sigma}\alpha\underbrace{=}_{(\ast)}\int_{\partial\Sigma}\Pi_E\alpha=\int_{\Sigma}d\Pi_E\alpha\underbrace{=}_{(\ast)}\int_{\Sigma}d_c\alpha\,,
    \end{align*}
    where by $(\ast)$ we denote the equalities that follow from the fact that the integrals of forms in $\operatorname{Im}\delta_0$ over  $\partial\Sigma$ and $\Sigma$ vanish.
\begin{definAlph}
    [$R$-manifolds, Definition \ref{def: R manifolds}]
    Let \(\Sigma\subset\mathbb{G}\) be an oriented \(k\)-dimensional \(C^1\)-manifold, with or without boundary. We say that \(\Sigma\) is an \(R\)-manifold if
    \begin{align*}
        \int_\Sigma\eta=0
        \qquad
        \text{for every }
        \eta\in\Omega^k(\G)\cap\operatorname{Im}\delta_0 .
    \end{align*}
\end{definAlph}

\begin{theoremAlph}
    [Stokes' theorem on \(R\)-manifolds, Theorem \ref{thm: stokes on R manifolds}]\label{thm: A}
    
    Let \(\Sigma\subset\G\) be a \(k\)-dimensional \(R\)-manifold whose boundary \(\partial\Sigma\) is also an \(R\)-manifold. Then
    \begin{align*}
        \int_{\partial\Sigma}\alpha
        =
        \int_\Sigma d_c\alpha
        \qquad
        \text{for all }
        \alpha\in E_0^{k-1}.
    \end{align*}
\end{theoremAlph}
In view of Theorem \ref{thm: A}, the Rumin complex, and in particular the action of the Rumin differential, admits an interpretation through Stokes' theorem. From this perspective, the Rumin complex can be understood not as a genuinely new differential complex, but rather as the de Rham complex viewed through integration along suitable classes of manifolds. More precisely, the differential $d_c$
 arises from the exterior derivative once forms are tested only against manifolds for which the contributions coming from $\operatorname{Im}\delta_0$ vanish. Section \ref{section: Rumin leibniz} makes this viewpoint precise by showing that the Rumin complex can be reconstructed from the de Rham complex using only the Leibniz rule for the exterior derivative together with integration over appropriate manifolds, such as $R$-manifolds.

A distinctive feature of the Heisenberg groups is that this reconstruction does not require selecting preferred representatives in $\operatorname{Im}\delta_0$
 when applying the Leibniz rule, reflecting the intrinsic nature of the Rumin complex in this setting (see Subsection \ref{subsection: Rumin complex Leibniz Hn}). In contrast, already for $\H^1\times\R$ (see Subsection \ref{subsection: rumin on HtimesR}), recovering the Rumin complex from the Leibniz rule and integration necessarily involves a choice of representatives in $\operatorname{Im}\delta_0$. This highlights a phenomenon already well known for the Rumin complex on arbitrary Carnot groups, namely that its construction depends explicitly on the choice of scalar product introduced in Subsection \ref{subsection: introducing a scalar product}.

As observed in Proposition \ref{prop: when R manifold is not needed}, requiring both \(\Sigma\) and \(\partial\Sigma\) to be \(R\)-manifolds is a rather restrictive assumption, and one that is not always necessary. In particular, this requirement can be avoided whenever
\[
    \Pi_E\alpha=\alpha
    \ \text{ or }\
    d\Pi_E\alpha\in\ker\Box_0\cap\Omega^\bullet(\G)\ \text{ for a given }\ \alpha\in E_0^\bullet\,.
\]
This phenomenon already appears in the simple setting of Heisenberg groups \(\mathbb H^n\). Indeed, it is well known that Stokes' theorem for the Rumin complex is valid on what we define as \textit{locally smooth intrinsic graphs with boundary}, that is $C^1$ orientable manifolds $\Sigma$ with boundary $\partial\Sigma$, where both $\Sigma$ and $\partial\Sigma$ are, up to a negligible set, locally the images of a $C^2$ intrinsic graph (see Definition \ref{def:LSIG}). We also recall that the same conclusion also holds for $C^1_{\mathbb H}$-regular submanifolds with boundary by means of smooth approximation arguments \cite{StokesFranchi,dimarco2025submanifoldsboundarysubriemannianheisenberg}. However, as we show in Proposition \ref{prop: stokes on heisenberg groups}, there is a substantial difference between smooth intrinsic graphs of low dimension and those of low codimension. More precisely, in the \((2n+1)\)-dimensional Heisenberg group \(\mathbb H^n\),
\begin{itemize}
    \item if \(0<k\le n\), smooth intrinsic graphs are \(R\)-manifolds;
    \item if \(n+1\le k<2n\), smooth intrinsic graphs are generally \textit{not} \(R\)-manifolds.
\end{itemize}
Nevertheless, the special properties of the Rumin complex in $\H^n$ compensates for this failure. Indeed, for every \(\alpha\in E_0^k\), one has $\Pi_E\alpha=\alpha$ for $k\ge n+1$, while $
    d\Pi_E\alpha
    \in
    \ker\Box_0\cap\Omega^{k+1}(\mathbb H^n)$ for $k\ge n$.
These two properties are sufficient to recover the validity of Stokes' theorem on locally smooth intrinsic graphs with boundary also in the high-dimensional regime.

The situation changes drastically for general positively graded groups, and even for Carnot groups. In Proposition \ref{prop: stokes intrinsic graphs H1timesR}, we explicitly show that Stokes' theorem for the Rumin complex may fail on locally smooth intrinsic graphs with boundary in the four-dimensional Carnot group \(\mathbb H^1\times\mathbb R\).

A remarkable feature of our approach is that all the smooth intrinsic graphs that are shown in this paper to be $R$-manifolds are characterised solely through their degree. This concept appears in \cite{Mag13Vit} to study explicit integration formulas to compute the spherical measure constructed from a homogeneous distance of the group when it is restriced to smooth submanifolds. Area formulas for the spherical measures in Carnot groups or general positively graded group show that the degree of a submanifold coincides with its Hausdorff dimension with respect to
any fixed homogeneous distance of the group.

Indeed, Corollary \ref{cor: integrating forms on smooth intrinsic graphs} shows that integration over a smooth intrinsic graph annihilates all differential forms whose weight exceeds the degree of the graphing subgroup. Consequently, the problem of determining whether a given locally smooth intrinsic graph is an $R$-manifold reduces to comparing the degree of the of the submanifold with the possible weights of forms appearing in $\operatorname{Im}\delta_0$. In particular, the validity of Stokes-type formulas for these classes of locally smooth intrinsic graphs with boundary is governed entirely by the grading structure of the group and by the degree of the submanifold, rather than by finer geometric properties of the graph itself.

This degree-related perspective underlying the notion of $R$-manifold naturally extends beyond the Rumin complex and leads to the study of a new family of differential subcomplexes, namely the spectral complexes
$$\{(E_{j,l}^{\bullet,\bullet},\Delta_j)\}_{j\in I_{\bullet,\bullet}}$$
introduced in \cite{tripaldi2026spectralcomplexestruncatedmulticomplexes}. These complexes arise by interpreting the de Rham complex of a positively graded Lie group as a truncated multicomplex and considering the spectral sequence associated with the filtration by homogeneous weight. In this framework, the exterior differential decomposes according to the grading \eqref{eq: decomposition of d} and the corresponding spectral sequence isolates, page by page, the components of the differential responsible for specific jumps in weight.

To describe these spectral complexes, one introduces the spaces $Z_r^{p,\bullet}$ and $B_r^{p,\bullet}$, which provide explicit representatives for the pages of the spectral sequence and admit a natural reinterpretation in terms of Rumin forms and the Rumin differential. More precisely, the decomposition
\begin{align*}
    d_c=\sum_{j\in I_{p,k}}d_c^j
\end{align*}
of the Rumin differential into homogeneous operators of bidegree $(j,1-j)$ gives rise to a corresponding family of spectral differentials
$$\Delta_j\colon E_{j,l}^{p,\bullet}\longrightarrow E_{m,j}^{p+j,\bullet}$$
each encoding the contribution of the exterior derivative associated with a weight increase of exactly $j$. The resulting spectral complexes therefore refine the Rumin complex by separating the different homogeneous components of the differential structure.

Similarly to what happens for the Rumin complex, these operators admit a direct interpretation through Stokes' theorem. Indeed, we show that each differential $\Delta_j$ coincides, up to exact forms and terms of sufficiently high weight, with the classical exterior derivative. Consequently, when integrating over an orientable $C^1$ manifold $\Sigma$ with boundary such that $\operatorname{deg}(\Sigma)-\operatorname{deg}(\partial\Sigma)=j$, all irrelevant terms vanish for purely degree-theoretic reasons, and one obtains a Stokes' formula associated specifically with the spectral differential $\Delta_j$. In this way, the geometry of the manifold determines which component of the spectral decomposition governs the integration process. More precisely, this can be rephrased as follows.
\begin{theoremAlph}
[Stokes' theorem for spectral complexes, Theorem \ref{thm: stokes for spectral complexes}] \label{thm: B}Fix a bidegree $(p,k-1)$ for which the operator $$\Delta_j\colon E_{j,l}^{p,k-1-p}\longrightarrow E_{m,j}^{p+j,k-p-j}$$
is nonzero, for some admissible choice of $l,m\in\mathbb{Z}^+$. Let $\Sigma\subset\mathbb{G}$ be an orientable $k$-dimensional $C^1$ manifold such that $\operatorname{deg}(\Sigma)=p+j$, and assume that $\partial\Sigma$ is also an orientable $C^1$ manifold with $\operatorname{deg}(\partial\Sigma)=p$. Then
\begin{align*}
    \int_{\partial\Sigma}\alpha=\int_\Sigma \Delta_j\alpha\ \text{ for every }\alpha\in E_{j,l}^{p,k-1-p}\,.
\end{align*}
\end{theoremAlph}

In contrast with the Rumin complex, the validity of Stokes' theorem for spectral complexes depends crucially on the degree of the underlying submanifold. In Remark~\ref{rmk: spectral stokes no on R manifolds}, we show that the operators $\Delta_j\colon E_{j,l}^{p,\bullet}\longrightarrow E_{m,j}^{p+j,\bullet}$ do not, in general, satisfy a Stokes' formula on arbitrary $R$-manifolds. The reason is that the representatives defining the action of $\Delta_j$ are only determined modulo higher-weight terms and $d$-exact forms whose primitives need not vanish under integration on $R$-manifolds as soon as $j\ge 2$. This shows that the correct framework for Stokes' theorem for spectral complexes is instead dictated by degree compatibility.
This phenomenon naturally motivates the introduction of distinguished classes of submanifolds adapted to spectral complexes, for which a Stokes-type theorem for the operators $\Delta_j$
 becomes valid.

\begin{definAlph}
[Spectral manifolds, Definition \ref{def: spectral manifold}]\em 
Let $\Sigma\subset\mathbb G$ be a $k$-dimensional orientable $C^1$ manifold with boundary $\partial\Sigma$. Denote by $P_h$ the set of weights of nontrivial Rumin forms $E_0^h$ in degree $h$. We say that $\Sigma$ is a \emph{spectral manifold with boundary} if
\[
\deg(\Sigma)\in P_k
\qquad\text{and}\qquad
\deg(\partial\Sigma)\in P_{k-1}.
\]
In case $\Sigma$ has no boundary, so it is orientable and only the condition $\deg(\Sigma)\in P_k$ is satisfied, we say that $\Sigma$ is a {\em spectral manifold}.
\end{definAlph}

Crucially, by Lemma \ref{lem: deg boundary less}, for a spectral manifold with boundary, we have necessarily that  $\operatorname{deg}(\partial\Sigma)<\operatorname{deg}(\Sigma)$. This means that, by construction of the spectral complexes, there exists some choices of $l,m\in\mathbb Z^+$ so that $$\Delta_j\colon E_{j,l}^{\operatorname{deg}(\partial\Sigma),k-1-\operatorname{deg}(\partial\Sigma)}\longrightarrow E_{m,j}^{\operatorname{deg}(\Sigma),k-\operatorname{deg}(\Sigma)}$$
since $j=\operatorname{deg}(\Sigma)-\operatorname{deg}(\partial\Sigma)\ge 1$. This property directly implies that locally smooth intrinsic graphs are indeed spectral manifolds with boundary.

\begin{theoremAlph}
 [Stokes' theorem on locally smooth intrinsic graphs with boundary, Corollary \ref{cor: stokes thm for intrinsic graphs}]
Let $\Sigma\subset\mathbb{G}$ be an orientable $k$-dimensional locally smooth intrinsic graph with boundary $\partial\Sigma$, then
\[
\int_{\partial\Sigma}\alpha
=
\int_\Sigma \Delta_j \alpha
\quad
\text{for every } \alpha\in E_{j,l}^{p,k-1-p}\,,
\]
where $p=\operatorname{deg}(\Sigma)$, $j=\operatorname{deg}(\Sigma)-\operatorname{deg}(\partial
\Sigma)$, and some $l\in\mathbb{N}$ such that $E_{j,l}^{p,k-1-p}\neq 0$.
\end{theoremAlph}
The fact that locally smooth intrinsic graphs with boundary satisfy this degree condition is a direct consequence of Proposition \ref{prop: complementary subgroups vs Rumin forms}, so that for $\operatorname{graph}(\phi)$ a $k$-dimensional $(\W,\V)$-graph, we have that $\operatorname{deg}(\operatorname{graph}(\phi))=\operatorname{deg}(\W)\in P_k$ for each degree.

In Section \ref{section: spectral Leibniz}, we reinterpret the spectral complexes $\{(E_{j,l}^{\bullet,\bullet},\Delta_j)\}_{j\in I_{\bullet,\bullet}}$ directly in terms of the de Rham complex and Theorem \ref{thm: B}. More precisely, we show that the action of the spectral differentials $\Delta_j$ can be recovered from the exterior derivative through the Leibniz rule together with integration over suitable spectral manifolds with boundary. In contrast with the situation for the Rumin complex, no choice of preferred representatives is required when applying the Leibniz rule. The construction is entirely determined by the homogenous weight of the forms involved and the de Rham differential itself, highlighting the fact that even though the spaces of forms $E_{j,l}^{p,\bullet}=Z_j^{p,\bullet}\cap\big(B_l^{p,\bullet}\big)^\perp$ are defined using an appropriate scalar product, the action of the spectral differential does not depend on it.

Spectral manifolds, spectral complexes, and Stokes' theorem also serve as our guiding ideas to find the associated theory of currents. In a few words, as the de Rham currents are associated with the de Rham complex, in positively graded groups one should expect to define currents by duality from suitable subcomplexes of the de Rham complex.

In Heisenberg groups, currents were studied in \cite{FSSC6}, as dual of the Rumin complex, using intrinsic $\H$-regular submanifolds as the intrinsically smooth domains of integration. Then a number of interesting results on Heisenberg currents subsequently appeared, with applications to geometric measure theory, in \cite{StokesFranchi,FSSC6,Can21jga,Vit22,JulNGolVit2023,DiMJulNGoloVit25,franchi2025currents}. A notion of Rumin currents in Carnot groups has also been proposed and studied in \cite{julia2023flat}. 

In Section~\ref{section: currents}, we define spectral currents as duals of spectral complexes. 
The fact that the boundary is also a spectral current is not an obvious fact and requires careful considerations on the defining properties of the forms appearing in these subcomplexes, viewed as quotient spaces (see Remark \ref{eq: Z and B as quotients}). Just like in the definition of spectral manifold, the degree, or weight, of the current plays a crucial role in their construction and their properties. This idea is encoded in the concept of weighted comass (see Definition \ref{def:weightedcomass}) and allows us to define the mass of a $(p+j,k,m,j)$-spectral current $T\colon \mathscr D_{m,j}^{p+j,k-p-j}\longrightarrow\R$. Note that the boundary $\partial^jT$ of a $(p+j,k,m,j)$-spectral current $T$ turns out to be a $(p,k-1,j,l)$-spectral current.

In the case of an oriented $k$-dimensional $C^1$ spectral manifold $\Sigma$ with boundary $\partial\Sigma$, then by Theorem \ref{thm: B}, we get
\begin{align*}
    [[\partial^j\Sigma]]_p(\omega)=\int_{\partial\Sigma}\omega=\int_\Sigma\Delta_j\omega=[[\Sigma]]_{p+j}(\Delta_j\omega)\ \text{ for any }\omega\in \mathscr D_{j,l}^{p,k-1-p}\,,
\end{align*}
where $\operatorname{deg}(\partial\Sigma)=p$ and $\operatorname{deg}(\Sigma)=p+j$.

This appears to be a unifying view including the well known theories of currents in Heisenberg groups.


\section{The de Rham complex on positively graded Lie groups}
In this paper, we will be studying the validity of Stokes' theorem for both the Rumin complex as well as the spectral complexes introduced in \cite{tripaldi2026spectralcomplexestruncatedmulticomplexes}. In order to make this exposition as self-contained as possible, we provide a brief overview of the main properties of the de Rham complex $(\Omega^\bullet(\G),d)$ on a given connected, simply-connected positively gradable Lie group $\G$. Crucially, we will show that the de Rham complex is a truncated multicomplex, which is the critical property needed for constructing spectral complexes. 

Let us start by fixing some notation. Given an $n$-dimensional real Lie algebra $\mathfrak{g}$ with basis {$(X_1,\ldots,X_n)$}, 
{we will denote by $\mathfrak{g}^\ast$ the dual space, and by $(\theta_1,\ldots,\theta_n)$ the dual basis, namely $\theta_i(X_j)=\delta_{ij}$. 
Real-valued linear functionals on $\mathfrak{g}$, that are elements of $\mathfrak g^\ast$, are also known as 1-covectors.}
\begin{definition}[{Homogeneous and }positively graded Lie groups]\label{def: hom Lie group}
  {If $\mathfrak{g}$ admits a family $\delta_\lambda\colon\mathfrak{g}\to\mathfrak{g}$ of automorphisms for $\lambda>0$, we say that the connected, simply-connected Lie group $\G$ with Lie algebra $\mathfrak{g}$ is \textit{homogeneous}. These automorphisms are called 
  \textit{dilations}, and have the form $\delta_\lambda=\operatorname{Exp}(A\ln\lambda)$, where $A$ is a diagonalisable operator on $\mathfrak{g}$ with positive eigenvalues and $\operatorname{Exp}$ is the exponential on all the endomorphisms of the vector space $\mathfrak{g}$. 

  Let $(X_1,\ldots,X_n)$ be a basis of eigenvectors of $A$. In this basis, $\delta_\lambda$ is represented by the diagonal matrix $\operatorname{Mat}(\delta_\lambda)=\operatorname{diag}(\lambda^{\nu_1},\ldots,\lambda^{\nu_n})$. The eigenvalues repeated according with their multiplicity are $\nu_1,\ldots,\nu_n\in\R$. The distinct eigenvalues $w_1,\ldots,w_s$ are called the \textit{weights} of the dilations and are labelled in increasing order $0<w_1<\cdots<w_s$.} We set $\mathfrak{g}_{w_j}$ as the $w_j$-eigenspace for $A$. Then $\mathfrak{g}$ decomposes into a direct sum
  \begin{align}\label{eq: grading direct sum decomposition}
      \mathfrak{g}=\bigoplus_{j=1}^s\mathfrak{g}_{w_j}\ \text{ such that }[\mathfrak{g}_{w_i},\mathfrak{g}_{w_j}]\subseteq\mathfrak{g}_{w_i+w_j}\ ,\ 1\le i,j\le s\,.
  \end{align}
  This direct sum decomposition is often referred to as a homogeneous structure.

  {When the weights of the homogenous structure \eqref{eq: grading direct sum decomposition} are positive integers, the group $\G$ is said to be \textit{positively graded}. This is equivalent to requiring the weights to be positive rationals, see} \cite[Lemma 3.1.9]{Fischer2016}.

  If in addition one can take $w_1=1$, with $\mathfrak{g}_1$ generating the whole Lie algebra $\mathfrak{g}$ via iterated brackets, the Lie group $\G$ is called \textit{stratifiable} and we refer to the corresponding direct sum decomposition \eqref{eq: grading direct sum decomposition} as {the} stratification (given a stratifiable group $\G$, the choice of stratification is unique up to Lie algebra automorphisms \cite[Proposition 9.2.9]{donne2024metric}). 
  
\end{definition}
As a straightforward consequence of the fact that the direct sum decomposition \eqref{eq: grading direct sum decomposition} is finite, we get that $\mathfrak{g}$ (and hence also the homogeneous Lie group $\G$ with Lie algebra $\mathfrak{g}$) is nilpotent. Note that a given {nilpotent} Lie group $\G$ may admit several non-equivalent homogeneous structures \cite{hakavuori2022gradings}.

Since the {positively graded group} $\G$ we are considering is connected, simply-connected and nilpotent, its exponential map $\exp\colon\mathfrak{g}\to\G$ is a bijection and a global diffeomorphism. This then allows us to extend the dilations to the group. We will keep the same notation for the dilations on the group, so that $\delta_\lambda\colon\G\to\G$ with $\delta_\lambda(\exp X)=\exp(\delta_\lambda X)$ for any $X\in\mathfrak{g}$.


\begin{definition}[Carnot groups]\label{def: carnot groups} {A stratifiable Lie group $\G$ is \textit{Carnot} if it is equipped with a stratification} {and a scalar product on $\mathfrak{g}_1$}. As such, Carnot groups represent the simplest examples of {sub-Riemannian manifolds}. 
\end{definition}
Although Carnot groups are the most important setting for the applications of the results proved in this paper, the subcomplexes we will be constructing require a suitable scalar product to be defined on the whole Lie algebra $\mathfrak{g}$ (see Subsection \ref{subsection: introducing a scalar product}). In the case of a non-stratifiable positively gradable Lie group, it is standard to work with a Riemannian metric on $\G$ adapted to the homogeneous structure \eqref{eq: grading direct sum decomposition}. However, in the case of a Carnot group, where the scalar product is only defined on the first layer $\mathfrak{g}_1$, one should extend it to a scalar product on the whole of $\mathfrak{g}$ by taking a Riemannian metric on $\G$ that \textit{tames} the given {sub-Riemannian} one \cite[Chapter 1]{montgomery2002tour}. There are of course an infinite amount of such Riemannian metrics, but for our purposes it is sufficient to fix one which is adapted to the homogeneous structure (the stratification in this case).

The dilations $\delta_\lambda\colon\mathfrak{g}\to\mathfrak{g}$ with $\lambda>0$ naturally extend to the exterior algebra of covectors $\bigwedge^\bullet\mathfrak{g}^\ast$ and to the space of smooth forms $\Omega^\bullet(\G)$ via the formula
\begin{align*}
    \left(\delta_\lambda\alpha\right)_x(V_1,\ldots,V_k):=\alpha_{{\delta_\lambda x}}(\delta_\lambda V_1,\ldots,\delta_\lambda V_k)\ \text{ where }\alpha\in\Omega^k(\G),\,x\in\G\text{ and }V_1,\ldots,V_k\in\Gamma(T\G)\big\vert_x\cong\mathfrak{g}\,.
\end{align*}

We keep the same notation $\delta_\lambda$ for these extensions. One can check that the $\delta_\lambda$s respect the wedge product, that is
\begin{align*}
    \delta_\lambda(\alpha_1\wedge\alpha_2)=(\delta_\lambda\alpha_1)\wedge(\delta_\lambda\alpha_2)\ \text{ for all }\alpha_1,\alpha_2\in\Omega^\bullet(\G)\,.
\end{align*}
\begin{definition}[Weights of forms]\label{def: weights of forms}
    We say that a form $\alpha\in\Omega^\bullet(\G)$ is \textit{homogeneous} if there exists a $p\in\R^+$ such that $\delta_\lambda\alpha=\lambda^p\alpha$ for all $\lambda>0$. We refer to the number $p$ as the \textit{weight} of $\alpha$ and we will use the shorthand notation $w(\alpha)=p$. An alternative way of expressing that a form $\alpha$ is such that $w(\alpha)=p$ is to say that $\alpha$ has \textit{pure weight }$p$. 
\end{definition}
Notice that 0-forms have weight $0$, while the volume form $\operatorname{vol}$ has weight $Q$ equal to the homogeneous dimension $\sum_{j=1}^sw_j\operatorname{dim}\mathfrak{g}_{w_j}$ of the group $\G$. Such dimension precisely coincides with the Hausdorff dimension of with respect to a fixed homogeneous distance, according to the next definition.

\begin{definition}
Given a positively graded group $\G$, a distance $\rho$ on $\G$ is {\em homogeneous} if 
\begin{align*}
    \rho(xp_1,xp_2)=\rho(p_1,p_2)\ \text{ and }\ \rho(\delta_\lambda(p_1),\delta_\lambda(p_2))=\lambda\rho(p_1,p_2)\ \text{ for all }x,p_1,p_2\in\G\quad \text{and}\quad\lambda>0.
\end{align*}
\end{definition}
Homogeneous distances can be constructed in any homogeneous group, \cite{HebSik90}, and also other general constructions are available, \cite{FSSC5}. In H-type groups, the Cygan-Koranyi distance is another well known example of homogeneous distance, \cite{Cygan81}.

If $\G$ is stratified and equipped with a left invariant sub-Riemannian metric, then the well-known {\em Carnot-Carath\'eodory distance} can be defined and it is another example of homogeneous distance.

 We will denote by $\mathcal{H}^d$ and $\mathcal{S}^d$ the $d$-dimensional Hausdorff and spherical Hausdorff measures induced by a fixed homogenous distance $\rho$, respectively. We will denote by $Q$ the Hausdorff dimension of $\G$ with respect to $\rho$ (which always coincides with the weight of the volume form of $\G$).

\begin{remark}
    When dealing with a Lie group $\G$, one can consider the subcomplex of the de Rham complex consisting of left-invariant differential forms $\Omega^\bullet_L(\G)$. A left-invariant $k$-form is uniquely determined by its value at the identity, where it defines a linear map $\bigwedge^k\mathfrak{g}\to\R$, by identifying the tangent space at the identity with the Lie algebra $\mathfrak{g}$, and so $\Omega^\bullet_L(\G)=\bigwedge^\bullet\mathfrak{g}^\ast$.

    Moreover, in the case of a homogeneous Lie group $\G$, we can identify the tangent space $T_x\G$ to $\G$ at any point $x\in\G$ with $\mathfrak{g}\cong T_e\G$ by means of the isomorphism $dL_x$, where $L_x$ denotes the left-translation by $x\in\G$. For $\xi\in\bigwedge^k\mathfrak{g}^\ast$ and $f\in C^\infty(\G)$, we can regard $f\otimes\xi$ as a smooth $k$-form by $(f\otimes\xi)_x=f(x)(L_{x^{-1}})^*\xi$, giving rise to the isomorphism
    \begin{align*}
        \operatorname{Hom}_\R\left(\bigwedge\nolimits^k\mathfrak{g},C^\infty(\G)\right)\cong C^\infty(\G)\otimes\bigwedge\nolimits^k\mathfrak{g}^\ast\to\Gamma\left(\bigwedge\nolimits^k(T^*\G)\right)=\Omega^k(\G)\,.
    \end{align*}
\end{remark}
Without loss of generality, {one can consider a \textit{graded basis} $(X_1,\ldots,X_n)$, namely a basis of $\mathfrak g$ adapted to the direct sum decomposition \eqref{eq: grading direct sum decomposition}: } 
\begin{align*}
    \mathfrak{g}_{w_1}=\operatorname{span}_\R\{X_1,\ldots,X_{m_1}\}\ \text{ and }\ \mathfrak{g}_{w_i}=\operatorname{span}_\R\{X_{m_{i-1}+1},\ldots,X_{m_i}\}\ \text{ for }2\le i\le s\,.
\end{align*}
Furthermore, its dual basis $(\theta_1,\ldots\theta_n)$ also reflects the direct sum decomposition in terms of weights, since $w(\theta_i)=p$ if and only if $X_i\in\mathfrak{g}_{p}$.

\begin{lemma}[Forms of different weight are linearly independent]\label{lem: different weight implies lin indep} Let us consider $\alpha_1,\alpha_2\in\Omega^k(\G)$ two arbitrary smooth $k$-forms. If they are both homogeneous with $w(\alpha_1)\neq w(\alpha_2)$, then they are linearly independent.
\end{lemma}
\begin{proof}
    Let us first prove the claim for two left-invariant 1-forms $\theta_1,\theta_2\in\mathfrak{g}^\ast$. If $w(\theta_1)=p_1\neq w(\theta_2)=p_2$, by definition of weights we have that $\theta_1\in\mathfrak{g}^\ast_{p_1}$ and $\theta_2\in\mathfrak{g}^\ast_{p_2}$ and they are therefore linearly independent by the direct sum decomposition \eqref{eq: grading direct sum decomposition}. 
    
    Given $\xi_1,\xi_2\in\bigwedge^k\mathfrak{g}^\ast$ with $k>1$, if $w(\xi_1)\neq w(\xi_2)$ then without loss of generality we can assume $\xi_1=\theta_{i_1}\wedge\cdots\wedge\theta_{i_k}$ and $\xi_2=\theta_{j_1}\wedge\cdots\wedge\theta_{j_k}$ with
    \begin{align*}
        w(\xi_1)=w(\theta_{i_1})+\cdots+w(\theta_{i_k})\neq w(\xi_2)=w(\theta_{j_1})+\cdots+w(\theta_{j_k})\,.
    \end{align*}
    This means that there exists at least one $l\in\{1,\ldots,k\}$ such that $w(\theta_{i_l})=p_1\neq w(\theta_{j_l})=p_2$, i.e. $\theta_{i_l}\in\mathfrak{g}_{p_1}^\ast$ and $\theta_{j_l}\in\mathfrak{g}_{p_2}^\ast$ belong to different subspaces of the direct sum decomposition.

    Finally, the claim follows from the fact that a smooth form $\alpha\in\Omega^k(\G)$ is homogeneous of weight $p$ if $\alpha=\sum_{j}f_j\otimes\xi_j$ where $f_j\in C^\infty(\G)$ and $w(\xi_j)=p$ for each $p$.
\end{proof}

As a direct consequence, the space of smooth forms $\Omega^k(\G)$ admits a direct sum decomposition given by the weights. Throughout this paper, we will express this decomposition using the following notation
\begin{align}
    \Omega^k(\G)=\Gamma\left(\bigwedge\nolimits^k\mathfrak{g}^\ast\right)=\bigoplus_{\substack{p+q=k\\ 0\le p\le Q}}\Gamma\left(\bigwedge\nolimits^{p,q}\mathfrak{g}^\ast\right)=\bigoplus_{\substack{p+q=k\\ 0\le p\le Q}}C^\infty(\G)\otimes\bigwedge\nolimits^{p,q}\mathfrak{g}^\ast=\bigoplus_{\substack{p+q=k\\0\le p\le Q}}\Omega^{p,q}(\G)\,,
\end{align}
where $\bigwedge^{p,q}\mathfrak{g}^\ast$ and $\Omega^{p,q}(\G)$ denote respectively the space of left-invariant and smooth forms of weight $p$ and degree $p+q=k$.

 Given a smooth form $\alpha\in\Omega^\bullet(\G)$ which is \textit{not} homogeneous, that is $\alpha=\sum_{j=1}^{\kappa}\alpha_j$ with each summand having pure weight $w(\alpha_j)=p_j$, we will denote by $(\alpha)_{p_j}$ the projection of $\alpha$ onto $\Omega^{p_j,\bullet}(\G)$, that is $(\alpha)_{p_j}=\alpha_j$.

\begin{remark}
    When working with differential forms on Carnot groups, $k$-forms of weight $p$ are typically denoted by $\Omega^{k,p}(\G)$. However, since spectral sequence techniques are needed to define the spectral complexes used later in the paper, we adopt the notation that is standard in that context.
\end{remark}
\begin{lemma}\label{lem: diving exterior derivative}
    Given a {positively graded} group $\G$, the direct sum decomposition \eqref{eq: grading direct sum decomposition} of its Lie algebra $\mathfrak{g}$ induces a decomposition of the exterior differential $d$ acting on smooth forms, which can be easily expressed in terms of the weight increase. {For an arbitrary} $\alpha\in\Omega^{p,k-p}(\G)$ of pure weight $p$ and degree $k$, one can write
    \begin{align}\label{eq: decomposition of d}
        d\alpha=d_0\alpha+d_{w_1}\alpha+\cdots+d_{w_s}\alpha
    \end{align}
    where $d_0\alpha\in\Omega^{p,k+1-p}(\G)$ and $d_{w_i}\alpha\in\Omega^{p+w_i,k+1-w_i-p}(\G)$ for each $i=1,\ldots,s$.
\end{lemma}
\begin{proof}
    Given an arbitrary $k$-form of weight $p$, we have $\alpha=\sum_jf_j\otimes\xi_j$ with each $f_j\in C^\infty(\G)$ and $\xi_j\in\bigwedge^{p,k-p}\mathfrak{g}^\ast$. The exterior differential applied to $\alpha$ has the following expression
    \begin{align*}
        \begin{split}
            d\bigg(\sum_jf_j\otimes\xi_j\bigg)\left(V_1,\ldots,V_{k+1}\right)=&\sum_j\sum_{i=1}^{k+1}(-1)^{i-1}V_if_j\otimes\xi_j(V_1,\ldots,\hat{V}_i,\ldots,V_{k+1})+\\&+\sum_j\sum_{1\le i<l\le k+1}(-1)^{i+l}f_j\otimes\xi_j\left([V_i,V_l],V_1,\ldots,\hat{V}_i,\ldots,\hat{V}_l,\ldots ,V_{k+1}\right)
        \end{split}
    \end{align*}   
    for any $V_1,\ldots,V_{k+1}\in\Gamma(T\G)$.

    Using the more streamlined notation $\alpha=\sum_jf_j\xi_j$ to express the $k$-form, we get the formula
    \begin{align*}
d\bigg(\sum_jf_j\xi_j\bigg)=\sum_j\left(df_j\wedge\xi_j+f_jd\xi_j\right)=\sum_jdf_j\wedge\xi_j+\sum_jf_jd\xi_j\,.
    \end{align*}
    If we consider {a graded basis $(X_1,\ldots,X_n)$} with dual basis $(\theta_1,\ldots,\theta_n)$, we obtain a very explicit expression for the first summand, that is
    \begin{align*}
        \sum_jdf_j\wedge\xi_j=\sum_j\sum_{l=1}^nX_lf_j\theta_l\wedge\xi_j=\sum_{i=1}^s\sum_{X_l\in\mathfrak{g}_{w_i}}\sum_j X_lf_j\theta_l\wedge\xi_j=\sum_{i=1}^sd_{w_i}\alpha\,.
    \end{align*}
    For each $i=1,\ldots,s$, we see that
    \begin{align*}
        d_{w_i}\alpha=\sum_{X_l\in\mathfrak{g}_{w_i}}\sum_jX_lf_j\theta_l\wedge\xi_j=\big(d\alpha\big)_{p+w_i}\in\Omega^{p+w_i,k+1-p-w_i}(\G)
    \end{align*}
    since each $X_l\in\mathfrak{g}_{w_i}$ and so $w(\theta_l)=w_i$. 

    Regarding the second summand, one needs to show that unless $d\xi_j$ vanishes, we have $w(d\xi_j)=w(\xi_j)=p$, i.e. it keeps the weight constant. In the case of homogeneous Lie groups, one can show this by using the group's dilations. Indeed, this second summand coincides with the action of the exterior derivative on left-invariant forms, which can be seen as
    \begin{align*}
        d\xi_j(X_1,\ldots,X_{k+1})=\sum_{1\le i<l\le k+1}(-1)^{i+l}\xi_j\left([X_i,X_l],X_1,\ldots,\hat{X}_i,\ldots,\hat{X}_l,\ldots,X_{k+1}\right)\text{ for all }X_i\in\mathfrak{g}\,.
    \end{align*}
    This formula, combined with the fact that dilations $\delta_\lambda$ are automorphisms of the Lie algebra $\mathfrak{g}$, readily implies that $d$ applied to $\bigwedge^k\mathfrak{g}^\ast$ commutes with the dilations. More explicitly, given a left-invariant 1-form $\theta_j\in\mathfrak{g}_{w_i}$, we have that for any $X_{i_1},X_{i_2}\in\mathfrak{g}$
    \begin{align*}
        \delta_\lambda\left(d\theta_j\right)(X_{i_1},X_{i_2})=&d\theta_j(\delta_\lambda X_{i_1},\delta_\lambda X_{i_2})=-\theta_j\left([\delta_\lambda X_{i_1},\delta_\lambda X_{i_2}]\right)=-\theta_j\left(\delta_\lambda[X_{i_1},X_{i_2}]\right)\\=&-\delta_\lambda\theta_j([X_{i_1},X_{i_2}])=d(\delta_\lambda\theta_j)(X_{i_1},X_{i_2})\,.
    \end{align*}
    Proceeding by induction, if we assume that $d(\delta_\lambda\xi_j)=\delta_\lambda (d\xi_j)$ for a given left-invariant $k$-form $\xi_j\in\bigwedge^k\mathfrak{g}^\ast$, by the Leibniz rule we have
    \begin{align*}
        \delta_\lambda(d(\theta_j\wedge\xi_j))=&\delta_\lambda(d\theta_j\wedge\xi_j-\theta_j\wedge d\xi_j)=\delta_\lambda(d\theta_j)\wedge\delta_\lambda\xi_j-\delta_\lambda\theta_j\wedge\delta_\lambda (d\xi_j)=d(\delta_\lambda\theta_j)\wedge\delta_\lambda\xi_j-\delta_\lambda\theta_j\wedge d(\delta_\lambda\xi_j)\\=&d\left(\delta_\lambda\theta_j\wedge\delta_\lambda\xi_j\right)=d(\delta_\lambda(\theta_j\wedge\xi_j))\,.
    \end{align*}
    This commutativity then implies that $w(d\xi_j)=w(\xi_j)=p$ and so
    $d_0\alpha=\sum_jf_jd\xi_j\in\Omega^{p,k+1-p}(\G)$.
    \end{proof}    
Before being able to construct spectral complexes in this setting, we first need to show that the cochain complex that we are working on, namely the de Rham complex, is indeed a truncated multicomplex.
\begin{definition}[Definition 2.1 in \cite{livernet2020spectral}] \label{def: truncated multicomplex}Let $\mathcal{K}$ be a commutative unital ground ring. An $s$-multicomplex (also known as twisted chain complex) is a $(\Z,\Z)$-graded $\mathcal{K}$ module $C$ equipped with maps $d_i\colon C\to C$ for $i\ge 0$ of bidegree $|d_i|=(i,1-i)$ such that
\begin{align}\label{eq: multicomplex maps rules}
    \sum_{i+j=n}d_id_j=0\ \text{ for all }n\ge 0\text{ and }d_k=0\text{ for all }k>s\,.
\end{align}
\end{definition}
We should also mention that in Definition \ref{def: truncated multicomplex} we are choosing a cohomological sign and degree convention for the differentials $d_i$.

\begin{proposition}\label{prop: de Rham is multicomplex} The de Rham complex $(\Omega^\bullet(\G),d)$ on a positively graded group with direct sum decomposition $\mathfrak{g}=\mathfrak{g}_{w_1}\oplus\cdots\oplus\mathfrak{g}_{w_s}$ is a truncated $w_s$-multicomplex.
\end{proposition}
\begin{proof}
    The direct sum decomposition according to weights given in \eqref{eq: grading direct sum decomposition} endows the space of smooth forms $\Omega^\bullet(\G)$ with a bigrading. In Lemma \ref{lem: different weight implies lin indep}, we established the existence of differential maps $d_{j}\colon\Omega^\bullet(\G)\to\Omega^\bullet(\G)$ of bidegree $|d_{j}|=(j,1-j)$. Note that, except for $d_0$ which is a $C^\infty(\G)$-linear map, the differential maps $d_{w_i}$ with $i=1,\ldots,s$ are $\R$-linear. For this reason, even though each $\Omega^{p,q}(\G)$ is a $C^\infty(\G)$-module of bidegree $(p,q)\in\N\times\Z$, in this setting $\mathcal{K}=\R$. This fact will play a role when constructing the subspaces of forms $Z_r^{\bullet,\bullet}$ and $B_r^{\bullet,\bullet}$ introduced in Definition \ref{Z and B defined}, which will turn out to be $\R$-submodules of $\Omega^{\bullet,\bullet}(\G)$. 
    
    We are left to prove that the formulae in \eqref{eq: multicomplex maps rules} also hold.
    These follow directly from the fact that $(\Omega^\bullet(\G),d)$ is a complex, i.e. $d^2=0$, and that forms of different weight are linearly independent (Lemma \ref{lem: different weight implies lin indep}).  
Let us first express the exterior derivative as $d=d_0+d_1+\cdots+d_{w_s}$, where for each $j\in\Z^+$ $d_j$ is the zero map unless $j= w_{i}$ for some $i=1,\ldots,s$.
    Indeed, given a $k$-form $\alpha\in\Omega^{p,k-p}(\G)$, if we expand the expression for $d^2\alpha$ according to the weight, we get
    \begin{align*}
        d^2\alpha=&d\left(d_0\alpha+d_{1}\alpha+\cdots+d_{w_s}\alpha\right)=(d_0+d_{1}+\cdots+d_{w_s})\left(d_0\alpha+d_{1}\alpha+\cdots+d_{w_s}\alpha\right)\\=&d_0^2\alpha+\left(d_0d_{1}+d_{1}d_0\right)\alpha+\left(d_0d_2+d_1^2+d_2d_0\right)\alpha+\cdots+d_{w_s}^2\alpha=\sum_{n=0}^{2w_s}\sum_{i+j=n}d_jd_i\alpha=0\,.
    \end{align*}
    Each summand $\sum_{i+j=n}d_jd_i\alpha$ has weight $p+n$, that is
    \begin{align*}
        \sum_{i+j=n}d_jd_i\alpha\in\Omega^{p+n,k+1-p-n}(\G)\ \text{ for each }n=0,\ldots,2w_s\,.
    \end{align*}
    By Lemma \ref{lem: different weight implies lin indep}, we get that each summand must be zero.
\end{proof}

Finally, in the context of the de Rham complex on positively graded groups, weights of differential forms are always non-negative. In other words, the space of smooth forms $\Omega^\bullet(\G)$ is an $(\Z^+,\Z)$-graded $C^\infty(\G)$-module. Not only this, but since the group $\G$ is finite-dimensional, the range for the weights is finite and takes integer values between 0 and the weight $Q$ of the volume form.
\subsection{Introducing a scalar product on $\Omega^\bullet(\G)$}\label{subsection: introducing a scalar product}
The Rumin complex $(E_0^\bullet,d_c)$ (as presented in Proposition \ref{prop: Rumin computing de RHam}) is defined as subspaces of the space of smooth forms $E_0^\bullet\subset\Omega^\bullet(\G)$. These $E_0^\bullet$ are isomorphic to the cohomology of the complex $(\Omega^\bullet(\G),d_0)$, however since we want to consider them as subspaces of smooth forms and not as quotients, we require a way of identifying complements of the subspace $\operatorname{Im}d_0$. This can be easily obtained by introducing a scalar product on $\mathfrak{g}$ adapted to the homogeneous structure \eqref{eq: grading direct sum decomposition}, which canonically extends to the space of left-invariant forms $\Omega^\bullet_L(\G)=\bigwedge^\bullet\mathfrak{g}^\ast$. Since the scalar product $\langle\cdot,\cdot\rangle_k$ defined on each $\bigwedge^k\mathfrak{g}^\ast$ is adapted to the positive grading, we have that forms of different weight are orthogonal, i.e. for each degree $k$ 
\begin{align}\label{eq: left diff weight implies orthogonal}
    \text{ given }\theta\in\bigwedge\nolimits^{p_1,k-p_1}\mathfrak{g}^\ast\ \text{ and }\ \xi\in\bigwedge\nolimits^{p_2,k-p_2}\mathfrak{g}^\ast\,,\ \text{ if }p_1\neq p_2\text{ then }\langle\theta,\xi\rangle_k=0\,.
\end{align}
It is then possible to define the Hodge-$\star$ operator as the linear map acting on $\bigwedge^\bullet\mathfrak{g}^\ast$ as
\begin{align*}
    \star\colon\bigwedge\nolimits^k\mathfrak{g}^\ast\xrightarrow[]{\cong}\bigwedge\nolimits^{n-k}\mathfrak{g}^\ast\ ,\ \theta\wedge\star\xi:=\langle\theta,\xi\rangle_k\operatorname{vol}.
\end{align*}
The scalar product can be further extended to the space of smooth forms $\Omega^\bullet(\G)$ viewed as a $C^\infty(\G)$-module generated by $\bigwedge^\bullet\mathfrak{g}^\ast$.   In particular, when considering the subspace $\Omega^\bullet_c(\G)$ of compactly supported smooth forms, one can introduce a scalar product, the so-called $L^2$-inner product on forms, defined as
\begin{align*}
    \langle\alpha,\beta\rangle_{L^2(\Omega^\bullet(\G))}:=\int_\G\alpha\wedge\star\beta\,.
\end{align*}
Here the Hodge-$\star$ operator is naturally extended to a $C^\infty(\G)$-linear map $\star\colon\Omega^k(\G)\longrightarrow\Omega^{n-k}(\G)$ by viewing the space of smooth forms simply as a $C^\infty(\G)$-module generated by $\bigwedge^\bullet\mathfrak{g}^\ast$.
It then follows directly from the definition of the $L^2$-inner product on forms and \eqref{eq: left diff weight implies orthogonal} that also in the case of smooth differential forms, forms of different weight are orthogonal.

To reduce the notation used throughout the paper, we will keep $\langle\cdot,\cdot\rangle_k$ to also denote the $L^2$-inner product between two smooth differential forms of degree $k$. 
Finally, given a subspace $S$ of $\Omega^\bullet(\G)$, we will denote the orthogonal projection onto $S$ by $\operatorname{pr}_S$.
\begin{definition}[The adjoint of $d_0$]\label{def: adjoint of d_0} We can use the scalar product just introduced to define the formal transpose (adjoint) of $d_0$, which we will denote by $\delta_0$. In other words, the map $\delta_0\colon\Omega^{p,k-p}(\G)\to\Omega^{p,k-1-p}(\G)$ is defined by imposing
\begin{align*}
    \langle d_0\alpha,\beta\rangle_{k+1}=\langle\alpha,\delta_0\beta\rangle_k\ \text{ for any }\alpha\in\Omega^k(\G)\text{ and }\beta\in\Omega^{k+1}(\G)\,.
\end{align*}
The fact that $\delta_0$ keeps the weight of forms constant, i.e. $w(\delta_0\alpha)=w(\alpha)$, is a direct consequence of the fact that elements of different weight are orthogonal and that $d_0$ keeps the weight of forms constant.
\end{definition}
Another crucial operator that can be defined from $d_0$ using the scalar product is its partial inverse. This operator will be central to the construction of the so-called Rumin differential $d_c$.
\begin{definition}[The partial inverse $d_0^{-1}$]\label{def: d_0^{-1}} The map $d_0$ being $C^\infty(\G)$-linear implies that it acts as a bijection from $\left(\ker d_0\right)^\perp=\operatorname{Im}\delta_0$ onto $\operatorname{Im}d_0$. Following Rumin's construction, it is customary to use the shorthand notation $d_0^{-1}$ to denote the linear map given by $$d_0^{-1}:=d_0^{-1}\operatorname{pr}_{\operatorname{Im}d_0}\colon\Omega^{p,k-p}(\G)\longrightarrow\Omega^{p,k-1-p}(\G)\,.$$
In particular, it readily follows that $(d_0^{-1})^2=0$ and $\ker d_0^{-1}=\ker\delta_0=\left(\operatorname{Im}d_0\right)^\perp$.
\end{definition}
Before defining the space of Rumin forms, let us stress that the map $d_0$ is an extension of the Chevalley-Eilenberg differential to the space of smooth forms, so in particular it is a linear map between finite dimensional spaces (we refer to \cite{F+T1} for a more thorough explanation of this fact). In particular, this implies that the map $d_0\colon\Omega^{p,k-p}(\G)\to\Omega^{p,k-1-p}(\G)$ has closed range, a critical property that will be used when working with the algebraic Laplacian associated with $d_0$.
\begin{definition}[The space of Rumin forms] As shown in Proposition \ref{prop: de Rham is multicomplex}, the de Rham complex $\Omega^\bullet(\G)$ on a positively graded group $\G$ with differential $d=d_0+d_{w_1}+\cdots+d_{w_s}$ is a truncated multicomplex. In particular, equation \eqref{eq: multicomplex maps rules} for $n=0$ gives that $d_0^2=0$, that is $(\Omega^\bullet(\G),d_0)$ is also a complex. Therefore, it is possible to compute its cohomology. As already mentioned, we are interested in defining subspaces of forms and not quotients. Using the $L^2$-scalar product introduced in Subsection \ref{subsection: introducing a scalar product}, one can define the space of Rumin forms $E_0^\bullet$ associated with the truncated multicomplex $(\Omega^\bullet(\G),d=d_0+d_{w_1}+\cdots+d_{w_s})$ as
\begin{align}\label{eq: Rumin forms defined}
    E_0^k=\ker d_0\cap\left(\operatorname{Im}d_0\right)^\perp\cap\Omega^k(\G)=\ker d_0\cap\ker\delta_0\cap\Omega^k(\G)\,,
\end{align}
the last equality following from the closed range theorem.
\end{definition}
\begin{definition}[Orthogonal projection onto $E_0^\bullet$]\label{def: Pi_0} Let us consider the operator 
\begin{align*}
    \Pi_0:=\operatorname{Id}-d_0d_0^{-1}-d_0^{-1}d_0\colon\Omega^{p,k-p}(\G)\to\Omega^{p,k-p}(\G)\,.
\end{align*}
From the definition of the partial inverse map $d_0^{-1}$, we have that
\begin{itemize}
    \item $\Pi_0$ maps elements of $\Omega^{p,k-p}(\G)$ onto elements of $\Omega^{p,k-p}(\G)$, i.e. it preserves both the weight and the degree of forms;
    \item $d_0^{-1}d_0=\operatorname{pr}_{\operatorname{Im}\delta_0}$ and $d_0d_0^{-1}=\operatorname{pr}_{\operatorname{Im}d_0}$;
    \item $\Pi_0^2=(\operatorname{Id}-d_0d_0^{-1}-d_0^{-1}d_0)(\operatorname{Id}-d_0d_0^{-1}-d_0^{-1}d_0)=\operatorname{Id}-d_0d_0^{-1}-d_0^{-1}d_0=\Pi_0$.
\end{itemize}
Therefore, the map $\Pi_0$ is an orthogonal projection onto the subspace
\begin{align*}
    \operatorname{Im}\Pi_0=\ker d_0\cap\ker d_0^{-1}=\ker d_0\cap\ker\delta_0=E_0\,,
\end{align*}
that is $\Pi_0=\operatorname{pr}_{E_0}$.
\end{definition}
A useful characterisation of the space of Rumin forms is as the ``harmonic forms'' of the algebraic Laplacian defined using the operator $d_0$ and its adjoint $\delta_0$. Moreover, such a Laplacian admits a Hodge-decomposition thanks to the fact that $d_0$ has closed range.
\begin{definition}[The algebraic Laplacian $\Box_0$]\label{def: box 0} Let us consider the Laplacian operator defined using the algebraic differential $d_0$ and its adjoint $\delta_0$, that is
\begin{align*}
    \Box_0:=d_0\delta_0+\delta_0d_0\colon\Omega^{p,k-p}(\G)\longrightarrow\Omega^{p,k-p}(\G)\,.
\end{align*}
Just like the projection $\Pi_0$, $\Box_0$ preserves both the weight and the degree of forms. It is a symmetric operator since $\langle \Box_0\alpha_1,\alpha_2\rangle_k=\langle\alpha_1,\Box_0\alpha_2\rangle_k$ for any $\alpha_1,\alpha_2\in\Omega^k(\G)$, and its kernel coincides with the space of Rumin forms
\begin{align*}
    \ker\Box_0=\ker d_0\cap\ker\delta_0=E_0\,.
\end{align*}
The inclusion $\ker d_0\cap\ker \delta_0\subset\ker\Box_0$ is direct, while the converse one follows from the fact that if $\alpha\in\ker\Box_0\cap\Omega^k(\G)$ then
\begin{align*}
    0=\langle\Box_0\alpha,\alpha\rangle_k=\langle d_0\delta_0\alpha,\alpha\rangle_k+\langle\delta_0d_0\alpha,\alpha\rangle_k=|d_0\alpha|^2_k+|\delta_0\alpha|^2_k\ \Longrightarrow\ d_0\alpha=\delta_0\alpha=0\,.
\end{align*}
\end{definition}
\begin{proposition}[Hodge decomposition for $\Box_0$]\label{prop: Hodge decomposition for Box0}
    The space of smooth forms $\Omega^\bullet(\G)$ admits an orthogonal direct sum (or a Hodge) decomposition in terms of $\Box_0$, $d_0$ and its adjoint $\delta_0$. More explicitly,
    \begin{align}\label{eq: hodge decompo Box_0}
        \Omega^\bullet(\G)=\operatorname{Im}d_0\oplus\ker\Box_0\oplus\operatorname{Im}\delta_0\,.
    \end{align}
\end{proposition}
\begin{proof}
    The statement follows from the fact that the orthogonal complement of the kernel of a map with closed range coincides with the range of its adjoint: for each degree $k$ we have
    \begin{align*}
        \Omega^k(\G)=&\ker d_0\oplus\left(\ker d_0\right)^\perp=\ker d_0\oplus\operatorname{Im}\delta_0\\=&\ker d_0\cap\ker\delta_0\oplus\ker d_0\cap\left(\ker\delta_0\right)^\perp\oplus\operatorname{Im}\delta_0\\=&\ker\Box_0\oplus\ker d_0\cap\operatorname{Im}d_0\oplus\operatorname{Im}\delta_0=\ker\Box_0\oplus\operatorname{Im}d_0\oplus\operatorname{Im}\delta_0\,.
    \end{align*}
\end{proof}
\section{The Rumin complex on $R$-manifolds}
The aim of this section is to introduce a special class of submanifoldsc on which Stokes' theorem is valid for the Rumin complex $(E_0^\bullet,d_c)$. To this end, we briefly recall the construction of the cochain maps $d_c\colon E_0^k\to E_0^{k+1}$, highlighting the properties that will be needed to state the main result of the section, namely Theorem \ref{thm: stokes on R manifolds}.

The presentation of the Rumin complex adopted here follows the approach developed in \cite{F+T1}. Equivalent formulations can be found in the original works of Rumin \cite{rumin1999differential,rumin_grenoble,rumin_palermo}, as well as in the later expositions \cite{CompensatedCompactness,FranchiTesi,FT}.
\begin{lemma}\label{lem: stricly increasing weights implies nilpotent} If $B\colon\Omega^{\bullet}(\G)\to\Omega^\bullet(\G)$ is an operator that strictly increases the weight of forms so that
\begin{align*}
    \text{given }\alpha\in\Omega^{p,\bullet}(\G),\ \text{ we have that }B\alpha\in\Omega^{p',\bullet}(\G)\text{ with }p'>p\,,
\end{align*}
then $B$ is nilpotent, i.e. there exists $N\in\N$ independent of $B$ such that $B^N\equiv 0$.
\end{lemma}
\begin{proof}
    The nilpotency of any operator that strictly increases the weight of forms follows from the assumption that the range of possible weights is bounded, from 0 up to the weight of the volume form $Q=w(\operatorname{vol})$. In particular, the statement holds for any such operator $B$, regardless of how $B$ acts on the degree of forms.
\end{proof}
\begin{lemma}\label{lem: inverting nilpotent operators}
    Let $A_0\colon\Omega^\bullet(\G)\to\Omega^\bullet(\G)$ be an invertible linear map that preserves the weights, while $B\colon\Omega^\bullet(\G)\to\Omega^\bullet(\G)$ is an operator that strictly increases the weights. Then the operator $A:=A_0-B\colon\Omega^\bullet(\G)\to\Omega^\bullet(\G)$ is invertible and its inverse $A^{-1}\colon\Omega^\bullet(\G)\to\Omega^\bullet(\G)$ has the form
    \begin{align*}
        A^{-1}=A_0^{-1}\sum_{j=0}^{N-1}\left(BA_0^{-1}\right)^j=\bigg(\sum_{j=0}^{N-1}\left(A_0^{-1}B\right)^j\bigg)A_0^{-1}\,,
    \end{align*}
    with $N\in\N$ the nilpotency exponent of $B$, as discussed in Lemma \ref{lem: stricly increasing weights implies nilpotent}.
\end{lemma}
\begin{proof}
    By construction, we have $A=A_0-B=\left(\operatorname{Id}-BA_0^{-1}\right)A_0=A_0\left(\operatorname{Id}-A_0^{-1}B\right)$. 

    The operators $BA_0^{-1}$ and $A_0^{-1}B$ act on $\Omega^\bullet(\G)$ by strictly increasing weights, hence they are both nilpotent by Lemma \ref{lem: stricly increasing weights implies nilpotent}. Consequently, we get invertibility via Neumann series: both $\operatorname{Id}-BA_0^{-1}$ and $\operatorname{Id}-A_0^{-1}B$ are invertible with 
    \begin{align*}
        \left(\operatorname{Id}-BA_0^{-1}\right)^{-1}=\sum_{j=0}^{N-1}\left(BA_0^{-1}\right)^j\ \text{ and }\ \left(\operatorname{Id}-A_0^{-1}B\right)^{-1}=\sum_{j=0}^{N-1}\left(A_0^{-1}B\right)^j\,.
    \end{align*}
    The claim follows directly from the fact that $A^{-1}=A_0^{-1}\left(\operatorname{Id}-BA_0^{-1}\right)^{-1}=\left(\operatorname{Id}-A_0^{-1}B\right)^{-1}A_0^{-1}$. Again, just like in Lemma \ref{lem: stricly increasing weights implies nilpotent}, the statement holds for any such operators $A_0$ and $B$, regardless of how they act on the degree of forms.
\end{proof}
Using these two lemmata, we are now able to construct the necessary projection maps to define the Rumin differential $d_c$.
\begin{lemma}\label{lem: Pi_E}
    Let us consider the operator $b:=d_0^{-1}d_0-d_0^{-1}d=-d_0^{-1}(d-d_0)$ acting on the space of smooth forms $\Omega^\bullet(\G)$.
    \begin{enumerate}
        \item The map $b\colon\Omega^{\bullet}(\G)\longrightarrow\Omega^\bullet(\G)$ preserves the degree, it is nilpotent and hence $\operatorname{Id}-b$ is invertible.
        \item The operator
        \begin{align}\label{eq: Pi projection}
            \Pi:=\left(\operatorname{Id}-b\right)^{-1}d_0^{-1}d+d\left(\operatorname{Id}-b\right)^{-1}d_0^{-1}
        \end{align}
        is the projection of $\Omega^\bullet(\G)$ onto
        \begin{align*}
            F:=\operatorname{Im}d_0^{-1}+\operatorname{Im}dd_0^{-1}=\operatorname{Im}\delta_0+\operatorname{Im}d\delta_0
        \end{align*}
        along
        \begin{align*}
            E:=\ker d_0^{-1}\cap\ker d_0^{-1}d=\ker\delta_0\cap\ker\delta_0d\,.
        \end{align*}
    \end{enumerate}
\end{lemma}
\begin{proof}$\phantom{=}$
    \begin{enumerate}
        \item By definition, the operator $b$ strictly increases the weight and so by Lemma \ref{lem: stricly increasing weights implies nilpotent} it is nilpotent, i.e. there exists $N\in\N$ such that $b^N=0$. We can now apply Lemma \ref{lem: inverting nilpotent operators} taking $A_0=\operatorname{Id}$ and $B=b$, so the operator $\operatorname{Id}-b$ is invertible and its inverse is given by
        \begin{align}\label{eq: inverse Id-b}
            \left(\operatorname{Id}-b\right)^{-1}=\sum_{j=0}^{N-1}b^j=\sum_{j=0}^{N-1}\left[-d_0^{-1}(d-d_0)\right]^j\,.
        \end{align}
        This also implies that $\Pi$ as defined in \eqref{eq: Pi projection} is a well-defined operator on $\Omega^\bullet(\G)$.
        \item  By \eqref{eq: Pi projection} and \eqref{eq: inverse Id-b}, we have that the image of $\Pi$ is included in the subspace $F\subset\Omega^\bullet(\G)$ defined in the statement. Moreover, using the properties of $d$, $d_0$ and $d_0^{-1}$ we get that
        \begin{align*}
            \Pi d_0^{-1}=d_0^{-1}\ \text{ and }\ \Pi dd_0^{-1}=dd_0^{-1}
        \end{align*}
        so that $\Pi=\operatorname{Id}$ on $F$ and $\Pi^2=\Pi$, i.e. $\Pi$ is a projection onto $\operatorname{Im}\Pi=F$ along $\ker\Pi$. Finally, $\ker\Pi$ contains the subspace $E$ introduced in the statement. Since we also have the equalities
        \begin{align*}
            d_0^{-1}\Pi=d_0^{-1}\ \text{ and }\ d_0^{-1}d\Pi=d_0^{-1}d\,,
        \end{align*}
        we obtain that $\ker\Pi=E$.
    \end{enumerate}
\end{proof}
Following Rumin's notation, we denote the two projections onto $F$ and $E$ respectively as
\begin{align*}
    \Pi_F:=\Pi\ \text{ and }\ \Pi_E:=\operatorname{Id}-\Pi\,,
\end{align*}
where $\Pi$ is the projection defined in \eqref{eq: Pi projection}.
\begin{lemma}\label{lem: properties of proejctions}
    The projector operators $\Pi_E$, $\Pi_F$, and $\Pi_0$ enjoy the following properties:
    \begin{enumerate}
        \item $d_0^{-1}\Pi_E=\Pi_Ed_0^{-1}=0$;
        \item $d\Pi_F=\Pi_Fd$ and $d\Pi_E=\Pi_Ed$;
        \item $\Pi_0\Pi_E\Pi_0=\Pi_0$ and $\Pi_E\Pi_0\Pi_E=\Pi_E$.
    \end{enumerate}
\end{lemma}
\begin{proof}
    The equalities in (1) follow from the fact that $d_0^{-1}\Pi=\Pi d_0^{-1}$, while those in (2) can easily be checked by carrying out the explicit computations that show $d\Pi=d(\operatorname{Id}-b)^{-1}d_0^{-1}d=\Pi d$.

    Finally, to get the equalities in (3), we need to apply the fact that $\operatorname{Im}\Pi\subset\operatorname{Im}d_0^{-1}\subset\left(\ker\Box_0\right)^\perp$ and the equalities in (1), so that
    \begin{align*}
        \Pi_E\Pi_0\Pi_E=\Pi_E^2-\Pi_E\left(d_0d_0^{-1}+d_0^{-1}d_0\right)\Pi_E=\Pi_E\ \text{ and }\ \Pi_0\Pi_E\Pi_0=\Pi_0^2-\Pi_0\Pi\,\Pi_0=\Pi_0\,.
    \end{align*}
\end{proof}
As a consequence of Lemma \ref{lem: properties of proejctions}, we get the following result.
\begin{proposition}[Theorem 2.6 in \cite{rumin_grenoble}]\label{prop: Rumin computing de RHam} The de Rham complex $(\Omega^\bullet(\G),d=d_0+d_{w_1}+\cdots+d_{w_s})$ splits into two subcomplexes $(E^\bullet,d)$ and $(F^\bullet,d)$. Moreover, the following operator
\begin{align}\label{eq: formula d_c}
    d_c:=\Pi_0d\Pi_E\colon E_0^\bullet\subset\Omega^k(\G)\longrightarrow E_0^\bullet\subset\Omega^{k+1}(\G)
\end{align}
    satisfies $d_c^2=0$, the complex $(E_0^\bullet,d_c)$ computes the same cohomology as the initial de Rham complex $(\Omega^\bullet(\G),d)$ and it is known as the Rumin complex.
\end{proposition}
\begin{proof}
    The fact that $(E_0^\bullet,d_c)$ is a complex follows from applying properties (2) and (3) of Lemma \ref{lem: properties of proejctions} to the explicit formula \eqref{eq: formula d_c} together with the fact that $d^2=0$. By property (2) in Lemma \ref{lem: properties of proejctions}, we get that $\Pi_E$ is a homotopical equivalence between the de Rham complex $(\Omega^\bullet(\G),d)$ and the subcomplex $(E^\bullet,d)$. Finally, property (3) in Lemma \ref{lem: properties of proejctions} further implies that $E$ and $E_0$ are in bijection, and that $\Pi_E$ restricted to $E_0$ and $\Pi_0$ restricted to $E$ are inverse maps of each other. Hence, the complex $(E^\bullet,d)$ is conjugated to $(E_0^\bullet,d_c)$ with $d_c=\Pi_0d\Pi_E$, as defined in the statement.
\end{proof}
The relations which occur between the three complexes $(\Omega^\bullet(\G),d)$, $(E^\bullet,d)$, and $(E_0^\bullet,d_c)$ can be gathered in a more visual way through the following diagram:
\begin{center}

\begin{tikzpicture}[node distance=2.5cm, auto]

\pgfmathsetmacro{\shift}{0.3ex}

\node (A) {$\Omega^\bullet(\G)$};
\node(B)[right of=A] {$\Omega^\bullet(\G)$};
\node (C) [below of=A] {$E^\bullet$};
\node (D) [right of=C] {$E^\bullet$};
\node (E) [below of=C] {$  E_0^\bullet$};
\node (F) [below of=D] {$  E_0^\bullet$};

\draw[->](A) to node {\small $ {d}$}(B);
\draw[transform canvas={xshift=-0.5ex},->] (A) --(C) node[left,midway] {\footnotesize $\Pi_E$};
\draw[transform canvas={xshift=.5ex},->](C) -- (A) node[right,midway] {\footnotesize $\iota$};
\draw[transform canvas={xshift=-0.5ex},->] (B) --(D) node[left,midway] {\footnotesize $\Pi_E$};
\draw[transform canvas={xshift=.5ex},->](D) -- (B) node[right,midway] {\footnotesize $\iota$};
\draw[->](C) to node {\small $ {d}$}(D);
\draw[transform canvas={xshift=-0.5ex},->] (C) --(E) node[left,midway] {\footnotesize $\Pi_{  0}$};
\draw[transform canvas={xshift=.5ex},->](E) -- (C) node[right,midway] {\footnotesize $\Pi_E$};
\draw[transform canvas={xshift=-0.5ex},->] (D) --(F) node[left,midway] {\footnotesize $\Pi_{  0}$};
\draw[transform canvas={xshift=.5ex},->](F) -- (D) node[right,midway] {\footnotesize $\Pi_E$};
\draw[->](E) to node {\small $ {d}_c$}(F);
\end{tikzpicture}
\end{center}
In order to study the properties of the Rumin differential needed to justify the introduction of special manifolds on which Stokes' theorem holds for $(E_0^\bullet,d_c)$, we now describe more explicitly the action of the differential $d_c$. Throughout the computations, we will also adopt the shorthand notation introduced in \cite{tripaldi2026spectralcomplexestruncatedmulticomplexes} in order to simplify the expressions involved.
\begin{definition}\label{def: partial operators}
    Given the maps $d_0,\,d_{1},\ldots,d_{w_s}$ that make up the exterior derivative $d$ as expressed in \eqref{eq: decomposition of d} and satisfy the multicomplex relations \eqref{eq: multicomplex maps rules}, together with the partial inverse $d_0^{-1}$ introduced in Definition \ref{def: d_0^{-1}}, we are going to define by induction the following maps on $\Omega^\bullet(\G)$:
    \begin{align}\label{eq: partial maps}
        \partial_1=d_1\ \text{ and }\ \partial_r=d_r-\sum_{j=1}^{r-1}d_{r-j}d_0^{-1}\partial_j\ \text{ for }r\ge 2\,.
    \end{align}
\end{definition}
For clarity, let us see the explicit expression of the operator $\partial_4$:
\begin{align*}
    \partial_4=&d_4-\sum_{j=1}^3d_{3-j}d_0^{-1}\partial_j=d_4-d_3d_0^{-1}\partial_1-d_2d_0^{-1}\partial_2-d_1d_0^{-1}\partial_3\\=&d_4-d_3d_0^{-1}d_1-d_2d_0^{-1}\left(d_2-d_1d_0^{-1}\partial_1\right)-d_1d_0^{-1}\left(d_3-d_2d_0^{-1}\partial_1-d_1d_0^{-1}\partial_2\right)\\=&d_4-d_3d_0^{-1}d_1-d_2d_0^{-1}\left(d_2-d_1d_0^{-1}d_1\right)-d_1d_0^{-1}\left(d_3-d_2d_0^{-1}d_1-d_1d_0^{-1}d_2+d_1d_0^{-1}d_1d_0^{-1}d_1\right)
\end{align*}
Note that by definition, each differential operator $\partial_r$ increases the weight of a form by $r$ and its degree by 1. This can be more compactly expressed by stating that $\partial_r$ has bidegree $|\partial_r|=(r,1-r)$, that is:
\begin{align*}
    \text{ for any }\alpha\in\Omega^{p,k-p}(\G)\ \text{ we have }\ \partial_r\alpha\in\Omega^{p+r,k+1-p-r}(\G)\,.
\end{align*}
\begin{lemma}[Expressing $d_c$ on $E_0$]\label{lem: expressing d_c on E_0} Using the operators introduced in Definition \ref{def: partial operators}, one can express the operator $d_c$ acting on $E_0^\bullet$ as follows:
\begin{align}\label{eq: formula for d_c in terms of partials}
    d_c\alpha=\Pi_0\sum_{i=1}^{N-1}\partial_r\alpha=\sum_{r=1}^{N-1}\partial_r\alpha-d_0d_0^{-1}\sum_{r=1}^{N-1}\partial_r\alpha-d_0^{-1}d_0\sum_{r=1}^{N-1}\partial_r\alpha\ \text{ for any }\alpha\in E_0^\bullet\,.
\end{align}
Here is $N\in\mathbb{N}$ taken so that $\partial_N\alpha=0$ for every $\alpha\in E_0^\bullet$.
\end{lemma}
\begin{proof}
The existence of some $N\in\N$ for which $\partial_N\alpha=0$ for every $\alpha\in E_0^\bullet$ is guaranteed by the fact that the range of possible weights is bounded. Note that, in general, there will be no relationship between $w_s$ and $N$. Depending on the group $\G$ (and hence the multicomplex structure of the de Rham complex), but also on the weight and degree of the form $\alpha\in\Omega^\bullet(\G)$, one could have $w_s$ bigger, smaller, or equal to $N$. The only thing that is certain for any given positively gradable group $\G$ is that both $w_s$ and $N$ are finite positive integers.

Let us now focus on the action of $\Pi$ on elements of $E_0$ as defined in \eqref{eq: Pi projection}. If $\alpha\in E_0^\bullet$, then $\alpha\in\ker d_0\cap\ker\delta_0=\ker d_0\cap\ker d_0^{-1}$, and so the action of $\Pi$ simplifies to
\begin{align*}
    \Pi\alpha=&\left(\operatorname{Id}-b\right)^{-1}d_0^{-1}d\alpha=\sum_{j=0}^{N-1}\left[-d_0^{-1}(d-d_0)\right]^jd_0^{-1}(d-d_0)\alpha=\sum_{j=1}^{N-1}(-1)^{j-1}\left[d_0^{-1}(d-d_0)\right]^j\alpha\\=&d_0^{-1}(d-d_0)\alpha-\left[d_0^{-1}(d-d_0)\right]^2\alpha+\left[d_0^{-1}(d-d_0)\right]^3\alpha+\cdots+(-1)^{N-2}\left[d_0^{-1}(d-d_0)\right]^{N-1}\alpha\\=&\underbrace{d_0^{-1}d_1}_{d_0^{-1}\partial_1}\alpha+\underbrace{d_0^{-1}(d_2-d_1d_0^{-1}d_1)}_{d_0^{-1}\partial_2}\alpha+\underbrace{d_0^{-1}\left(d_3-d_2d_0^{-1}d_1-d_1d_0^{-1}d_2+d_1d_0^{-1}d_1d_0^{-1}d_1\right)}_{d_0^{-1}\partial_3}\alpha+\cdots\\=&\sum_{r=1}^{N-1}d_0^{-1}\partial_r\alpha
\end{align*}
so that
\begin{align*}
    d\Pi_E\alpha=&d\bigg(\alpha-\sum_{r=1}^{N-1}d_0^{-1}\partial_r\alpha\bigg)=(d-d_0)\alpha-(d-d_0)d_0^{-1}\sum_{r=1}^{N-1}\partial_r\alpha-d_0d_0^{-1}\sum_{r=1}^{N-1}\partial_r\alpha\\=&\sum_{j=1}^{w_s}d_j\alpha-\sum_{j=1}^{w_s}d_jd_0^{-1}\sum_{r=1}^{N-1}\partial_r\alpha-d_0d_0^{-1}\sum_{r=1}^{N-1}\partial_r\alpha\\=&d_1\alpha+\cdots+d_{w_s}\alpha-d_1d_0^{-1}\sum_{r=1}^{N-1}\partial_r\alpha-d_2d_0^{-1}\sum_{r=1}^{N-1}\partial_r\alpha-\cdots-d_{w_s}d_0^{-1}\sum_{r=1}^{N-1}\partial_r\alpha-d_0d_0^{-1}\sum_{r=1}^{N-1}\partial_r\alpha\\=&d_1\alpha+\underbrace{d_2\alpha-d_1d_0^{-1}\partial_1\alpha}_{\partial_2\alpha}+\underbrace{d_3\alpha-d_2d_0^{-1}\partial_1\alpha-d_1d_0^{-1}\partial_2\alpha}_{\partial_3\alpha}+\cdots+\partial_{N-1}\alpha-d_0d_0^{-1}\sum_{r=1}^{N-1}\partial_r\alpha\\=&\sum_{r=1}^{N-1}\partial_r\alpha-d_0d_0^{-1}\sum_{r=1}^{N-1}\partial_r\alpha=\sum_{r=1}^{N-1}\partial_r\alpha-\operatorname{pr}_{\operatorname{Im}d_0}\sum_{r=1}^{N-1}\partial_r\alpha=\operatorname{pr}_{(\operatorname{Im}d_0)^\perp}\sum_{r=1}^{N-1}\partial_r\alpha
\end{align*}
Finally, since $d_0^{-1}$ acts trivially on $(\operatorname{Im}d_0)^\perp$ and we have just shown that $d\Pi_E\alpha\in\ker\Box_0\oplus\operatorname{Im}\delta_0$, we get
\begin{align*}
    \Pi_0d\Pi_E\alpha=&\left(\operatorname{Id}-d_0^{-1}d_0\right)d\Pi_E\alpha=\left(\operatorname{Id}-d_0^{-1}d_0\right)\bigg(\sum_{r=1}^{N-1}\partial_r\alpha-d_0d_0^{-1}\sum_{r=1}^{N-1}\partial_r\alpha\bigg)\\=&\sum_{r=1}^{N-1}\partial_r\alpha-d_0d_0^{-1}\sum_{r=1}^{N-1}\partial_r\alpha-d_0^{-1}d_0\sum_{r=1}^{N-1}\partial_r\alpha\,.
\end{align*}    
\end{proof}
\begin{remark}\
    The fact that $d\Pi_E\alpha$ belongs to $\left(\operatorname{Im}d_0\right)^\perp=\ker\Box_0\oplus\operatorname{Im}\delta_0$ also follows from the fact that $d\Pi_E\alpha=\Pi_Ed\alpha\in E\subset\ker\delta_0=\left(\operatorname{Im} d_0\right)^\perp$ by Lemma \ref{lem: Pi_E}. This property will be fundamental in the definition of the appropriate manifolds over which we should integrate in order to obtain Stokes' theorem for $d_c$ (see Definition \ref{def: R manifolds}).
\end{remark}
Since $d_0d_0^{-1}$ and $d_0^{-1}d_0$ are two projections of bidegree $(0,0)$, i.e. they keep both the weight and the degree constant, the operator $d_c$ consists of a sum of operators $\partial_r-d_0d_0^{-1}\partial_r-d_0^{-1}d_0\partial_r$ of bidegree $|\partial_r|=(r,1-r)$. Since we will be interested in each one of such operators, we introduce the following notation
\begin{definition}\label{def: d_c^r}
    Given the Rumin differential $d_c$ whose explicit expression is given in \eqref{eq: formula for d_c in terms of partials}, we introduce the notation
    \begin{align*}
        d_c^r=\partial_r-d_0d_0^{-1}\partial_r-d_0^{-1}d_0\partial_r\ \text{ for each }r=1,\ldots,N-1
    \end{align*}
    to denote each summand of $d_c$ of bidegree $|d_c^r|=(r,1-r)$.
\end{definition}
Using the properties of the Rumin differential $d_c$, and in particular the fact that $E\subset(\operatorname{Im}d_0)^\perp=\ker\Box_0\oplus\operatorname{Im}\delta_0$, we can now introduce the concept of $R$-manifolds, that is manifolds for which Stokes' theorem holds for the Rumin complex $(E_0^\bullet,d_c)$.
\begin{definition}[$R$-manifolds]\label{def: R manifolds}
    Consider $\Sigma\subset\mathbb{G}$ an oriented $k$-dimensional $C^1$-manifold (with or without boundary). We say that $\Sigma$ is an $R$-manifold if 
    \begin{align*}
        \int_\Sigma\eta=0\ \text{ for any }\eta\in\Omega^k(\G)\cap\operatorname{Im}\delta_0\,.
    \end{align*}
\end{definition}
Since our results throughout the paper rely on being able to apply the classical Stokes' theorem, we require sufficient Euclidean regularity of the manifold with boundary we are considering.
\begin{theorem}[Stokes' theorem on $R$-manifolds]\label{thm: stokes on R manifolds}
    Let $\Sigma\subset\G$ be a $k$-dimensional $R$-manifold whose boundary $\partial\Sigma$ is also an $R$-manifold, then Stokes' theorem holds for the Rumin complex $(E_0^\bullet,d_c)$, that is
    \begin{align*}
        \int_{\partial\Sigma}\alpha=\int_\Sigma d_c\alpha\ \text{ for all }\alpha\in E_0^{k-1}\,.
    \end{align*}
\end{theorem}
\begin{proof}
    Let $\alpha\in E_0^{k-1}$ be Rumin form of degree $k-1$, then on one hand
    \begin{align*}
        \Pi_E\alpha=\alpha-\Pi\alpha=\alpha-d_0^{-1}\sum_{r=1}^{N-1}\partial_r\alpha\in\ker\Box_0\oplus\operatorname{Im}\delta_0
    \end{align*}
    while on the other
    \begin{align*}
        d\Pi_E\alpha=d\Pi_E\alpha-d_0^{-1}d_0d\Pi_E\alpha+d_0^{-1}d_0d\Pi_E\alpha=d_c\alpha+d_0^{-1}d_0d\Pi_E\alpha\in\ker\Box_0\oplus\operatorname{Im}\delta_0\,.
    \end{align*}
    Therefore, the claim directly follows from the classical Stokes' theorem since
    \begin{align*}
        \int_{\partial\Sigma}\alpha\underbrace{=}_{(\ast)}\int_{\partial\Sigma}\Pi_E\alpha=\int_{\Sigma}d\Pi_E\alpha=\int_{\Sigma}\left(d_c\alpha+d_0^{-1}d_0d\Pi_E\alpha\right)\underbrace{=}_{(\ast)}\int_{\Sigma}d_c\alpha\,,
    \end{align*}
    where by $(\ast)$ we denote the equalities that follow from $\partial\Sigma$ and $\Sigma$ being $R$-manifolds.
\end{proof}
\begin{remark}
    As a side note, we remark the same statement can be reached by using the properties listed in Lemma \ref{lem: properties of proejctions} since $$d\Pi_E\alpha=\Pi_E d\alpha=\Pi_E\Pi_0\Pi_Ed\alpha=\Pi_E\Pi_0d\Pi_E\alpha=\Pi_E d_c\alpha=d_c\alpha+\operatorname{Im}\delta_0\,.$$
    This is the idea behind the definition of the boundary operator $\partial_c$ introduced in \cite{julia2023flat}.
\end{remark}
Note that the condition that a manifold $\Sigma$ be an $R$-manifold, as in Definition~\ref{def: R manifolds}, is stronger than what is actually required for the validity of Stokes' theorem for the Rumin complex. Indeed, when dealing with the boundary term, the assumption that $\partial\Sigma$ is an $R$-manifold is unnecessary whenever $\Pi_E\alpha=\alpha$. Similarly, the requirement that $\Sigma$ itself be an $R$-manifold can be dropped whenever $d\Pi_E\alpha\in\ker\Box_0$.
Both situations admit natural characterisations in terms of the homogeneous weights of the forms involved.
\begin{proposition}\label{prop: when R manifold is not needed}
    Given $\G$ a positively graded group, let us consider a Rumin form $\alpha\in E_0^{k-1}$ of homogeneous weight $w(\alpha)=p$, then
    \begin{enumerate}
        \item if $p=\max \{w(\beta)\mid \beta\in\Omega^{k-1}(\G)\}$, then $\Pi_E\alpha=\alpha$;
        \item if $d\Pi_E\alpha=(d\Pi_E\alpha)_{p+j_1}+\cdots+(d\Pi_E\alpha)_{p+j_s}\in\oplus_{i=1}^s\Omega^{p+j_i,k-p-j_i}(\G)$ then $\left(d\Pi_E\alpha\right)_{p+j_1}\in E_0^k$\,.
    \end{enumerate}
\end{proposition}
Before proceeding with the proof of this Proposition, we need a technical lemma regarding the behaviour of the differentials $\partial_r$ (see also \cite[Lemma 3.3]{tripaldi2026spectralcomplexestruncatedmulticomplexes}).
\begin{lemma}\label{lem: d_0 of partial_r}
    Given an arbitrary $k$-form $\alpha\in\Omega^{p,k-p}(\G)$ with $d_0\alpha=0$, i.e. $\alpha\in\ker d_0$, we have
    \begin{align*}
        d_0\partial_1\alpha=0\ \text{ and }\ d_0\partial_r\alpha=-\sum_{i=1}^{r-1}d_i\left(\partial_{r-i}-d_0d_0^{-1}\partial_{r-i}\right)\alpha\ \text{ for any }r\ge 2\,.
    \end{align*}
\end{lemma}
\begin{proof}
    The first claim readily follows from the fact that $\alpha\in\ker d_0$ and the multicomplex relations \eqref{eq: multicomplex maps rules} for the maps $d_i$ since
    \begin{align*}
        d_0\partial_1\alpha=d_0d_1\alpha=-d_1d_0\alpha=0\ \Longrightarrow\ \partial_1\alpha\in\ker d_0\,.
    \end{align*}
    Proceeding similarly, if $r=2$, we get
    \begin{align*}
        d_0\partial_2\alpha=d_0(d_2-d_1d_0^{-1}d_1)\alpha=-d_1^2\alpha-d_2d_0\alpha+d_1d_0d_0^{-1}d_1\alpha=-d_1(d_1-d_0d_0^{-1}d_1)\alpha=-d_1(\partial_1-d_0d_0^{-1}\partial_1)\alpha\,.
    \end{align*}
    In general, for $r\ge 2$, the multicomplex relations \eqref{eq: multicomplex maps rules} can be re-expressed as
    \begin{align}\label{eq: d_0 of d_r}
        d_0d_r=-\sum_{i=1}^{r-1}d_id_{r-i}-d_rd_0
    \end{align}
    and so
    \begin{align*}
        d_0\partial_r\alpha=&d_0\bigg(d_r-\sum_{i=1}^{r-1}d_{r-i}d_0^{-1}\partial_i\bigg)\alpha=-\sum_{i=1}^{r-1}d_id_{r-i}\alpha-d_rd_0\alpha+\sum_{i=1}^{r-1}\sum_{j=1}^{r-i-1}d_jd_{r-i-j}d_0^{-1}\partial_i\alpha+\sum_{i=1}^{r-1}d_{r-i}d_0d_0^{-1}\partial_i\alpha\\=&-\sum_{j=1}^{r-1}d_jd_{r-j}\alpha+\sum_{j=1}^{r-2}\sum_{i=1}^{r-j-1}d_jd_{r-i-j}d_0^{-1}\partial_i\alpha+\sum_{j=1}^{r-1}d_{j}d_0d_0^{-1}\partial_{r-j}\alpha\\=&-d_{r-1}d_1\alpha+d_{r-1}d_0d_0^{-1}\partial_1\alpha-\sum_{j=1}^{r-2}d_j\bigg(d_{r-j}-\sum_{i=1}^{r-j-1}d_{(r-j)-i}d_0^{-1}\partial_i\bigg)\alpha+\sum_{j=1}^{r-2}d_jd_0d_0^{-1}\partial_{r-j}\alpha\\=&-d_{r-1}(\partial_1-d_0d_0^{-1}\partial_1)\alpha-\sum_{j=1}^{r-2}d_j\left(\partial_{r-j}-d_0d_0^{-1}\partial_{r-j}\right)\alpha=-\sum_{j=1}^{r-1}d_j\left(\partial_{r-j}-d_0d_0^{-1}\partial_{r-j}\right)\alpha\,.
    \end{align*}
\end{proof}
\begin{proof}[Proof of Proposition \ref{prop: when R manifold is not needed}]
    The first property follows from the assumption that $\alpha$ is a form of maximum weight amongst all forms in degree $k-1$. Indeed, since $\Omega^{w,k-1-w}(\G)=0$ for every $w>p=w(\alpha)$, we have
    \begin{align*}
        \Pi_E\alpha=\alpha-\underbrace{d_0^{-1}\sum_{r\ge 1}\partial_r\alpha}_{\text{weight }\ge p+1}=\alpha\in\ker\Box_0\cap\Omega^{k-1}(\G)\,.
    \end{align*}
    The second property requires applying the formulae found in Lemma \ref{lem: d_0 of partial_r} to prove, as a first step, that if $\left(d\Pi_E\alpha\right)_{p+l}=0$ for each $l=1,\ldots,j_1-1$, then $\partial_{j_1}\alpha\in\ker d_0\cap\Omega^k(\G)$. This can be quickly seen using the equality $\left(d\Pi_E\alpha\right)_{r}=\partial_r\alpha-d_0d_0^{-1}\partial_r\alpha$ for each $r\in\N$.
    
    If $j_1=1$, then by Lemma \ref{lem: d_0 of partial_r} we directly have that $d_1\alpha\in\ker d_0\cap\Omega^k(\G)$. For $j_1>1$, assuming $\left(d\Pi_E\alpha\right)_{p+l}=0$ for each $l=1,\ldots,j_1-1$ means that $\partial_l\alpha-d_0d_0^{-1}\partial_l\alpha=0$, and so Lemma \ref{lem: d_0 of partial_r} gives
    \begin{align*}
        d_0\partial_{j_1}\alpha=-\sum_{l=1}^{j_1-1}d_{j_1-l}\left(\partial_l-d_0d_0^{-1}\partial_l\right)\alpha=0\ \Longrightarrow\ \partial_{j_1}\alpha\in\ker d_0\,.
    \end{align*}

    Given $\alpha\in E_0^{k-1}$ of homogeneous weight $p$, assuming that
    \begin{align*}
        d\Pi_E\alpha=\left(d\Pi_E\alpha\right)_{p+j_1}+\left(d\Pi_E\alpha\right)_{p+j_2}+\cdots+\left(d\Pi_E\alpha\right)_{p+j_s}\in\bigoplus_{i=1}^s\Omega^{p+j_i,k-p-j_i}(\G)\,.
    \end{align*}implies that the non-trivial component of lowest weight is $\left(d\Pi_E\alpha\right)_{p+j_1}$, that is $\left(d\Pi_E\alpha\right)_{p+l}=0$ for every $l=1,\ldots,j_1$. This might either happen because $\ker\Box_0\cap\Omega^{p+l,k-p-l}(\G)=0$, $d_c^l\alpha=0$, or a combination of the two, for each $l=1,\ldots,j_1-1$.

    By the previous computations, we get $\partial_{j_1}\alpha\in\ker d_0\cap\Omega^k(\G)$, and so $\left(d\Pi_E\alpha\right)_{p+j_1}=\partial_{j_1}\alpha-d_0d_0^{-1}\partial_{j_1}\alpha\in\ker\Box_0\cap\Omega^k(\G)$. 
    
    As a side note, once we know that $\left(d\Pi_E\alpha\right)_{p+j_1}\neq 0$, there is no way of controlling whether $(d\Pi_E\alpha)_{p+i}$ belongs to $\ker\Box_0\cap\Omega^k(\G)$ for higher weights. Indeed, in general, $(d\Pi_E\alpha)_{p+i}\in\big(\ker\Box_0\oplus\operatorname{Im}\delta_0\big)\cap\Omega^k(\G)$ for any $i>j_1$ (see Subsection \ref{subsection: rumin on HtimesR} for an explicit example where this holds true).
\end{proof}

\section{Re-expressing the Rumin complex in view of Stokes' theorem}\label{section: Rumin leibniz}

In view of Theorem \ref{thm: stokes on R manifolds}, it is possible to reinterpret the Rumin complex, and in particular the action of the Rumin differential, in terms of Stokes' theorem. The aim of this section is to show that the Rumin complex can be recovered from the de Rham complex solely through the Leibniz rule for the exterior derivative together with integration along suitable manifolds, such as \(R\)-manifolds.

A key feature of the Heisenberg groups is that no choice of ``special representatives'' in \(\operatorname{Im}\delta_0\) is required when applying the Leibniz rule. This highlights the naturality of the Rumin complex in this setting (see Subsection \ref{subsection: Rumin complex Leibniz Hn}). By contrast, already in the case of \(\H^1\times\R\) (see Subsection \ref{subsection: rumin on HtimesR}), recovering the Rumin complex from the Leibniz rule and integration does require choosing representatives in \(\operatorname{Im}\delta_0\). Consequently, the construction depends explicitly on the choice of scalar product introduced in Subsection \ref{subsection: introducing a scalar product}.

The general form of this construction on positively graded groups is stated in Proposition \ref{prop: d_c from Leibniz}.

\subsection{The Rumin complex on $\H^n$}\label{subsection: Rumin complex Leibniz Hn} Let us start by considering the 3-dimensional Heisenberg group, denoting by $(X,Y,T)$ the orthonormal basis of the Lie algebra $\mathfrak{h}^1$ with non-trivial bracket relation $[X,Y]=T$ and by $(dx,dy,\tau)$ its dual basis. It is well-known that $\H^1$ is stratifiable with a 2-step stratification $\mathfrak{h}^1=V_1\oplus V_2$ given by $V_1=\operatorname{span}_\R\{X,Y\}$ and $V_2=\operatorname{span}_\R\{T\}$. Let us study the space of smooth forms $\Omega^\bullet(\H^1)$ in terms of the Hodge decomposition \eqref{eq: hodge decompo Box_0} in each degree
\begin{itemize}
    \item $\Omega^0(\H^1)=C^\infty(\H^1)=\ker\Box_0\cap\Omega^0(\H^1)$;
    \item $\Omega^1(\H^1)=(\ker\Box_0\oplus\operatorname{Im}\delta_0)\cap\Omega^1(\H^1)$, where $\ker\Box_0\cap\Omega^1(\H^1)=\Omega^{1,0}(\H^1)=\operatorname{span}_{C^\infty(\H^1)}\{dx,dy\}$, $\operatorname{Im}\delta_0\cap\Omega^1(\H^1)=\Omega^{2,-1}(\H^1)=\operatorname{span}_{C^\infty(\H^1)}\{\tau\}$, and $\operatorname{Im}d_0\cap\Omega^1(\H^1)=0$;
    \item $\Omega^{2}(\H^1)=(\ker\Box_0\cap\operatorname{Im}d_0)\cap\Omega^2(\H^1)$, where $\ker\Box_0\cap\Omega^2(\H^1)=\Omega^{3,-1}(\H^1)=\operatorname{span}_{C^\infty(\H^1)}\{dx\wedge\tau,dy\wedge\tau\}$, $\operatorname{Im}d_0\cap\Omega^2(\H^1)=\Omega^{2,0}(\H^1)=\operatorname{span}_{C^\infty(\H^1)}\{dx\wedge dy\}$, and $\operatorname{Im}\delta_0\cap\Omega^2(\H^1)=0$;
    \item $\Omega^3(\H^1)=\Omega^{4,-1}(\H^1)=\ker\Box_0\cap\Omega^3(\H^1)=\operatorname{span}_{C^\infty(\H^1)}\{dx\wedge
     dy\wedge\tau\}$.
\end{itemize}

Let us consider an arbitrary Rumin 1-form $\alpha=f_1dx+f_2dy\in E_0^1$, then
\begin{align*}
    d\alpha=&(Xf_2-Yf_1)dx\wedge dy-Tf_1dx\wedge\tau-Tf_2dy\wedge\tau=-(Xf_2-Yf_1)d\tau-Tf_1dx\wedge\tau-Tf_2dy\wedge\tau\\=&-d\big[(Xf_2-Yf_1)\tau\big]+d(Xf_2-Yf_1)\wedge \tau-Tf_1 dx\wedge\tau-Tf_2dy\wedge
    \tau\\=&-d\big[(Xf_2-Yf_1)\tau\big]+\big[X(Xf_2-Yf_1)-Tf_1\big]dx\wedge\tau+\big[Y(Xf_2-Yf_1)-Tf_2\big]dy\wedge
    \tau
\end{align*}
where we only used that fact that $dx\wedge dy\in\operatorname{Im}d_0$ and $$(Xf_2-Yf_1)dx\wedge dy=-d_0\big[(Xf_2-Yf_1)\tau\big]=-(Xf_2-Yf_1)d_0\tau=-(Xf_2-Yf_1)d\tau\,.$$ It is therefore sufficient to use the Leibniz rule, namely $gd\tau=d(g\tau)-dg\wedge\tau$ for any $g\in C^\infty(\H^1)$.

Rearranging the final formula, we get
\begin{align*}
    d\big[f_1 dx+f_2dy+(Xf_2-Yf_1)\tau\big]=\big[X(Xf_2-Yf_1)-Tf_1\big]dx\wedge\tau+\big[Y(Xf_2-Yf_1)-Tf_2\big]dy\wedge
    \tau\,.
\end{align*}
This can easily be re-expressed in terms of the Rumin differential, since by the Leibniz rule we have
\begin{align*}
    gd\tau=&d(g\tau)-dg\wedge \tau=d\big[gd_0^{-1}d\tau\big]-dg\wedge d_0^{-1}d\tau=d\big[d_0^{-1}(gd\tau)\big]-(d-d_0)d_0^{-1}(gd\tau)
\end{align*}
and so
\begin{align*}
    d\alpha=\underbrace{(Xf_2-Yf_1)dx\wedge dy}_{d_1\alpha}\underbrace{-Tf_1dx\wedge\tau-Tf_2dy\wedge\tau}_{d_2\alpha}=d\big[d_0^{-1}d_1\alpha\big]-(d-d_0)d_0^{-1}d_1\alpha+d_2\alpha\,,
\end{align*}
that is
\begin{align*}
    d\big[\underbrace{\alpha-d_0^{-1}d_1\alpha}_{\Pi_E\alpha}\big]=d_2\alpha-d_1d_0^{-1}d_1\alpha \longleftrightarrow d\Pi_E\alpha=d_c\alpha\,.
\end{align*}

In other words, up to an exact form \(d\eta\) with \(\eta\in\operatorname{Im}\delta_0\subset\Omega^1(\H^1)\), the Rumin differential \(d_c\) coincides with the exterior derivative \(d\). Therefore, when integrating over an \(R\)-manifold \(\Sigma\) whose boundary \(\partial\Sigma\) is also an \(R\)-manifold, Theorem \ref{thm: stokes on R manifolds} follows directly from the classical Stokes' theorem for the exterior derivative \(d\).

The same phenomenon holds more generally on any \(\H^n\); see Proposition \ref{prop: stokes on heisenberg groups} for a detailed discussion of the validity of Stokes' theorem in \(\H^n\).

It is important to note that, in some situations, it is essential that the initial form \(\alpha\in E_0^{k-1}\) be a Rumin form. To illustrate this phenomenon, it suffices to consider the case of \(\H^2\), although analogous considerations extend readily to every \(\H^n\) with \(n\geq 1\). Let us denote by $(X_1,X_2,Y_1,Y_2,T)$ the orthonormal basis of the 5-dimensional Lie algebra $\mathfrak{h}^2$ with non-trivial bracket relations $[X_1,Y_1]=[X_2,Y_2]=T$ and by $(dx_1,dx_2,dy_1,dy_2,\tau)$ its dual basis. Then $\H^2$ is stratifiable with a 2-step stratification $\mathfrak{h}^2=V_1\oplus V_2$ given by $V_1=\operatorname{span}_\R\{X_1,X_2,Y_1,Y_2\}$ and $V_2=\operatorname{span}_\R\{T\}$. Then the direct sum decomposition \eqref{eq: decomposition of d} in this case is as follows:
\begin{itemize}
    \item $\Omega^0(\H^2)=C^\infty(\H^2)=\ker\Box_0\cap\Omega^0(\H^2)$;
    \item $\Omega^1(\H^2)=(\ker\Box_0\oplus\operatorname{Im}\delta_0)\cap\Omega^1(\H^2)$, where $\ker\Box_0\cap\Omega^1(\H^2)=\Omega^{1,0}(\H^2)=\operatorname{span}_{C^\infty(\H^2)}\{dx_1,dx_2,dy_1,dy_2\}$, $\operatorname{Im}\delta_0\cap\Omega^1(\H^2)=\Omega^{2,-1}(\H^2)=\operatorname{span}_{C^\infty(\H^2)}\{\tau\}$, and $\operatorname{Im}d_0\cap\Omega^1(\H^2)=0$;
    \item $\Omega^2(\H^2)=(\operatorname{Im}d_0\oplus\ker\Box_0\oplus\operatorname{Im}\delta_0)\cap\Omega^2(\H^2)$, where $\ker\Box_0\cap\Omega^{2}(\H^2)=\operatorname{span}_{C^\infty(\H^2)}\{dx_1\wedge dx_2,dx_1\wedge dy_2,dx_2\wedge dy_1, dy_1\wedge dy_2,dx_1\wedge dx_2-dy_1\wedge dy_2\}$, $\operatorname{Im}d_0\cap\Omega^2(\H^2)=\operatorname{span}_{C^\infty(\H^2)}\{d\tau\}\subset\Omega^{2,0}(\H^2)$, and $\operatorname{Im}\delta_0\cap\Omega^2(\H^2)=\Omega^{3,-1}(\H^2)=\operatorname{span}_{C^\infty(\H^2)}\{dx_1\wedge\tau,dx_2\wedge\tau,dy_1\wedge\tau,dy_2\wedge\tau\}$;
    \item $\Omega^3(\H^2)=(\operatorname{Im}d_0\oplus\ker\Box_0\oplus\operatorname{Im}\delta_0)\cap\Omega^3(\H^2)$, where $\ker\Box_0\cap\Omega^3(\H^2)=\operatorname{span}_{C^\infty(\H^2)}\{dx_1\wedge dx_2\wedge\tau, dx_1\wedge dy_2\wedge\tau, dx_2\wedge dy_1\wedge\tau, dy_1\wedge dy_2\wedge\tau, (dx_1\wedge dy_1-dx_2\wedge dy_2)\wedge\tau\}\subset\Omega^{4,-1}(\H^2)$, $\operatorname{Im}\delta_0\cap\Omega^3(\H^2)=\operatorname{span}_{C^\infty(\H^2)}\{d\tau\wedge\tau\}$, and $\operatorname{Im}d_0\cap\Omega^3(\H^2)=\Omega^{3,0}(\H^2)=\operatorname{span}_{C^\infty(\H^2)}\{dx_1\wedge dx_2\wedge dy_1,dx_1\wedge dx_2\wedge dy_2,dx_1\wedge dy_1\wedge dy_2, dx_2\wedge dy_1\wedge dy_2\}$;
    \item $\Omega^4(\H^2)=(\ker\Box_0\oplus\operatorname{Im}d_0)\cap\Omega^4(\H^2)$, where $\ker\Box_0\cap\Omega^4(\H^2)=\Omega^{5,-1}(\H^2)=\operatorname{span}_{C^\infty(\H^2)}\{dx_1\wedge dx_2\wedge dy_1\wedge\tau, dx_1\wedge dx_2\wedge dy_2\wedge\tau, dx_1\wedge dy_1\wedge dy_2\wedge\tau,dx_2\wedge dy_1\wedge dy_2\wedge\tau\}$, $\operatorname{Im}d_0\cap\Omega^4(\H^2)=\Omega^{4,0}(\H^2)=\operatorname{span}_{C^\infty(\H^2)}\{dx_1\wedge dx_2\wedge dy_1\wedge dy_2\}$, and $\operatorname{Im}\delta_0\cap\Omega^4(\H^2)=0$;
    \item $\Omega^5(\H^2)=\ker\Box_0\cap\Omega^5(\H^2)=\Omega^{6,-1}(\H^2)=\operatorname{span}_{C^\infty(\H^2)}\{dx_1\wedge dx_2\wedge dy_1\wedge dy_2\wedge\tau\}$.
\end{itemize}

Let us consider an arbitrary horizontal \(2\)-form
\[
\alpha=fd\tau+\overline{\alpha}\in\ker d_0\ \text{ where }\ \overline{\alpha}=\Pi_0\alpha\in E_0^2\,.
\]Using reasoning analogous to the one above, we obtain
\begin{align*}
    d\alpha
    &=d(fd\tau)+d\overline{\alpha}=d\big[d(f\tau)-df\wedge\tau\big]
      +d\big[d_0^{-1}d_1\overline{\alpha}\big]
      +d_c\overline{\alpha}=d\big(d_0^{-1}d_1\overline{\alpha}-df\wedge\tau\big)
      +d_c\overline{\alpha}\\
    &\longleftrightarrow
    d\Big(\alpha-\underbrace{d_0^{-1}d_1\overline{\alpha}+df\wedge\tau}_{\in\operatorname{Im}\delta_0}\Big)
    =d_c\overline{\alpha}\,.
\end{align*}

The crucial point is that the correction term belongs to \(\operatorname{Im}\delta_0\). Consequently, when integrating over an \(R\)-manifold \(\Sigma\) whose boundary \(\partial\Sigma\) is also an \(R\)-manifold, this term does not contribute and we obtain
\begin{align*}
    \int_{\partial\Sigma}\alpha
    =\int_{\partial\Sigma}(fd\tau+\overline{\alpha})
    =\int_\Sigma d_c\overline{\alpha}
    =\int_{\partial\Sigma}\overline{\alpha}.
\end{align*}

Consider now an arbitrary \(3\)-form $\alpha\in\ker d_0\cap\Omega^3(\H^2)$,
for instance
\[
\alpha=fdx_1\wedge d\tau+\overline{\alpha}\ \text{ where }\
\overline{\alpha}=\Pi_0\alpha\in E_0^3\subset\Omega^{4,-1}(\H^2)\ \text{ and }\
fdx_1\wedge d\tau\in\operatorname{Im}d_0\subset\Omega^{3,0}(\H^2).
\]
Then
\begin{align*}
    d\alpha
    &=d(fdx_1\wedge d\tau)+d\overline{\alpha}=d\big[d(fdx_1\wedge\tau)-df\wedge dx_1\wedge\tau\big]
      +d_c\overline{\alpha}=d\big(df\wedge dx_1\wedge\tau\big)
      +d_c\overline{\alpha}\\
    &\longleftrightarrow
    d\Big(\alpha+\underbrace{df\wedge dx_1\wedge\tau}_{\notin\operatorname{Im}\delta_0}\Big)
    =d_c\overline{\alpha}\,.
\end{align*}

In this case, the previous argument breaks down since, in general,
$df\wedge dx_1\wedge\tau\in\Omega^{4,-1}(\H^2)$
does not belong to \(\operatorname{Im}\delta_0\). Hence its integral over \(\partial\Sigma\) does not vanish in general. Therefore, for an arbitrary \(R\)-manifold \(\Sigma\) whose boundary \(\partial\Sigma\) is also an \(R\)-manifold, one obtains
\[
\int_{\partial\Sigma}\alpha
-\int_{\partial\Sigma}df\wedge dx_1\wedge\tau
=\int_{\Sigma}d_c\overline{\alpha}
=\int_{\partial\Sigma}\overline{\alpha}
\neq
\int_{\partial\Sigma}\alpha\,.
\]

In this sense, when studying Stokes' theorem on \(\H^2\), it becomes necessary to restrict to Rumin forms when considering degrees at least 3.

The same reasoning applies more generally on \(\H^n\). For \(k\leq n+1\), given an arbitrary form
\[
\alpha\in\ker d_0\cap\Omega^{k-1}(\H^n)\ \text{ with }\
\alpha=\Pi_0\alpha+\operatorname{Im}d_0,
\]
one obtains
\[
\int_{\partial\Sigma}\alpha
=
\int_{\Sigma}d_c\Pi_0\alpha
=
\int_{\partial\Sigma}\Pi_0\alpha,
\]
provided that both \(\Sigma\) and \(\partial\Sigma\) are \(R\)-manifolds.

Equivalently, up to an exact form whose primitive belongs to \(\operatorname{Im}\delta_0\), the exterior derivative of $\alpha\in\ker d_0\cap\Omega^{k-1}(\H^n)$
coincides with the Rumin differential \(d_c\) of $\Pi_0\alpha\in\ker\Box_0\cap\Omega^{k-1}(\H^n)$.

By contrast, when \(k>n+1\), one must restrict to Rumin forms
\[
\alpha\in\ker\Box_0\cap\Omega^{k-1}(\H^n)
\]
in order to obtain Stokes' theorem for \(d_c\).

\subsection{The Rumin complex on $\mathbb H^1\times\R$}\label{subsection: rumin on HtimesR}Let us consider the 4-dimensional connected, simply-connected nilpotent Lie group $\mathbb H^1\times\R$, whose Lie algebra $\mathfrak{h}^1\times\R$ is spanned by the orthonormal basis $(X_1,X_2,X_3,X_4)$ with only one non-trivial Lie bracket $[X_1,X_2]=X_4$. For completeness, we also recall the non-commutative group law $(\mathbb H^1\times\R,\ast)$, together with the corresponding left-invariant vector fields and 1-forms (see also \cite{LeDonneTripaldi2021}). Given $x=(x_1,x_2,x_3,x_4)$ and $y=(y_1,y_2,y_3,y_4)\in\mathbb H^1\times\R$, one has
\begin{align*}
    (x_1,x_2,x_3,x_4)\ast(y_1,y_2,y_3,y_4)=\big(x_1+y_1,x_2+y_2,x_3+y_3,x_4+y_4+\tfrac{x_1y_2-x_2y_1}{2}\big)
\end{align*}
and
\begin{align*}
    X_1=&\partial_1-\frac{x_2}{2}\partial_4\ ,\ X_2=\partial_2+\frac{x_1}{2}\partial_4\ ,\ X_3=\partial_3\ ,\ X_4=\partial_4\\
    \theta_1=dx_1&\ ,\ \theta_2=dx_2\ ,\ \theta_3=dx_3\ ,\ \theta_4=dx_4-\frac{x_1}{2}dx_2+\frac{x_2}{2}dx_1\,.
\end{align*}
Using the fact that $d\theta_1=d\theta_2=d\theta_3=0$, while $d\theta_4=-\theta_1\wedge\theta_2$, one readily obtains an explicit description of the space of smooth forms through the direct sum decomposition \eqref{eq: hodge decompo Box_0} in each degree. To simplify the otherwise cumbersome notation, we will denote the group $\H^1\times\R$ simply by $\G$ thoughout the remainder of this subsection.
\begin{itemize}
    \item $\Omega^0(\G)=C^\infty(\G)=\ker\Box_0\cap\Omega^0(\G)$;
    \item $\Omega^1(\G)=(\ker\Box_0\oplus \operatorname{Im}\delta_0)\cap\Omega^1(\G)$, where $\ker\Box_0\cap\Omega^1(\G)=\Omega^{1,0}(\G)=\operatorname{span}_{C^\infty(\G)}\{\theta_1,\theta_2,\theta_3\}$, $\operatorname{Im}\delta_0\cap\Omega^1(\G)=\operatorname{span}_{C^\infty(\G)}\{\theta_4\}$, and $\operatorname{Im}d_0\cap\Omega^1(\G)=0$;
    \item $\Omega^2(\G)=(\operatorname{Im}d_0\oplus\ker\Box_0\oplus\operatorname{Im}\delta_0)\cap\Omega^2(\G)$, where $\ker\Box_0\cap\Omega^{2}(\G)=\operatorname{span}_{C^\infty(\G)}\{\theta_1\wedge\theta_3,\theta_2\wedge\theta_3,\theta_1\wedge\theta_4,\theta_2\wedge\theta_4\}\subset\Omega^{2,0}(\G)\oplus\Omega^{3,-1}(\G)$, $\operatorname{Im}d_0\cap\Omega^2(\G)=\operatorname{span}_{C^\infty(\G)}\{\theta_1\wedge\theta_2\}\subset\Omega^{2,0}(\G)$, and $\operatorname{Im}\delta_0\cap\Omega^2(\G)=\operatorname{span}_{C^\infty(\G)}\{\theta_3\wedge\theta_4\}\subset\Omega^{3,-1}(\G)$;
    \item $\Omega^3(\G)=(\ker\Box_0\oplus\operatorname{Im}d_0)\cap\Omega^3(\G)$, where $\ker\Box_0\cap\Omega^3(\G)=\Omega^{4,-1}(\G)=\operatorname{span}_{C^\infty(\G)}\{\theta_1\wedge\theta_2\wedge\theta_4,\theta_1\wedge\theta_3\wedge\theta_4,\theta_2\wedge\theta_3\wedge\theta_4\}$, $\operatorname{Im}d_0\cap\Omega^3(\G)=\Omega^{3,0}(\G)=\operatorname{span}_{C^\infty(\G)}\{\theta_1\wedge\theta_2\wedge\theta_3\}$, and $\operatorname{Im}\delta_0\cap\Omega^3(\G)=0$;
    \item $\Omega^4(\G)=\ker\Box_0\cap\Omega^4(\G)=\Omega^{5,-1}(\G)=\operatorname{span}_{C^\infty(\G)}\{\theta_1\wedge\theta_2\wedge\theta_3\wedge\theta_4\}$.
\end{itemize}
Similarly to what was done in the case of $\H^n$, let us first consider an arbitrary horizontal 1-form $\alpha=f_1\theta_1+f_2\theta_2+f_3\theta_3\in\ker\Box_0\cap\Omega^1(\G)$:
\begin{align*}
    d\alpha=&(X_1f_2-X_2f_1)\theta_1\wedge\theta_2+(X_1f_3-X_3f_1)\theta_1\wedge\theta_3+(X_2f_3-X_3f_2)\theta_2\wedge\theta_3+\\&-X_4f_1\theta_1\wedge\theta_4-X_4f_2\theta_2\wedge\theta_4-X_4f_3\theta_3\wedge\theta_4\\=&-(X_1f_2-X_2f_1)d\theta_4+(X_1f_3-X_3f_1)\theta_1\wedge\theta_3+(X_2f_3-X_3f_2)\theta_2\wedge\theta_3+\\&-X_4f_1\theta_1\wedge\theta_4-X_4f_2\theta_2\wedge\theta_4-X_4f_3\theta_3\wedge\theta_4\\=&-d\big[(X_1f_2-X_2f_1)\theta_4\big]+d(X_1f_2-X_2f_1)\wedge\theta_4+(X_1f_3-X_3f_1)\theta_1\wedge\theta_3+(X_2f_3-X_3f_2)\theta_2\wedge\theta_3+\\&-X_4f_1\theta_1\wedge\theta_4-X_4f_2\theta_2\wedge\theta_4-X_4f_3\theta_3\wedge\theta_4\\=&-d\big[(X_1f_2-X_2f_1)\theta_4\big]+(X_1f_3-X_3f_1)\theta_1\wedge\theta_3+(X_2f_3-X_3f_2)\theta_2\wedge\theta_3+\\&+\big[X_1(X_1f_2-X_2f_1)-X_4f_1\big]\theta_1\wedge\theta_4+\big[X_2(X_1f_2-X_2f_1)-X_4f_2\big]\theta_2\wedge\theta_4+\\&+\big[X_3(X_1f_2-X_2f_1)-X_4f_3\big]\theta_3\wedge\theta_4
\end{align*}
so that by rearranging the expression we get
\begin{align*}
    d\big[\alpha+\underbrace{(X_1f_2-X_2f_1)\theta_4}_{\operatorname{Im}\delta_0}\big]=d\big[\underbrace{\alpha-d_0^{-1}d_1\alpha}_{\Pi_E\alpha}\big]=&d_1\alpha-d_0d_0^{-1}d_1\alpha+(d_2-d_1d_0^{-1}d_1)\alpha\\=&d_c^1\alpha+d_c^2\alpha+\underbrace{\big[X_3(X_1f_2-X_2f_1)-X_4f_3\big]\theta_3\wedge\theta_4}_{\operatorname{Im}\delta_0}\\=&d_c\alpha+\underbrace{\big[X_3(X_1f_2-X_2f_1)-X_4f_3\big]\theta_3\wedge\theta_4}_{\operatorname{Im}\delta_0}\,.
\end{align*}
This identity is not as straightforward as the one obtained for the Rumin differential on \(\H^n\). In the present setting, the exterior derivative of \(\alpha\) does not coincide with the Rumin differential \(d_c\alpha\) modulo an exact form whose primitive belongs to \(\operatorname{Im}\delta_0\). Instead, one has
\[
d\alpha=d_c\alpha+\zeta \ \mod d(\operatorname{Im}\delta_0)\ ,\ \text{ where }
\zeta\in\operatorname{Im}\delta_0\,.
\]
Nevertheless, similarly to the previous situation, both correction terms vanish after integrating on \(R\)-manifolds. Consequently, provided that both \(\Sigma\) and \(\partial\Sigma\) are \(R\)-manifolds, the Stokes' theorem of Theorem \ref{thm: stokes on R manifolds} again follows directly from the classical Stokes' theorem for the exterior derivative \(d\).

Interestingly, the situation becomes even more delicate when dealing with \(2\)-forms. Let us consider a horizontal Rumin \(2\)-form
\[
\alpha=f_1\theta_1\wedge\theta_3+f_2\theta_2\wedge\theta_3
\in\ker\Box_0\cap\Omega^{2,0}(\G).
\]
A direct computation gives
\begin{align*}
    d\alpha
    ={}&(X_1f_2-X_2f_1)\theta_1\wedge\theta_2\wedge\theta_3+X_4f_1\theta_1\wedge\theta_3\wedge\theta_4
    +X_4f_2\theta_2\wedge\theta_3\wedge\theta_4\,.
\end{align*}
Let us focus on the term $\theta_1\wedge\theta_2\wedge\theta_3\in\operatorname{Im}d_0\cap\Omega^3(\G)$.
Since $
d_0^{-1}(\theta_1\wedge\theta_2\wedge\theta_3)
=
\theta_3\wedge\theta_4
\in\operatorname{Im}\delta_0\cap\Omega^2(\G)$,
one could repeat the same reasoning as before and write
\[
(X_1f_2-X_2f_1)d(\theta_3\wedge\theta_4)
=
(X_1f_2-X_2f_1)\theta_1\wedge\theta_2\wedge\theta_3\,,
\] before applying the Leibniz rule to recover the expression of the Rumin differential up to the exact form
\[
d\big[(X_1f_2-X_2f_1)\theta_3\wedge\theta_4\big].
\]

However, unlike the case of \(\H^n\), there is no canonical reason to choose the primitive belonging to \(\operatorname{Im}\delta_0\). Indeed, for every \(c_1,c_2\in\mathbb R\),
\begin{align*}
    \theta_1\wedge\theta_2\wedge\theta_3
    &=
    d(\theta_3\wedge\theta_4)=
    d(c_1\theta_1\wedge\theta_4
    +c_2\theta_2\wedge\theta_4
    +\theta_3\wedge\theta_4)
\end{align*} since $\theta_1\wedge\theta_4,\,
\theta_2\wedge\theta_4
\in
\ker\Box_0\cap\Omega^{3,-1}(\G)$.

Notice that choosing an arbitrary horizontal \(2\)-form would generate a recursive chain of corrections, which is precisely the obstruction avoided by the definition of the projection \(\Pi_E\). By contrast, choosing a primitive in \(\ker d_0\cap\Omega^{3,-1}(\G)\) is perfectly admissible and leads to genuinely different expressions.

Indeed, proceeding in this way yields
\begin{align*}
    d\alpha
    ={}&
    (X_1f_2-X_2f_1)
    d(c_1\theta_1\wedge\theta_4
    +c_2\theta_2\wedge\theta_4
    +\theta_3\wedge\theta_4)+X_4f_1\theta_1\wedge\theta_3\wedge\theta_4
    +X_4f_2\theta_2\wedge\theta_3\wedge\theta_4\\
    ={}&
    d\Big[(X_1f_2-X_2f_1)
    (c_1\theta_1\wedge\theta_4
    +c_2\theta_2\wedge\theta_4
    +\theta_3\wedge\theta_4)\Big]\\
    &-d(X_1f_2-X_2f_1)
    \wedge
    (c_1\theta_1\wedge\theta_4
    +c_2\theta_2\wedge\theta_4
    +\theta_3\wedge\theta_4)+X_4f_1\theta_1\wedge\theta_3\wedge\theta_4
    +X_4f_2\theta_2\wedge\theta_3\wedge\theta_4\\
    ={}&
    d\Big[(X_1f_2-X_2f_1)
    (c_1\theta_1\wedge\theta_4
    +c_2\theta_2\wedge\theta_4
    +\theta_3\wedge\theta_4)\Big]\\
    &+\big[c_1X_2(X_1f_2-X_2f_1)
    -c_2X_1(X_1f_2-X_2f_1)\big]
    \theta_1\wedge\theta_2\wedge\theta_4\\
    &+\big[X_4f_1
    -c_1X_3(X_1f_2-X_2f_1)
    +X_1(X_1f_2-X_2f_1)\big]
    \theta_1\wedge\theta_3\wedge\theta_4\\
    &+\big[X_4f_2
    -c_2X_3(X_1f_2-X_2f_1)
    +X_2(X_1f_2-X_2f_1)\big]
    \theta_2\wedge\theta_3\wedge\theta_4\,.
\end{align*}

Rewriting the previous identity in terms of the projected representative \(\Pi_E\alpha\), we obtain
\begin{align*}
    d\Big[\Pi_E\alpha
    +(X_1f_2-X_2f_1)&
    \underbrace{(c_1\theta_1\wedge\theta_4
    +c_2\theta_2\wedge\theta_4)}_{\ker\Box_0\cap\Omega^{3,-1}(\G)}
    \Big]
    ={}
    d_c\alpha\\
    &+\big[c_1X_2(X_1f_2-X_2f_1)
    -c_2X_1(X_1f_2-X_2f_1)\big]
    \theta_1\wedge\theta_2\wedge\theta_4\\
    &-c_1X_3(X_1f_2-X_2f_1)
    \theta_1\wedge\theta_3\wedge\theta_4-c_2X_3(X_1f_2-X_2f_1)
    \theta_2\wedge\theta_3\wedge\theta_4\,.
\end{align*}

In particular, the resulting expression depends explicitly on the choice of primitive for the \(\operatorname{Im}d_0\)-component. This phenomenon sharply contrasts with the situation in \(\H^n\), where the Leibniz rule canonically recovers the Rumin differential.

In other words, even if both \(\Sigma\) and \(\partial\Sigma\) are \(R\)-manifolds, integrating the previous identity gives
\begin{align*}
    \int_{\partial\Sigma}\alpha
    &+
    \int_{\partial\Sigma}(X_1f_2-X_2f_1)
    (c_1\theta_1\wedge\theta_4+c_2\theta_2\wedge\theta_4)\\
    &=
    \int_\Sigma d_c\alpha+
    \int_\Sigma
    \big[c_1X_2(X_1f_2-X_2f_1)
    -c_2X_1(X_1f_2-X_2f_1)\big]
    \theta_1\wedge\theta_2\wedge\theta_4\\
    &\quad-
    \int_{\Sigma}
    c_1X_3(X_1f_2-X_2f_1)
    \theta_1\wedge\theta_3\wedge\theta_4-
    \int_{\Sigma}
    c_2X_3(X_1f_2-X_2f_1)
    \theta_2\wedge\theta_3\wedge\theta_4 .
\end{align*}
Therefore, in general, the Stokes' theorem of Theorem \ref{thm: stokes on R manifolds} is recovered only for the distinguished choice \(c_1=c_2=0\), namely the choice of primitive lying in \(\operatorname{Im}\delta_0\).

Finally, as in the case of \(\H^n\) in degrees \(k>n\), one sees that it is necessary to start with a Rumin form
\[
\alpha\in\ker\Box_0\cap\Omega^{k-1}(\G)
\]
in each degree (see also Subsection \ref{subsection: intrinsic graphs vs R mfld on H1timesR}). For instance, let us consider an arbitrary \(2\)-form
\[
\alpha=f\,d\theta_4+\overline{\alpha}_2+\overline{\alpha}_3\in\ker d_0\cap\Omega^2(\G),
\]
where
\[
\overline{\alpha}_2=(\Pi_0\alpha)_2\in\ker\Box_0\cap\Omega^{2,0}(\G)\ \text{ and }\
\overline{\alpha}_3=(\Pi_0\alpha)_3\in\ker\Box_0\cap\Omega^{3,-1}(\G)\,.
\]
Then
\begin{align*}
    d\alpha
    &=
    d(f\,d\theta_4)+d\overline{\alpha}_2+d\overline{\alpha}_3=
    d\big[d(f\theta_4)-df\wedge\theta_4\big]
    +d\big[d_0^{-1}d_1\overline{\alpha}_2\big]
    +d_c^2\overline{\alpha}_2
    +d_c^1\overline{\alpha}_3 .
\end{align*}
Equivalently,
\begin{align*}
    d\Big[
    \alpha
    &-
    \underbrace{d_0^{-1}d_1\overline{\alpha}_2
    +X_3f\,\theta_3\wedge\theta_4}_{\operatorname{Im}\delta_0}
    +
    \underbrace{X_1f\,\theta_1\wedge\theta_4
    +X_2f\,\theta_2\wedge\theta_4}_{\ker\Box_0}
    \Big]
    =
    d_c\overline{\alpha},
\end{align*}
where \(\overline{\alpha}=\Pi_0\alpha=\overline{\alpha}_2+\overline{\alpha}_3\).

Thus, even if both \(\Sigma\) and \(\partial\Sigma\) are \(R\)-manifolds, the \(\ker\Box_0\)-component does not disappear in general, and
\[
    \int_{\partial\Sigma}\alpha
    +
    \int_{\partial\Sigma}\Pi_0(df\wedge\theta_4)
    =
    \int_\Sigma d_c\overline{\alpha}
    =
    \int_{\partial\Sigma}\overline{\alpha}
    \neq
    \int_{\partial\Sigma}\alpha \,.
\]
This shows that, in order to recover Stokes' theorem for the Rumin differential, one must restrict to Rumin forms rather than arbitrary \(d_0\)-closed representatives.

\subsection{The Rumin differential from the Leibniz rule on arbitrary Carnot groups}

The computations carried out in the previous subsections for \(\H^n\) and \(\H^1\times\R\) show that the Rumin differential can be recovered recursively from the classical exterior derivative by repeatedly applying the Leibniz rule and separating the resulting terms according to the Hodge decomposition \eqref{eq: hodge decompo Box_0} induced by \(d_0\), \(\delta_0\), and \(\Box_0\). 

In the case of Heisenberg groups, this procedure is canonical and naturally reproduces the structure of the Rumin complex. By contrast, the example of \(\H^1\times\R\) highlights that, on more general positively graded groups, the reconstruction process may depend on the choice of primitives for the \(\operatorname{Im}d_0\)-components, and that the role of the projection onto \(\ker\Box_0\) becomes essential.

The aim of this subsection is to formalise this phenomenon in full generality. More precisely, we show that, starting from a Rumin form \(\alpha\in\ker\Box_0\cap\Omega^{k-1}(\G)\), the whole expression of the Rumin differential \(d_c\alpha\) can be reconstructed recursively from the Leibniz rule, modulo terms belonging to \(\operatorname{Im}\delta_0\).
\begin{proposition}[The Rumin complex from the Leibniz rule]\label{prop: d_c from Leibniz}
    Let $\G$ be a positively gradable Lie group and let $\alpha\in\ker\Box_0\cap\Omega^{k-1}(\G)$ be a Rumin form. Then one can retrieve the whole expression of the Rumin differential from applying the Leibniz rule recursively, up to forms in $\operatorname{Im}\delta_0$.
\end{proposition}
\begin{proof}
    Let $\alpha\in\ker\Box_0\cap\Omega^{k-1}(\G)$, then
    \begin{align*}
        d\alpha=&d_1\alpha+d_2\alpha+\cdots+d_{w_s}\alpha=d_0d_0^{-1}(d-d_0)\alpha+\Pi_0(d-d_0)\alpha+d_0^{-1}d_0(d-d_0)\alpha\\=&d\big[d_0^{-1}(d-d_0)\alpha\big]-(d-d_0)d_0^{-1}(d-d_0)\alpha+\Pi_0(d-d_0)\alpha+d_0^{-1}d_0(d-d_0)\alpha\\=&d\big[d_0^{-1}d\alpha\big]-d_0d_0^{-1}\big[(d-d_0)d_0^{-1}d\big]\alpha+\Pi_0\big[d\alpha-(d-d_0)d_0^{-1}d\alpha\big]+d_0^{-1}d_0\big[d\alpha-(d-d_0)d_0^{-1}d\alpha\big]\\=&d\big[d_0^{-1}d\alpha-d_0^{-1}(d-d_0)d_0^{-1}d\alpha\big]+(d-d_0)d_0^{-1}(d-d_0)d_0^{-1}d\alpha+\Pi_0\big[d\alpha-(d-d_0)d_0^{-1}d\alpha\big]+\\&+d_0^{-1}d_0\big[d\alpha-(d-d_0)d_0^{-1}d\alpha\big]=\cdots=\\=&d\big[\sum_{i=0}^N(-1)^i[d_0^{-1}(d-d_0)]^{i+1}\alpha\big]+\Pi_0(d-d_0)\big[\sum_{i=0}^N(-1)^i[d_0^{-1}(d-d_0)]^{i}\big]\alpha+\\&+d_0^{-1}d_0(d-d_0)\big[\sum_{i=0}^N(-1)^i[d_0^{-1}(d-d_0)]^{i}\big]\alpha
    \end{align*}
    where the fact that there exists $N\in\mathbb N$ such that $[d_0^{-1}(d-d_0)]^{N+1}=0$ is due to the fact that the range of possible weights of forms in each degree is finite. From the explicit expressions of $\Pi_E=\operatorname{Id}-\Pi$ as expressed in \eqref{eq: inverse Id-b}, we get that
    \begin{align*}
        \alpha-\sum_{i=0}^N(-1)^i[d_0^{-1}(d-d_0)]^{i+1}\alpha=\Pi_E\alpha
    \end{align*}
    while from Lemma \ref{lem: expressing d_c on E_0} we get that
    $$(d-d_0)\big[\sum_{i=0}^N(-1)^i[d_0^{-1}(d-d_0)]^{i}\big]\alpha=\sum_{i=1}^N\partial_i\alpha$$
    so that, once we rearrange the whole formula we get
    \begin{align*}
        d\big[\Pi_E\alpha\big]=&d\bigg[\alpha-\underbrace{\sum_{i=0}^N(-1)^i[d_0^{-1}(d-d_0)]^{i+1}\alpha}_{\operatorname{Im}\delta_0}\bigg]\\=&\Pi_0\bigg(\sum_{i=1}^N\partial_i\alpha\bigg)+d_0^{-1}d_0\bigg(\sum_{i=1}^N\partial_i\alpha\bigg)=d_c\alpha+\underbrace{d_0^{-1}d_0\bigg(\sum_{i=1}^N\partial_i\alpha\bigg)}_{\operatorname{Im}\delta_0}\,.
    \end{align*}
\end{proof}

\begin{remark}
    Observe that in the proof of Proposition \ref{prop: d_c from Leibniz} we crucially used the identity
    \begin{align*}
        d_0d_0^{-1}\xi
        =
        d\big[d_0^{-1}\xi\big]
        -(d-d_0)d_0^{-1}\xi\ \text{ for every }\xi\in\Omega^k(\G)\,.
    \end{align*}
    However, the choice of the primitive \(d_0^{-1}\xi\) is not unique. Indeed, for every \(\beta\in\ker d_0\), one also has
    \begin{align*}
        d_0d_0^{-1}\xi
        =
        d_0(\beta+d_0^{-1}\xi)
        =
        d\big(\beta+d_0^{-1}\xi\big)
        -(d-d_0)\big(\beta+d_0^{-1}\xi\big).
    \end{align*}

    Consequently, the recursive reconstruction of the exterior derivative in terms of the Rumin differential depends essentially on the distinguished choice of primitive given by \(d_0^{-1}\). In general, arbitrary choices of \(\beta\) may introduce additional terms for which the recursive procedure is no longer controlled, even if one assumes \(w(\beta)\geq w(\xi)\), which is the condition ensuring the nilpotency of the operator \(d_0^{-1}(d-d_0)\).
\end{remark}

\section{Intrinsic graphs}\label{section intrinsic graphs}
In this section, we present some properties of oriented integration on submanifolds, with a special emphasis on the integration of homogeneous differential forms on smooth intrinsic graphs. Oriented integration can also be represented using a Riemannian metric and the corresponding Riemannian surface measure. However, the idea of developing a theory of currents as duals of the Rumin complex leads one to consider a surface measure modelled on the weights of the complex, \cite{DiMJulNGoloVit25,Vit22,Can21jga,FSSC6,Mag06}. 
The spherical (or the Hausdorff) measure constructed from a homogeneous distance of the group naturally takes into account the weights of the Rumin forms. Actually, this was one of the motivations to find explicit area formulas for the spherical measure of submanifolds, that is strictly related to the degree (or weight) of the submanifold, as previously formulated in \cite{Mag12A}.

The notion of \textit{degree} as introduced in \cite{Mag13Vit} turns out to be crucial (see also \cite{Magnani2019Area,Mag22RS,CorMag25} for further developments). As we will see in Definition \ref{def: weight of manifold}, the notion of degree of a homogeneous $k$-vector is perfectly analogous to the notion of weight of a homogeneous differential form as presented in Definition \ref{def: weights of forms}. For this reason, the degree of a $k$-vector and its weight will be understood as being the same, and the two terms will be used interchangeably, as long as the terminology is clear from the context.

The aim of this section is to study the relationship between smooth intrinsic graphs and the $R$-manifolds introduced above. To this end, we first recall the definition of an intrinsic graph, its main properties and the notion of the degree of a submanifold. This will allow us to determine whether a given smooth intrinsic graph can be an $R$-manifold. 
To introduce intrinsic graphs, we need to refer to a specific factorisation of $\G$.

\begin{definition}[Complementary subgroups]\label{def: complementary subgroups} Given a positively graded group $\G$ together with two homogeneous subgroups $\W$ and $\V$ of $\G$, we say that $\W$ and $\V$ are {\em complementary subgroups} of $\G$ if
\begin{align*}
    \W\cap\V=\{e\}\ \text{ and }\ \G=\W\V.
\end{align*}
To emphasize the order of factorisation, we may say that $(\W,\V)$ is a {\em pair of complementary subgroups}.
\end{definition}
In equivalent terms, $\W$ and $\V$ are complementary subgroups if and only if for each $p\in\G$ there exists a unique pair $(w,v)\in\W\times\V$ such that $p=wv$. 
\begin{definition}[Group projections]\label{def: group proj} 
If $(\W,\V)$ is a pair of complementary subgroups, we know that for any $p\in\G$ there exists a unique pair $(w,v)\in\W\times\V$ such that $w v=g$. Consequently the following group projections are well-defined
\begin{align}
    \pi_\W\colon\G\to\W\ ,\ \pi_\W(w v)=w\ \text{ and }\  \pi_\V\colon\G\to\V\ ,\ \pi_\V(w v)=v.
\end{align}
\end{definition}
Notice that the order in the choice of $\W$ and $\V$ matters. Indeed, if $\W$ and $\V$ are complementary subgroups in $\G$, i.e. $\G=\W\V$, then it is also true that $\G=\V\W$. However the non-commutative group operation yields two different group projections. For this reason, indicating the ordered pair $(\W,\V)$ of complementary subgroups uniquely determines the group projections of Definition \ref{def: group proj}. 

We are now in the position to introduce the notion of intrinsic graph.
\begin{definition}[Intrinsic graphs]\label{def: intrinsic graphs} Let $(\W,\V)$ be a pair of complementary subgroups of $\G$, and let $A\subset\W$. We define the intrinsic graph of $\phi\colon A\to\V$ as the set
\begin{align*}
    \operatorname{graph}(\phi):=\{w\phi(w)\mid w\in A\}\subset\G\,.
\end{align*}
The \textit{graph map} $\Phi\colon A\to\G$ of $\phi$ is defined as $\Phi(w):=w\phi(w)$ for all $w\in A$.
\end{definition}
Intrinsic graphs are quite natural objects in the development of geometric measure theory of Carnot groups. They originally appeared in finding the proper version of the De Giorgi's rectifiability theorem in Heisenberg groups, \cite{FSSC01}. Subsequently, intrinsic graphs in Carnot groups have been thoroughly studied in \cite{FMS14,FranchiSerapioni}. The intrinsic regularity of an intrinsic graphs comes from the intrinsic differentiability of the defining mapping. Actually, in many cases this request is equivalent to require that the uniformly intrinsically differentiable intrinsic graph is locally the non-critical level set of a Pansu differentiable mapping, \cite{arena2009intrinsic,DiD21,ADDDLD24,CorniPhD}, although the problem is still open for arbitrary uniformly intrinsically differentiable graphs.

The relationship between simple left-invariant Rumin forms and pairs of complementary subgroups was already studied in Section 2 of \cite{FranchiSerapioni}. We will include a shorter version here, in order to make this presentation more self-contained.

\begin{lemma}[Proposition 2.9 in \cite{FranchiSerapioni}]\label{lem: subalgebras}
Let $\G$ be a Lie group of dimension $n$ and denote by $\mathfrak{g}$ the Lie algebra of left invariant vector fields on $\G$. Without loss of generality, we may assume that a scalar product is fixed in $\mathfrak{g}$. Let $\mathfrak{w}$ be a {$k$-dimensional subspace of $\mathfrak{g}$}, and let $Z_1,\ldots,Z_{n-k}$ be a basis of $\mathfrak{w}^\perp$. If we set $\theta_i:=Z_i^\ast$ for $i=1,\ldots,n-k$ and $\theta:=\theta_1\wedge\cdots\wedge\theta_{n-k}$, then we have
    \begin{itemize}
        \item [i.] $\mathfrak{w}=\{X\in\mathfrak{g}\mid \iota_X\theta=0\}=\{X\in\mathfrak{g}\mid \star\theta\wedge X^\ast=0\}$.
        \item[ii.] {If $\mathfrak w$ is a Lie subalgebra, then there} exists $\eta\in\mathfrak{g}^\ast$ such that $d\theta=\eta\wedge\theta$.
        \item[iii.]
        {Conversely,} if $\theta=\theta_1\wedge\cdots\wedge\theta_{n-k}\in\bigwedge^{n-k}\mathfrak{g}^\ast$ is a simple left invariant form such that $d\theta=\eta\wedge\theta$ for some $\eta\in\mathfrak{g}^\ast$, then
        $\mathfrak{w}:=\{X\in\mathfrak{g}\mid \iota_X\theta=0\}$ is a Lie subalgebra of $\mathfrak{g}$.
    \end{itemize}
    Here $\iota_X$ denotes the interior product. Notice that  if $\mathfrak g$ is nilpotent, then $\mathfrak w$ is a Lie subalgebra if and only if $d\theta=0$. 
\end{lemma}


The previous lemma gives the following proposition, also taken from \cite{FranchiSerapioni}.
\begin{proposition}[Theorem 2.10 in \cite{FranchiSerapioni}]\label{prop: complementary subgroups vs Rumin forms}
    Given two $\xi\in E_0^k$ and $\theta\in E_0^{n-k}$ two simple left invariant Rumin forms with $1\le k<n$ such that $\xi\wedge\theta\neq 0$, then
    \begin{align*}
        \mathfrak{w}:=\{X\in\mathfrak{g}\mid \iota_X\theta=0\}\ \text{ and }\ \mathfrak{v}:=\{X\in\mathfrak{g}\mid \iota_X\xi=0\}
    \end{align*}
    are both Lie subalgebras of $\mathfrak{g}$. Moreover, $\operatorname{dim}\mathfrak{w}=k$, $\operatorname{dim}\mathfrak{v}=n-k$, and $\mathfrak{g}=\mathfrak{w}\oplus\mathfrak{v}$. If, in addition, $\xi=\xi_1\wedge\cdots\wedge\xi_k$ and $\theta=\theta_1\wedge\cdots\wedge\theta_{n-k}$, where all the $\xi_i$s and the $\theta_i$s all have homogeneous weights $p_i$ and $q_i$ respectively, then both $\mathfrak{w}$ and $\mathfrak{v}$ are homogeneous Lie subalgebras of $\mathfrak{g}$, and so if we set
    \begin{align*}
        \W:=\exp(\mathfrak{w})\ \text{ and }\ \V:=\exp(\mathfrak{v})
    \end{align*}
    then $\W$ and $\V$ are complementary subgroups. In particular, since $\star E_0^k=E_0^{n-k}$, if $\xi\in E_0^k$, we can choose $\theta=\star\xi$, and in this case $\mathfrak{w}$ and $\mathfrak{v}$ are orthogonal.

    Conversely, suppose $\mathfrak{w}$ and $\mathfrak{v}$ are homogeneous Lie subalgebras of $\mathfrak{g}$ such that $\operatorname{dim}\mathfrak{w}=k$ and $\operatorname{dim}\mathfrak{v}=n-k$, and $\mathfrak{g}=\mathfrak{w}\oplus\mathfrak{v}$, then there exists a scalar product $\langle\cdot,\cdot\rangle$ on $\mathfrak{g}$, $\xi\in E_0^k$, and $\theta\in E_0^{n-k}$ such that $\xi\wedge\theta\neq 0$ and
    \begin{align*}
        \mathfrak{w}:=\{X\in\mathfrak
        g\mid \iota_X\theta=0\}\ \text{ and }\ \mathfrak{v}:=\{X\in\mathfrak{g}\mid \iota_X\xi=0\}\,.
    \end{align*}
\end{proposition}

\begin{remark}\label{rmk: example Rumin forms nonsimple}
Proposition \ref{prop: complementary subgroups vs Rumin forms} highlights a clear correspondence between the space of homogeneous simple Rumin forms and complementary subgroups.
    A natural question is whether one can always find a basis of Rumin forms made of homogeneous simple left-invariant forms. Indeed, this is not the case. Take for example the 9-dimensional 2-step nilpotent Lie group $\G$ whose Lie algebra $\mathfrak{g}$ has the following non-trivial brackets:
    \begin{align*}
        [X_1,X_2]=[X_3,X_4]=T_1\ ,\ [X_1,X_3]=T_2\ ,\ [X_1,X_4]=T_3\ ,\ [X_2,X_3]=T_4\ ,\ [X_2,X_4]=T_5\,.
    \end{align*}
    Then, if $(X_1,X_2,X_3,X_4,T_1,T_2,T_3,T_4,T_5)$ is taken as an orthonormal basis of $\mathfrak{g}$ and we use the notation $dx_i=X_i^\ast$ and $\tau_j=T_j$ to denote the dual orthonormal basis of $\mathfrak{g}^\ast$, straightforward computations show that
    \begin{align*}
        E_0^2=\operatorname{span}_{C^\infty(\G)}\{dx_1\wedge dx_2-dx_3\wedge dx_4\}\oplus \Theta^{3,-1}_{16}
    \end{align*}
    where $\Theta^{3,-1}_{16}$ denotes a 16-dimensional subspace of $\Omega^{3,-1}(\G)$, the space of smooth 2-forms of weight 3. In this case, it is clear that one cannot find a basis of simple covectors to span the 1-dimensional subspace of Rumin 2-forms of weight 2.
\end{remark}

To prove the statements needed to study whether smooth intrinsic graphs are $R$-manifolds, we will first introduce some crucial concepts. We refer for instance to \cite{Magnani2019Area,corni2026minimal} for more information. 

\begin{definition}[Definition 2.1 in \cite{Mag13Vit}] Let $(X_1,\ldots,X_n)$ be a basis of $\mathfrak{g}$ adapted to the grading \eqref{eq: grading direct sum decomposition}. The degree $\operatorname{deg}(j)$ of $X_j$ is the unique integer $p$ such that $X_j\in V_p$. Moreover, given $X_J:=X_{j_1}\wedge\cdots\wedge X_{j_k}$ be a simple $k$-vector of $\bigwedge^k\mathfrak{g}$, then the degree of $X_J$ is the integer $\operatorname{deg}(J)$ defined as the sum $\operatorname{deg}(j_1)+\cdots+\operatorname{deg}(j_k)$.    
\end{definition}
By comparing this concept of degree of a simple $k$-vector with Definition \ref{def: weights of forms}, it is clear that the two concepts coincide.

Let us now fix a graded metric $g$ on $\G$, namely a left-invariant Riemannian metric on $\G$ such that the subspaces $V_{w_i}$ in \eqref{eq: grading direct sum decomposition} are orthogonal. It is easy to observe that all left-invariant Riemannian metrics such that the basis $(X_1,\ldots,X_n)$ adapted to the grading is orthogonal are graded metrics and the family of $\{X_{J_k}=X_{j_1}\wedge\cdots\wedge X_{j_k}\}_{J_k}$ forms an orthonormal basis of $\bigwedge^k\mathfrak{g}$ with respect to the induced metric. The norm induced by $g$ on $\bigwedge^k\mathfrak{g}$ will be simply denoted by $|\cdot|$.
\begin{definition}[Definition 2.3 in \cite{Mag13Vit}] Let $\tau\in\bigwedge^k\mathfrak{g}$ be a simple $k$-vector and let $1\le p\le Q$ be a natural number. Let $\tau=\sum_{J_k}\tau_{J_k}X_{J_k}$ be expressed in terms of a fixed orthonormal basis $(X_1,\ldots,X_n)$ adapted to the grading. Similarly to the notation introduced in the case of covectors (and forms), the projection of $\tau$ to the component of degree $r$ is defined as
\begin{align*}
    \big(\tau\big)_r=\sum_{d(J_k)=r}\tau_{J_k}X_{J_k}\,.
\end{align*}
The \textit{degree} of $\tau$ is then defined as the maximum degree, that is
\begin{align*}
    \operatorname{deg}(\tau)=\max\{r\in\mathbb N\mid (\tau)_r\neq 0\}\,.
\end{align*}
\end{definition}
The same concept can be expressed in terms of weights, that is, given a non-homogeneous $k$-vector $\tau$ expressed as a linear combination of simple $k$-vectors, the degree of $\tau$ is understood as the maximal homogeneous weight of the basis $k$-vectors with non-zero coefficients.

We are now ready to define the degree of a submanifold.
\begin{definition}[Degree of a submanifold]\label{def: weight of manifold}
    Let $\Sigma$ be a $C^1$ $k$-dimensional submanifold of a positively graded Lie group $\G$. For $x\in\Sigma$, let $\tau_\Sigma(x)\in\bigwedge^k\mathfrak{g}$ be any unit tangent $k$-vector to $\Sigma$ at $x\in\Sigma$ with respect to an arbitrary auxiliary Riemannian metric $\Tilde{g}$, i.e. $\vert\tau_\Sigma(x)\vert_{\Tilde{g}}=1$. Then the degree of the point $x\in\Sigma$ is defined as
    \begin{align*}
        \operatorname{deg}_\Sigma(x)=\operatorname{deg}(\tau_\Sigma(x))\,,
    \end{align*}
    while the degree of the submanifold $\Sigma$ is defined as
    \begin{align*}
        \operatorname{deg}(\Sigma)=\max\{\operatorname{deg}_\Sigma(x)\mid x\in\Sigma\}\,.
    \end{align*}
    We will say that the point $x\in\Sigma$ has maximum degree if $\operatorname{deg}_\Sigma(x)=\operatorname{deg}(\Sigma)$.
\end{definition}
It is not difficult to show that this definition is independent of the choice of the auxiliary metric $\Tilde{g}$, since the degree is invariant under multiplication by non-zero scalars. Moreover, it is independent of the choice of the fixed adapted basis $(X_1,\ldots,X_n)$ and depends only on the tangent spaces $T_x\Sigma$ and the homogeneous structure \eqref{eq: grading direct sum decomposition} of $\mathfrak{g}$, i.e. on the homogeneous weight/degree of $\tau_\Sigma(x)$.

\begin{lemma}\label{lem: deg boundary less}
Let $\Sigma\subset\G$ be a $C^1$ oriented manifold with boundary.  Then $\deg(\partial \Sigma)<\deg(\Sigma)$.
\end{lemma}
\begin{proof}
Let $k$ be the dimension of $\Sigma$. Let us consider $p\in\partial \Sigma$ such that $d_{\partial\Sigma}(p)=\deg(\partial \Sigma)$, let $(e_1,\ldots,e_k)$ be a basis of $T_p\Sigma$, such that $(e_1,\ldots,e_{k-1})$ is a basis of $T_p(\partial \Sigma)$. Whenever $1\le k'\le n$, we set
\[
I(k',n)=\{(j_1,j_2,\ldots,j_{k'})\in(\Z^+)^{k'}\mid  1\le j_1<j_1<\cdots<j_{k'}\le n \}
\]
and observe that there exists $\beta=(\beta_1,\ldots,\beta_{k-1})\in I(k-1,n)$ such that $\deg(\beta)=\deg(\partial \Sigma)$.
Consider now the matrix
\[
\left(
\begin{array}{cccc}
c_1^1 & c_2^1 & \cdots  & c_k^1 \\
\vdots & \vdots & \cdots & \vdots \\
c_1^n & c_2^n & \cdots  & c_k^n
\end{array}
\right)
=
C,
\]
where $e_i=\sum_{j=1}^n c_i^j X_j(p)$. We consider the $k$-tangent vector 
\[
\tau_{\Sigma}(p)=\sum_{\gamma\in I(k,n)}\tau_{\gamma}^{\Sigma}\; X_{\gamma}
=\sum_{\gamma\in I_{\beta}(k,n)}\tau_{\gamma}^{\Sigma} X_{\gamma}
+\sum_{\gamma\in I(k,n)\setminus I_{\beta}(k,n)}\tau_{\gamma}^{\Sigma} X_{\gamma}
\in\Lambda^k(T_p\Sigma)
\]
and we have defined
\[
I_{\beta}(k,n)
=
\left\{
(j_1,\ldots,j_k)\in (\mathbb N^+)^k
\;\middle|\;
1\leq j_1<\cdots<j_k\leq n,\;
\{\beta_1,\ldots,\beta_{k-1}\}
\subset
\{j_1,\ldots,j_k\}
\right\}.
\]
By contradiction, if we have 
\[
\tau_{\gamma}^{\bar\Sigma}=0
\qquad
\text{for all } \gamma\in I_{\beta}(k,n),
\]
then all rows of \(C\) are linear combinations of the rows
\[
(c_1^{\beta_1},\ldots,c_k^{\beta_1}), \; 
(c_1^{\beta_2},\ldots,c_k^{\beta_2}), \; 
\ldots,
(c_1^{\beta_{k-1}},\ldots,c_k^{\beta_{k-1}})
\]
of \(C\), therefore proving that the rank of $C$ is less than $k$. 
This conflicts with the fact that $\tau_\Sigma(p)\neq0$.
We have proved that there exists $\alpha\in I_{\beta}(k,n)$ such that 
$\tau^\Sigma_{\alpha}\neq0$, hence 
\[
d(\Sigma)\ge \deg(\alpha)>\deg(\beta)=\deg(\partial \Sigma),
\]
where the strict inequality holds since $\alpha$ contains all indexes of $\beta$.
This completes the proof.
\end{proof}

\begin{lemma}\label{lem: forms of higher weight, integral vanishes}
Let $\alpha\in\Omega^{p,k-p}(\G)$ be a $k$-form of homogeneous weight $w(\alpha)=p$ and let $\Sigma\subset\G$ be an oriented $k$-dimensional compact submanifold. 
If $p>\deg(\Sigma)$, then $\int_\Sigma\alpha=0$.
\end{lemma}
\begin{proof}
We set $N=\deg(\Sigma)$. By assumption, $w(\alpha)=p$ and so we may write $\alpha=\sum_{I\colon w(\theta_I)=p}f_I\theta_I$. We can orient $\Sigma$ by a unit tangent $k$-vector field $\tau$ on $\Sigma$, where $\tau=\sum_{J\colon w(X_J)\le N}\tau_JX_J$. Here, both $f_I$ and $\tau_J$ are smooth functions on $\Sigma$. Then we have
    \begin{align*}
        \int_\Sigma\alpha=
\int_\Sigma \alpha(\tau)\,d\mathrm{vol}_k=
\sum_{I:\,w(\theta_I)=p}
\sum_{J:\,w(X_J)\le N}
\int_\Sigma f_I\tau_J\,\theta_I(X_J)\,d\mathrm{vol}_k
    \end{align*}
    where $\operatorname{vol}_k$ is the $k$-dimensional Riemannian surface measure on $\Sigma$, and \(\{\theta_I\}\) is the dual basis to \(\{X_I\}\). Since \(w(\theta_I)=p>w(X_J)\) for each $J$ by the assumption $p>N$, we readily get that \(I\neq J\) and hence $\theta_I(X_J)=0$ for every such pair $I,J$ by Lemma \ref{lem: different weight implies lin indep}. Thus the claim holds true.
\end{proof}

\begin{remark}
Let $\W\subset \mathbb G$ be any homogeneous subgroup of a positively graded group $\mathbb G$. Since $\W$ is invariant under the group dilations of $\mathbb G$, it inherits a natural structure of homogeneous positively graded Lie group. In particular, if the Lie algebra $\mathfrak w$ of $\W$ admits the graded decomposition $\mathfrak w = W_{w_1} \oplus \cdots \oplus W_{w_s}$, then it is straightforward to notice that the Jacobian of the dilations $\{\delta_\lambda|_\W\}_{\lambda>0}$ is $\lambda^{Q_\W}$, where 
\[
Q_\W= \sum_{i=1}^s w_i\, \dim W_{w_i}
\]
is the {\em homogeneous dimension} of $\W$. 
If we endowed $\W$ with a homogeneous distance $\rho$, then the uniqueness of the Haar measure yields $C>0$ such that
\[
\mathcal{S}^{Q_\W}\big(B_\W(w,r)\big)= C\, r^{Q_\W},
\]
where $B_\W(w,r)$ denotes the metric ball in $\W$ of center $w$ and radius $r$ with respect to $\rho$, 
and $\mathcal S^{Q_\W}$ is the $Q_\W$-dimensional spherical measure constructed from $\rho$. 
In particular, the Hausdorff dimension $\operatorname{dim}_H(\W)$ of $\W$ with respect to $\rho$ satisfies
\[
\dim_H(\W) = Q_\W\,.
\]
Moreover, we choose a graded basis $(Y_1,\ldots,Y_k)$ of $\mathfrak w$, where 
$k=\operatorname{dim}(\W)$ and $\operatorname{dim}(\W)$ denotes the topological dimension of $\W$. Due to the left invariance of $Y_j$, we have
\[
\deg\big(Y_1(w)\wedge \cdots \wedge Y_k(w)\big)=Q_\W=\dim_H\W
\]
for every $w\in\W$. We may finally conclude that 
\[
\deg(\W)=Q_\W=\dim_H(\W).
\]
\end{remark}
The next proposition proves the constancy of the pointwise degree on a smooth intrinsically differentiable intrinsic graph $\Sigma$. It mainly relies on two area formulas for the spherical measure of a submanifold. In particular the equality $\dim_H\Sigma=\dim_H\W$ will arise as an important tool. Here $\W$ denotes the first factor of the factorising pair $(\W,\V)$.
\begin{proposition}\label{prop: smooth intrinsic graphs constant degree}
Let $\Sigma=\operatorname{graph}(\phi)\subset\G$ be a $(\W,\V)$-graph, with $\phi\colon A\to\V$, $A\subset\W$ open, and $\phi$ of class $C^2$. If $\phi$ is also continuously intrinsically differentiable, then $\operatorname{deg}_\Sigma(x)=\deg(\W)$ for every $x\in\Sigma$.
\end{proposition}
\begin{proof}
Let $\Phi: A \subset W \to \mathbb G$ be the intrinsic graph map defined by $\Phi(w) = w  \phi(w)$.
We set $N=\deg(\W)$.
Since $\phi$ is continuously intrinsically differentiable, we can apply the area formula for intrinsic graphs (see \cite[Theorem 1.2]{CorMag23pr}). Moreover, by \cite[Proposition 7.2]{CorMag23pr}, the intrinsic Jacobian $J\Phi$ is continuous and strictly positive on $A$.

Therefore, for every open, bounded, and sufficiently small set $\Omega \subset \mathbb G$ such that $\Sigma_\Omega := \Sigma \cap \Omega \neq \emptyset$, we have $
0 < \mathcal S^N(\Sigma_\Omega) < \infty$,
where $N$ is the Hausdorff dimension of $\W$.

Let $N_0 := \operatorname{deg}(\Sigma)$ be the degree of $\Sigma$, and define $\Sigma_0 := \{x \in \Sigma : \operatorname{deg}_\Sigma(x) = N_0\}$,
which is a nonempty open subset of $\Sigma$. Set
$A_0 := \Phi^{-1}(\Sigma_0) \subset A$,
which is a nonempty open subset of $A$.

Since $\Phi$ is a $C^1$ diffeomorphism onto $\Sigma$, it provides a global parametrization of $\Sigma$. Hence, for every open, bounded, and sufficiently small set $\Omega' \subset \mathbb G$ such that $\emptyset \neq \Sigma_{\Omega'} := \Sigma \cap \Omega' \subset \Sigma_0$, we can apply the area formula for submanifolds (see \cite[(1.4)]{Mag13Vit}) and obtain $0 < \mathcal S^{N_0}(\Sigma_{\Omega'}) < \infty$.
Since a set cannot have finite positive spherical measure in two different dimensions, it follows that $N=N_0$.

Next, combining the intrinsic area formula with the classical Euclidean area formula, we obtain that there exists a constant $C>0$ such that for every open set $A' \subset A_0$,
\begin{align*}
\int_{A'} J\Phi(w)\, d\mathcal H^k_{|\cdot|}
=
\int_{\Phi(A')} \frac{J\Phi(\Phi^{-1}(x))}{J_g\Phi(\Phi^{-1}(x))}\, d\mathrm{vol}_k(x)
\le
C \int_{\Phi(A')} |\tau_\Sigma^N(x)|\, d\mathrm{vol}_k(x),
\end{align*}
where $J_g\Phi$ denotes the classical $k$-dimensional Jacobian of $\Phi$ with respect to the fixed scalar product $g$. Thus, taking the averaged integrals on $\Sigma_0\cap B_r(x)$ with respect to $\mathrm{vol}_k$ and letting $r\to 0^+$, we get  $J\Phi(\Phi^{-1}(x))\le C\vert\tau_\Sigma^N(x)\vert\cdot J_g\Phi(\Phi^{-1}(x))$. 
Since both $J\Phi$ and $J_g\Phi$ are continuous and strictly positive on $A$, it follows that
$
|\tau_\Sigma^N(\Phi(w))| > 0 \quad \text{for all } w \in A_0$.

We now prove that $A_0= A$. Assume by contradiction that $A_0 \subsetneq A$. Since $A$ is open and connected and $A_0$ is nonempty and open, there exists a point $w_0 \in \partial A_0 \cap A$.
Let $x_0 := \Phi(w_0) \in \Sigma$. Since $w_0 \notin A_0$, we have $x_0 \notin \Sigma_0$, hence
\[
\operatorname{deg}_\Sigma(x_0) < N = N_0\ \Longrightarrow\ 
|\tau_\Sigma^N(x_0)| = 0\,.
\]

On the other hand, taking a sequence $w_j \in A_0$ such that $w_j \to w_0$, and using the continuity of all terms, we can pass to the limit in $J\Phi(w_j) \le C\, |\tau_\Sigma^N(\Phi(w_j))| \, J_g\Phi(w_j)$
to obtain
\[
J\Phi(w_0) \le C\, |\tau_\Sigma^N(x_0)| \, J_g\Phi(w_0) = 0,
\]
which contradicts the fact that $J\Phi(w_0) > 0$.
Therefore, $A_0 = A$, and so every point of $\Sigma$ has degree $N$.
\end{proof}
As a direct consequence of Lemma \ref{lem: forms of higher weight, integral vanishes} together with Proposition \ref{prop: smooth intrinsic graphs constant degree}, we get the following result.
\begin{corollary}\label{cor: integrating forms on smooth intrinsic graphs}Let $\Sigma=\operatorname{graph}(\phi)\subset\G$ be a $(\W,\V)$-graph, $A\subset\W$ open, and $\phi\colon A\to\V$ of class $C^2$. If $\phi$ is intrinsically differentiable with continuous intrinsic differential, then if $k=\operatorname{dim}\W$ we have
\begin{align*}
    \int_\Sigma\alpha=0\ \text{ for any }\alpha\in\Omega^{p,k-p}(\G)\ \text{ such that }p>\operatorname{deg}(\W)\,.
\end{align*}
    
\end{corollary}

Due to the importance of intrinsic graphs in geometric measure theory on Carnot groups, we introduce a class of surface that are at an intermediate step between smooth surfaces and intrinsically regular surfaces.
\begin{definition}[Locally smooth intrinsic graphs]\label{def:LSIG}
We say that $\Sigma\subset\G$, an orientable $C^1$ submanifold, is a \textit{locally smooth intrinsic graph} if, for any point $p\in\Sigma$ outside a negligible set, there exists a neighbourhood $U$ of $p$ such that $U\cap\Sigma$ is a $C^2$ intrinsic graph.

We say that an orientable $C^1$ manifold with boundary $\partial\Sigma$ is a \textit{locally smooth intrinsic graph with boundary} if both $\Sigma$ and $\partial\Sigma$ are locally smooth intrinsic graphs.

\end{definition}

\subsection{Locally smooth intrinsic graphs vs. $R$-manifolds on Heisenberg groups}\label{subsection: R mflds Heisenberg}
Let us first study the simple case of $\G$ being the $2n+1$-dimensional Heisenberg group $\mathbb H^n$. It is well-known that Stokes' theorem for the Rumin complex holds on locally smooth intrinsic graphs with boundary given by a smooth intrinsic graph, as also shown in \cite{StokesFranchi,FranchiSerapioni,DiMJulNGoloVit25}. 

In this subsection, we will rephrase the same result in a slightly different way: we will first show that locally  smooth intrinsic graphs of topological dimension $k\le n$ are indeed $R$-manifolds, while those of topological dimension $k> n$ in general are not. Crucially, Stokes' theorem also holds in all dimensions, because for $k>n$ we are precisely in case (1) of Proposition~\ref{prop: when R manifold is not needed}, and for each $0<k<2n+1$ the Rumin differential is homogeneous, so we are also in case (2) of Proposition \ref{prop: when R manifold is not needed}. 
\begin{proposition}[Stokes' theorem for the Rumin complex on locally intrinsic graphs in $\H^n$]\label{prop: stokes on heisenberg groups}
    Let $\Sigma\subset\mathbb H^n$ be a $k$-dimensional locally smooth intrinsic graph $\Sigma$ with boundary $\partial\Sigma$, then
    \begin{align*}
        \int_{\partial \Sigma}\alpha=\int_\Sigma d_c\alpha\ \text{ for any }\alpha\in E_0^{k-1}\,.
    \end{align*}
\end{proposition}

       In the case of $\mathbb H^n$, there is a straightforward characterisation of the spaces of Rumin forms $E_0^k$ in terms of their weight (being horizontal or vertical). Moreover, the weight of forms will also determine which space in the direct sum decomposition \eqref{eq: hodge decompo Box_0} they belong to.
       
    Let us denote by $(X_i,Y_i,T)_{i=1}^n$ denote the orthonormal basis of the Lie algebra $\mathfrak{h}^n$ with non-trivial bracket relations $[X_i,Y_i]=T$ for each $i=1,\ldots,n$. We will use $(dx_i,dy_i,\tau)_{i=1}^n$ for its dual basis. It is well-known that $\mathbb H^n$ is stratifiable with a 2-step stratification $\mathfrak{h}^n=V_1\oplus V_2$ given by $V_1=\operatorname{span}_\R\{X_i,Y_i\}_{i=1}^n$ and $V_2=\operatorname{span}_\R\{T\}$, so that for each degree $k=1,\ldots,2n$ the space of smooth forms splits into \textit{horizontal} and \textit{vertical} forms, that is $\Omega^k(\G)=\Omega^{k,0}(\G)\oplus\Omega^{k+1,-1}(\G)$. In degree $k=0$, the space of 0-forms is made of smooth functions that only have weight 0, while in top degree the volume form has weight $2n+2$. Furthermore
    \begin{itemize}
        \item for $1\le k\le n$, $\operatorname{Im}\delta_0=\Omega^{k+1,-1}(\G)$;
        \item for $k=1$, $\operatorname{Im}d_0\cap\Omega^1(\G)=0$, $\ker \Box_0\cap\Omega^1(\G)=\Omega^{1,0}(\G)$, and $\Omega^{1}(\G)\cap\operatorname{Im}\delta_0=\Omega^{2,-1}(\G)$;
        \item for $1<k\le n$, $\Omega^{k,0}(\G)=\big(\operatorname{Im}d_0\oplus\ker\Box_0\big)\cap\Omega^k(\G)$, while $\Omega^{k+1,-1}(\G)=\operatorname{Im}\delta_0\cap\Omega^k(\G)$;
        \item for $n+1\le k\le 2n$,  $\Omega^{k,0}(\G)=\operatorname{Im}d_0\cap\Omega^{k}(\G)$, and $\Omega^{k+1,-1}(\G)=\big(\ker\Box_0\oplus\operatorname{Im}\delta_0\big)\cap\Omega^k(\G)$;
        \item for $k=2n$, $\operatorname{Im}\delta_0\cap\Omega^{2n}(\G)=0$, $\Omega^{2n,0}(\G)=\operatorname{Im}d_0\cap\Omega^{2n}(\G)$, and $\ker\Box_0\cap\Omega^{2n}(\G)=\Omega^{2n+1,-1}(\G)$.
    \end{itemize}
    Crucially, in each degree the space $\operatorname{Im}\delta_0\cap\Omega^k(\G)$ is always a subspace of the space of vertical forms $\Omega^{k+1,-1}(\G)$, that is the space of $k$-forms of weight $k+1$.
    
    On the other hand, intrinsic graphs on $\mathbb H^n$ of low dimension are always ``horizontal'' submanifolds, that is given $\operatorname{graph}(\phi)\subset\H^n$ a $(\W,\V)$-graph with $\operatorname{dim}\W=k\le n$, then $\operatorname{deg}(\operatorname{graph}(\phi))=\deg(\W)=k$. By Corollary \ref{cor: integrating forms on smooth intrinsic graphs}, we obtain that, as long as $1\le k\le n$,
    \begin{align*}
        \int_{\operatorname{graph}(\phi)}\eta=0\ \text{ for any }\eta\in\operatorname{Im}\delta_0\cap\Omega^{k}(\G)
    \end{align*}
    since $\operatorname{Im}\delta_0\cap\Omega^k(\G)=\Omega^{k+1,-1}(\G)$ if $k\le n$.
    As a direct consequence, we have that for a locally smooth intrinsic graph $\Sigma$ of dimension $1\le k\le n$, we have $\operatorname{deg}(\Sigma)=k$, since the degree of $\Sigma$ is the maximum amongst all possible degrees, which are necessarily all equal to $k$. By Corollary \ref{cor: integrating forms on smooth intrinsic graphs}, this implies that such low-dimensional locally smooth intrinsic graphs $\Sigma$ are indeed $R$-manifolds as defined in Definition \ref{def: R manifolds}. By Theorem \ref{thm: stokes on R manifolds}, this implies that for $1\le k\le n$, given $\Sigma$ an orientable $k$-dimensional locally smooth intrinsic graph with boundary $\partial\Sigma$, Stokes' theorem holds and
    \begin{align*}
        \int_{\partial\Sigma}\alpha=\int_\Sigma d_c\alpha\ \text{ for all }\alpha\in E_0^{k-1}\ \text{ with }0\le k<n\,.
    \end{align*}

    In the high-dimensional case (or, equivalently, low codimension as usually expressed in the literature), this is not the case. Indeed, for $n<k<2n$, we have that $\operatorname{Im}\delta_0\cap\Omega^k(\G)$ is a non-trivial subspace of $\Omega^{k+1,-1}(\G)$, and the integration of these forms on intrinsic smooth submanifolds of low codimension might not vanishing.
    
    We first notice that given $\operatorname{graph}(\phi)\subset\H^n$ a $(\W,\V)$-graph with $\operatorname{dim}\W=k\ge n+1$, then $\operatorname{deg}(\W)=k+1$ and so Corollary \ref{cor: integrating forms on smooth intrinsic graphs} does not apply. Indeed, already in $\H^2$, one can easily come up with an example of a smooth intrinsic graph of dimension $3$ over which the integral of the form $d\tau\wedge\tau\in\operatorname{Im}\delta_0$ does not vanish. 

    It suffices to consider the trivial smooth and continuously intrinsically differentiable intrinsic graph
    \[
    \Sigma_0=\{(x_1,0,x_2,0,t):x_1,x_2,t\in\R\}
    \]
    with $\phi\equiv0$ and coordinates $(x_1,y_1,x_2,y_2,t)$.
    We have that, since the 3-form $\tau\wedge d\tau\in\operatorname{Im}\delta_0$ can be expressed explicitly as
    \begin{align*}
    \tau\wedge d\tau=&\tau\wedge (dy_1\wedge dx_1+dy_2\wedge dx_2)=dt\wedge dy_1\wedge dx_1+dt\wedge dy_2\wedge dx_2+\\&+\frac{y_1 dx_1\wedge dy_2\wedge dx_2 -x_1 dy_1\wedge dy_2\wedge dx_2+y_2 dx_2\wedge dy_1\wedge dx_1-x_2 dy_2\wedge dy_1\wedge dx_1}{2}\,, 
    \end{align*}it is straightforward to check that for every nonnegative and nonvanishing $\varphi\in C_c^0(\Sigma_0)$, one has
    \[
    \int_{\Sigma_0} \varphi\, \tau\wedge d\tau=\int_{\operatorname{supp}\varphi} \varphi\,dt\wedge dx_3\wedge dx_1\neq0.
    \]
    In other words, in low codimension smooth and continuously intrinsically differentiable graphs \textit{need not be} $R$-manifolds, although Stokes' theorem for the Rumin complex still holds for them. This is due to the following two facts: for forms of degree $k>n$ we have
    \begin{itemize}
        \item $E_0^k\subset\Omega^{k+1,-1}(\G)$, that is Rumin forms have maximal weight, and so $\Pi_E\alpha=\alpha$ for every $\alpha\in E_0^k$ by (1) of Proposition \ref{prop: when R manifold is not needed};
        \item given $\alpha\in E_0^k\subset\Omega^{k+1,-1}(\G)$, we have $d\Pi_E\alpha =\big(d\Pi_E\alpha\big)_{k+2}$ and so $d\Pi_E\alpha\in\ker\Box_0$ by (2) of Proposition \ref{prop: when R manifold is not needed}, i.e. $d_c\alpha=d\Pi_E\alpha$ (in this specific case, since $\Pi_E\alpha=\alpha$, we have a further simplification, as $d_c\alpha=d\alpha$).
    \end{itemize}
    Hence, given a locally smooth intrinsic graph $\Sigma$ with boundary $\partial\Sigma$, we have
    \begin{align*}
        \int_{\partial\Sigma}\alpha=\int_{\Sigma}d\alpha=\int_\Sigma d_c\alpha\ \text{ for all }\alpha\in E_0^{k-1}\text{ with }n<k\le 2n+1\,,
    \end{align*}
    and both integrals are non-trivial since $\operatorname{deg}(\partial\Sigma)=k$ and $\operatorname{deg}(\Sigma)=k+1$, and we are integrating vertical forms $\alpha$ and $d_c\alpha$ on them respectively (has $\alpha$ and $d_c\alpha$ been horizontal form, both integrals would vanish by Corollary \ref{cor: integrating forms on smooth intrinsic graphs}).

    Finally, the case of degree $n$, where the Rumin differential $d_c\colon E_0^n\to E_0^{n+1}$ has differential order 2, is a mixture of both cases. Let $\Sigma$ be an $(n+1)$-dimensional locally smooth intrinsic graph with boundary $\partial\Sigma$. Since $\alpha\in E_0^{n}\subset\Omega^{n,0}(\G)$, we have $\Pi_E\alpha\neq \alpha$, however $\operatorname{deg}(\partial\Sigma)=n$, and so $\partial\Sigma$ is  an $R$-manifold. The $n+1$-dimensional $\Sigma$ is not since $\operatorname{deg}(\Sigma)=n+2$. However, $d\Pi_E\alpha=(d\Pi_E\alpha)_{n+2}$ for any $\alpha\in E_0^n$, and so by (2) in Proposition \ref{prop: when R manifold is not needed}  we have that $d\Pi_E\alpha\in\ker\Box_0\cap\Omega^{n+1}(\G)$ and $d_c\alpha=d\Pi_E\alpha$. We can therefore directly apply Stokes' theorem to get
    \begin{align*}
        \int_{\partial\Sigma}\alpha\underbrace{=}_{\partial\Sigma\; R\text{-mfld}}\int_{\partial\Sigma}\Pi_E\alpha\underbrace{=}_{\text{Stokes' thm}}\int_{\Sigma}d\Pi_E\alpha=\int_{\Sigma}d_c\alpha\ \text{ for all }\alpha\in E_0^n\,.
    \end{align*}

\subsection{Locally smooth intrinsic graphs vs. $R$-manifolds on the group $\mathbb H^1\times\R$}\label{subsection: intrinsic graphs vs R mfld on H1timesR}
Let us consider the 4-dimensional connected, simply-connected nilpotent Lie group $\H^1\times\R$ introduced in Subsection \ref{subsection: rumin on HtimesR}.
As already stated in Proposition \ref{prop: complementary subgroups vs Rumin forms}, the splittings of the Lie algebra are in a one-to-one correspondence with the space of simple Rumin forms. Let us use the computations already carried out in Subsection \ref{subsection: rumin on HtimesR} to study whether locally smooth intrinsic graphs in $\H^1\times\R$ are $R$-manifolds.
\subsubsection*{$\operatorname{dim}(\W)=1$} In degree 1, the space of Rumin 1-forms is 3-dimensional and spanned by $\{\theta_1,\theta_2,\theta_3\}$, 3 simple horizontal 1-forms which correspond to the following three choices of complementary subgroups for smooth intrinsic $(\W,\V)$-graphs:
\begin{itemize}
    \item $(\W,\V)$ with $\exp(\W)=\operatorname{span}_\R\{X_1\}$ and $\exp(\V)=\operatorname{span}_\R\{X_2,X_3,X_4\}$;
    \item $(\W,\V)$ with $\exp(\W)=\operatorname{span}_\R\{X_2\}$ and $\exp(\V)=\operatorname{span}_\R\{X_1,X_3,X_4\}$;
    \item $(\W,\V)$ with $\exp(\W)=\operatorname{span}_\R\{X_3\}$ and $\exp(\V)=\operatorname{span}_\R\{X_1,X_2,X_4\}$.
\end{itemize}
In each case, $\operatorname{deg}(\W)=1$, while $\operatorname{Im}\delta_0\cap\Omega^1(\H^1\times\R)=\Omega^{2,-1}(\H^1\times\R)$, and so by Corollary \ref{cor: integrating forms on smooth intrinsic graphs} we get that for any 1-dimensional smooth intrinsic graph  
\begin{align*}
    \int_{\operatorname{graph}(\phi)}\eta=0\ \text{ for any }\eta\in\operatorname{Im}\delta_0\cap\Omega^1(\H^1\times\R)
\end{align*}
by Corollary \ref{cor: integrating forms on smooth intrinsic graphs}, since $w(\eta)=2>1=\operatorname{deg}(\operatorname{graph}(\phi))$.
Therefore, every 1-dimensional locally smooth intrinsic graph is an $R$-manifold.
\subsubsection*{$\operatorname{dim}(\W)=3$} Using Hodge-$\star$ duality considerations, The case of 3-dimensional smooth intrinsic graphs is complementary to the 1-dimensional one, so that we get the following three choices of complementary subgroups for $(\W,\V)$-graphs:
\begin{itemize}
    \item $(\W,\V)$ with $\exp(\W)=\operatorname{span}_\R\{X_1,X_2,X_4\}$ and $\exp(\V)=\operatorname{span}_\R\{X_3\}$;
    \item $(\W,\V)$ with $\exp(\W)=\operatorname{span}_\R\{X_1,X_3,X_4\}$ and $\exp(\V)=\operatorname{span}_\R\{X_2\}$;
    \item $(\W,\V)$ with $\exp(\W)=\operatorname{span}_\R\{X_2,X_3,X_4\}$ and $\exp(\V)=\operatorname{span}_\R\{X_1\}$.
\end{itemize}
Interestingly, $\operatorname{Im}\delta_0\cap\Omega^3(\H^1\times\R)=0$ so 3-dimensional smooth intrinsic graphs are trivially $R$-manifolds for any choice of complementary subgroups.
\subsubsection*{$\operatorname{dim}(\W)=2$}
The space of Rumin 2-forms is spanned by the following simple forms $\{\theta_1\wedge\theta_3,\theta_2\wedge\theta_3,\theta_1\wedge\theta_4,\theta_2\wedge\theta_4\}\subset\Omega^{2,0}(\H^1\times\R)\oplus\Omega^{3,-1}(\H^1\times\R)$, which correspond to the following four choices of complementary subgroups for $(\W,\V)$-graphs:
\begin{enumerate}
    \item $(\W,\V)$ with $\exp(\W)=\operatorname{span}_\R\{X_1,X_3\}$ and $\exp(\V)=\operatorname{span}_\R\{X_2,X_4\}$;
    \item $(\W,\V)$ with $\exp(\W)=\operatorname{span}_\R\{X_2,X_3\}$ and $\exp(\V)=\operatorname{span}_\R\{X_1,X_4\}$;
    \item $(\W,\V)$ with $\exp(\W)=\operatorname{span}_\R\{X_1,X_4\}$ and $\exp(\V)=\operatorname{span}_\R\{X_2,X_3\}$;
    \item $(\W,\V)$ with $\exp(\W)=\operatorname{span}_\R\{X_2,X_4\}$ and $\exp(\V)=\operatorname{span}_\R\{X_1,X_3\}$\,.
\end{enumerate}

Note that for the first two choices of complementary subgroups we have that $\operatorname{deg}(\W)=2$, while for the final two $\operatorname{deg}(\W)=3$. Moreover, since $\operatorname{Im}\delta_0\cap\Omega^2(\H^1\times\R)
\subset\Omega^{3,-1}(\H^1\times\R)$, directly from Corollary \ref{cor: integrating forms on smooth intrinsic graphs}, we get that any locally smooth intrinsic graph of degree 2, i.e. constructed using the complementary subgroups (1) or (2), is an $R$-manifold. The same reasoning does not apply to locally smooth intrinsic graphs of degree 3 constructed using the complementary subgroups (3) and (4). It is not hard to construct examples of smooth intrinsic graphs of constant degree 3 that are \textit{not} $R$-manifolds.

\begin{remark}\label{remark: not an R manifold} Let us consider the pair of complementary subgroups $(\W,\V)$ in (3) as above. We are interested in showing that one can find an associated smooth intrinsic graph $\Phi(w)=w\phi(w)$ which \textit{is not} an $R$-manifold. For this claim to hold, it suffices to find a smooth and continuously intrinsically differentiable function $\phi\colon U\to\V$,
where $U\subset\W$ is any open bounded set, such that
\begin{align*}
    \int_{\Phi(U)}\theta_3\wedge\theta_4\neq 0\ \text{ since }\operatorname{Im}\delta_0\cap\Omega^2(\H^1\times\R)=\operatorname{span}_{C^\infty(\H^1\times\R)}\{\theta_3\wedge\theta_4\}\,.
\end{align*}
By the specific choice (3) of complementary subgroups, we get that
\begin{align*}
    \Phi(w)=w\phi(w)=\big(w_1,0,0,w_4\big)\,\big(0,\phi_2(w),\phi_3(w),0\big)=\big(w_1,\phi_2(w),\phi_3(w),w_4+w_1\phi_2(w)/2\big)\,.
\end{align*}
Then, using the explicit expression of the left-invariant 1-forms $\theta_i$ in Subsection \ref{subsection: rumin on HtimesR}, we get that
\begin{align*}
    \int_{\Phi(U)}\theta_3\wedge\theta_4=&\int_{U}\Phi^\ast\big(\theta_3\wedge\theta_4\big)=\int_U\Phi^\ast\big(dx_3\wedge dx_4+\tfrac{x_1}{2}dx_2\wedge dx_3-\tfrac{x_2}{2}dx_1\wedge dx_3\big)\\=&\int_U\big(d\Phi_3\wedge d\Phi_4+\tfrac{\Phi_1}{2}d\Phi_2\wedge d\Phi_3-\tfrac{\Phi_2}{2}d\Phi_1\wedge d\Phi_3\big)=\int_U\big(\partial_1\phi_3(w)-{\phi_2(w)}{}\partial_4\phi_3(w)\big)dx_1\wedge dx_4\,.
\end{align*}
A simple example of a smooth intrinsically differentiable map $\phi\colon U\subset\W\to\V$ for which the integral $\int_{\Phi(U)}\theta_3\wedge\theta_4$ does not vanish is given by $\phi(w)=(0,w_1)$, namely $\phi_2(w)\equiv 0$ and $\phi_3(w_1,w_4)=w_1$, so that $\partial_1\phi_3(w)-\phi_2(w)\partial_4\phi_3(w)=1$. Thus, the integral does not vanish. 
The continuous intrinsic differentiability of $\phi\colon U\to\V$ could be checked directly, or observing that $\Sigma=\Phi(U)$ is also a non-critical level set of a continuously Pansu differentiable function and then apply \cite[Theorem~4.3.7]{CorniPhD}. We observe that
\[
\operatorname{graph}(\phi)=
f^{-1}(0)=\{(w_1,0,w_1,w_4):(w_1,w_4)\in U\}\ , \text{ where }\ f(w_1,w_2,w_3,w_4)=(w_2,w_3-w_1)
\]
and $f:E\to\V\approx\R^2$, with $E=\{(w_1,w_2,w_3,w_4): (w_1,w_4)\in U, w_2,w_3\in\R\}$.
\end{remark}

The previous remark shows that 2-dimensional, smooth intrinsic graphs of degree 3 in $\H^1\times\R$ are not $R$-manifolds in general. However, unlike the case of $\H^n$, we are not able to leverage extra properties of the Rumin complex. In fact, we will see that Stokes' theorem for the Rumin complex will not hold on them.

\begin{proposition}[Stokes' theorem for the Rumin complex on locally smooth intrinsic graphs in $\H^1\times\R$]\label{prop: stokes intrinsic graphs H1timesR}
    Let $\alpha\in\ker\Box_0\cap\Omega^{p,k-p}(\H^1\times\R)$ be a Rumin $k$-form with $k\ge 1$ and let $\Sigma$ be a $(k+1)$-dimensional locally smooth intrinsic graph with boundary $\partial\Sigma$, then
    \begin{itemize}
        \item [$k=1$:] if $\operatorname{deg}(\Sigma)=2$, then $\int_{\partial\Sigma}\alpha=\int_\Sigma d_c\alpha$;
        \item [$k=1$:] if $\operatorname{deg}(\Sigma)=3$, then $\int_{\partial\Sigma}\alpha\neq\int_\Sigma d_c\alpha$;
        \item [$k=2$:] if $\operatorname{deg}(\partial\Sigma)=p=w(\alpha)$, then $\int_{\partial\Sigma}\alpha=\int_\Sigma d_c\alpha$;
        \item [$k=2$:] if $\operatorname{deg}(\partial\Sigma)<p=w(\alpha)$, then $\int_{\partial\Sigma}\alpha\neq\int_\Sigma d_c\alpha$.
        \end{itemize}
\end{proposition}
\begin{proof}
    Let us first consider the case of $\alpha\in\ker\Box_0\cap\Omega^1(\H^1\times\R)$. We have already shown that 1-dimensional locally smooth intrinsic graphs and 2-dimensional locally smooth intrinsic graphs of degree 2 are $R$-manifolds, hence the first claim directly follows from Theorem \ref{thm: stokes on R manifolds}.

    Let us consider the case where $\partial\Sigma$ is a 1-dimensional smooth intrinsic graph, while $\Sigma$ is a 2-dimensional locally smooth intrinsic graph of degree 3. Then, taking an arbitrary Rumin 1-form $\alpha=f_1\theta_1+f_2\theta_2+f_3\theta_3$ gives
    \begin{align*}
        d\Pi_E\alpha=&d\big(\alpha-d_0^{-1}d_1\alpha\big)=d\big(\alpha+(X_1f_2-X_2f_1)\theta_4\big)=d_c\alpha+[X_3(X_1f_2-X_2f_1)-X_4f_3]\theta_3\wedge\theta_4\,.
    \end{align*}
    Since in general $\int_\Sigma\theta_3\wedge\theta_4\neq 0$ (see Remark \ref{remark: not an R manifold}), we have that
    \begin{align*}
        \int_{\partial\Sigma}\alpha=\int_{\partial\Sigma}\Pi_E\alpha=\int_\Sigma d\Pi_E\alpha=\int_{\Sigma}d_c\alpha+[X_3(X_1f_2-X_2f_1)-X_4f_3]\theta_3\wedge\theta_4\neq\int_\Sigma d_c\alpha
    \end{align*}
    for an appropriate choice of the horizontal 1-form $\alpha$ for which $X_3(X_1f_2-X_2f_1)-X_4f_3\neq 0$.

    The case of $\alpha\in E_0^2$ and $\operatorname{deg}(\partial\Sigma)=w(\alpha)$ covers exactly two cases: when $w(\alpha)=2$ and $\partial\Sigma$ is a locally smooth intrinsic graph of degree 2, and when $w(\alpha)=3$ and $\partial\Sigma$ is a locally smooth intrinsic graph of degree 3. In the first case, since both $\partial\Sigma$ and $\Sigma$ are $R$-manifold, the claim follows directly from Theorem \ref{thm: stokes on R manifolds}. On the other hand, if $w(\alpha)=3$, we have that $\partial\Sigma$ is not an $R$-manifold, however since $3$ is the maximal possible weight of forms in degree 2, we have that property (1) of Proposition \ref{prop: when R manifold is not needed} applies, so that Stokes' theorem holds true since $\Pi_E\alpha=\alpha$.

    Finally, if $w(\alpha)=2$ while $\partial\Sigma$ is a locally smooth intrinsic graph of degree 3 (which is the only possibility for $w(\alpha)<\operatorname{deg}(\partial\Sigma)$ to be true), then $\Pi_E\alpha=\alpha+\operatorname{Im}\delta_0\neq \alpha$ and since in general $\int_{\partial\Sigma}\theta_3\wedge\theta_4\neq 0$ we have that
    \begin{align*}
        \int_{\partial\Sigma}\alpha\neq\int_{\partial\Sigma}\Pi_E\alpha=\int_\Sigma d\Pi_E\alpha=\int_\Sigma d_c\alpha
    \end{align*}
    the last equality following from the fact that $\Sigma$ is an $R$-manifold.
\end{proof}

To recap, in this section we have shown that in general locally smooth intrinsic graphs are not $R$-manifolds. In some cases, (for example in the case of Heisenberg groups) one can rely on some extra properties of the Rumin complex  to obtain Stokes' theorem on locally smooth intrinsic graphs $\Sigma$ with boundary $\partial\Sigma$, i.e. $\int_\Sigma d_c\alpha=\int_{\partial\Sigma}\alpha$. In the case where such properties do not hold, however, one is not able to recover the same formula.

\section{Stokes' Theorem for Spectral Complexes}

In this section, similarly to what we have done for the Rumin complex, we study the validity of Stokes' theorem for the spectral complexes $\{(E_{j,l}^{\bullet,\bullet},\Delta_j)\}_{j\in I_{\bullet,\bullet}}$ introduced in \cite{tripaldi2026spectralcomplexestruncatedmulticomplexes}. These subcomplexes arise from the de Rham complex when the latter is viewed as a truncated multicomplex (see Proposition \ref{prop: de Rham is multicomplex}). 

We begin by describing the main structural properties of these subcomplexes via the graded $\R$-modules $Z_r^{p,\bullet}$ and $B_r^{p,\bullet}$, introduced in \cite{livernet2020spectral}. These subspaces of forms represent an alternative method used to describe the quotient spaces $E_r^{p,\bullet}$ arising at each page of the spectral sequence associated with the filtration given by homogeneous weight. We then show that these spaces admit an equivalent formulation in terms of Rumin forms and the Rumin differential $d_c$.
Finally, once the main properties are covered, the validity of Stokes' theorem on smooth intrinsic graphs for the spectral complexes $\{(E_{j,l}^{\bullet,\bullet},\Delta_j)\}_{j\in I_{\bullet,\bullet}}$ follows as a direct consequence of Corollary \ref{cor: integrating forms on smooth intrinsic graphs}.

Let us quickly recall the definitions of the modules $Z_r^{p,\bullet}$ and $B_r^{p,\bullet}$, together with the induced differentials $\Delta_r\colon Z_r^{p,\bullet}/B_r^{p,\bullet}\to Z_r^{p+r,\bullet}/B_r^{p+r,\bullet}$. Here we adopt the indexing and notation consistent with the bidegree convention used in Definition~\ref{def: truncated multicomplex}. For a more thorough presentation of such submodules we refer to \cite{tripaldi2026spectralcomplexestruncatedmulticomplexes}.

\begin{definition}[Definition 2.6 in \cite{livernet2020spectral}]\label{Z and B defined} Let $\alpha\in\Omega^{p,\bullet}(\G)$ and let $r\ge 1$. We define the graded $\R$-submodules $Z_r^{p,\bullet}$ and $B_r^{p,\bullet}$ of $\ker d_0\cap\Omega^{p,\bullet}(\G)$ as follows.
\begin{align*}
    \alpha\in Z_r^{p,\bullet}\ \Longleftrightarrow&\ \text{for }1\le j\le r-1\,,\text{ there exists }z_{p+j}\in\Omega^{p+j,\bullet}(\G)\text{ such that}\\&\ d_0\alpha=0\text{ and }d_n\alpha=\sum_{i=0}^{n-1}d_iz_{p+n-i}\text{ for all }1\le n\le r-1\,.\\
    \alpha\in B_r^{p,\bullet}\ \Longleftrightarrow&\ \text{for }0\le k\le r-1\text{ there exists }c_{p-k}\in\Omega^{p-k,\bullet}(\G)\text{ such that}\\&\ \alpha=\sum_{k=0}^{r-1}d_kc_{p-k}\text{ and }0=\sum_{k=l}^{r-1}d_{k-l}c_{p-k}\text{ for }1\le l\le r-1\,.
\end{align*}
\end{definition}

\begin{proposition}[Theorem 2.10 in \cite{livernet2020spectral}]\label{prop: expression Delta spectral} The $r^{th}$ differential of the spectral sequence corresponds to the map
\begin{align*}
    \Delta_r\colon Z_r^{p,\bullet}/B_r^{p,\bullet}\longrightarrow Z_r^{p+r,\bullet}/B_r^{p+r,\bullet}\ ,\ \Delta_r\big([x]\big)=\bigg[d_rx-\sum_{i=1}^{r-1}d_iz_{p+r-i}\bigg]
\end{align*}
where $x\in Z_r^{p,\bullet}$ and the family $\{z_{p+j}\}_{1\le j\le r-1}$ is chosen so as to satisfy the relations of Definition \ref{Z and B defined}.    
\end{proposition}
\begin{remark}
    It easily follows from the definitions of both $Z_r^{\bullet,\bullet}$ and $B_r^{\bullet,\bullet}$ that 
    \begin{align*}
        \alpha\in B_r^{p,\bullet}\Longleftrightarrow \exists\, c_{p-r+1}\in Z_{r-1}^{p-r+1,\bullet} \text{ such that }d(c_{p-r+1}+c_{p-r+2}+\cdots+c_p)=\alpha+\Omega^{\ge p+1,\bullet}(\G)
    \end{align*}
    where the forms $c_{p-i}\in\Omega^{p-i,\bullet}(\G)$ for $i=0,\ldots,r-2$ satisfy the equations in Definition \ref{Z and B defined}, and $\Omega^{\ge p+1,\bullet}(\G)$ denotes all the forms of weight at least $p+1$.

    This characterisation will play a central role throughout this section and the next one, in particular in the analysis of the validity of Stokes' theorem. With a slight abuse of notation, we summarise it by writing
    \begin{align*}
        B_r^{p,\bullet}=dZ_{r-1}^{p-r+1,\bullet}+\Omega^{\ge p+1,\bullet}(\G)\,.
    \end{align*}
    Similarly, for any $\alpha\in Z_r^{p,\bullet}$, one has
    \begin{align*}
        \Delta_r\alpha=d(\alpha-z_{p+1}-\cdots-z_{p+r-1})+B_r^{p+r,\bullet}+\Omega^{\ge p+r+1,\bullet}(\G)\,,
    \end{align*}
    where the family $\{z_{p+j}\}_{1\le j\le r-1}$ is as in Definition \ref{Z and B defined}.
\end{remark}
\begin{remark}
    Using Definition \ref{Z and B defined}, one readily verifies that for $r=1$,
    \begin{align*}
        &\alpha\in Z_1^{p,\bullet}\ \Longleftrightarrow\ d_0\alpha=0\ \text{ and }\ \alpha\in B_1^{p,\bullet}\ \Longleftrightarrow\ \exists\, \beta\in\Omega^{p,\bullet}(\G)\ \text{ such that }d_0\beta=\alpha \,.
    \end{align*}
    In particular, $$Z_1^{p,\bullet}=\ker d_0\cap\Omega^{p,\bullet}(\G)\ \text{ and }\ B_1^{p,\bullet}=\operatorname{Im}d_0\cap\Omega^{p,\bullet}(\G)$$
    so there is a natural isomorphism between the space of Rumin forms $E_0^\bullet$ and the first page of the spectral sequence, $$E_0^\bullet\cong\bigoplus_{p=0}^QZ_1^{p,\bullet}/B_1^{p,\bullet}\,.$$ Moreover, using the scalar product introduced in Subsection \ref{subsection: introducing a scalar product}, we can express
    \begin{align*}
        E_0^k=\bigoplus_{p=0}^Q\,Z_1^{p,k-1}\cap\big(B_1^{p,k-p}\big)^\perp\,.
    \end{align*}
    A classical property underlying the  construction of spectral complexes is the chain of inclusions
    \begin{align*}
        B_1^{p,\bullet}\subseteq B_2^{p,\bullet}\subseteq\cdots\subseteq B_\infty^{p,\bullet}\subseteq Z_\infty^{p,\bullet}\subseteq\cdots\subseteq Z_2^{p,\bullet}\subseteq Z_1^{p,\bullet}\ \text{ for every }p\in\N\,.
    \end{align*}
    Notice that, when the multicomplex is given by the de Rham complex of a positively graded Lie group $(\Omega^\bullet(\G),d=d_0+d_{w_1}+\cdots+d_{w_s})$, the associated spectral sequence collapses after finitely many steps, i.e. there is exists a positive integer $M\le Q=w(\operatorname{vol})$ such that $B_\infty^{p,k-p}=B_M^{p,k-p}$ and $Z_\infty^{p,k-p}=Z_M^{p,k-p}$ for every $p=0,\ldots,Q$ and every $k=0,\ldots,\operatorname{dim}(\G)$. Moreover, for every $k=1,\ldots,\operatorname{dim}(\G)$ we also have that $B_\infty^{p,k-p}=Z_\infty^{p,k-p}$ at each weight.
\end{remark}
As shown in \cite{tripaldi2026spectralcomplexestruncatedmulticomplexes}, the spaces $Z_r^{p,\bullet}$ and $B_r^{p,\bullet}$, as well as the differentials $\Delta_r$, admit explicit descriptions in terms of Rumin forms and the Rumin differential.
\begin{proposition}[Propositions 3.4, 3.6 and 3.8 in \cite{tripaldi2026spectralcomplexestruncatedmulticomplexes}]\label{prop: spectral things in terms of Rumin}
    Let $Z_r^{p,\bullet}, B_r^{p,\bullet} \subset \Omega^{p,\bullet}(\mathbb{G})$ be as in Definition~\ref{Z and B defined}.
    
    One has $Z_1^{p,\bullet}=\ker d_0\cap\Omega^{p,\bullet}(\G)$. For $r\ge 2$, a form $\alpha\in\Omega^{p,\bullet}(\G)$ belongs to $Z_r^{p,\bullet}$ if and only if $\alpha\in\ker d_0\cap\Omega^{p,\bullet}(\G)$ and there exist forms ${\omega}_{p+i}\in\ker\Box_0\cap\Omega^{p+i,\bullet}(\G)$ with $i=1,\ldots,r-2$ such that
    \begin{align*}
        d_c^i\Pi_0\alpha=\sum_{j=1}^{i-1}d_c^{i-j}\omega_{p+j}\ \text{ for each }i=1,\ldots,r-1\,.
    \end{align*}
    On the other hand, $B_1^{p,\bullet}=\operatorname{Im}d_0\cap\Omega^{p,\bullet}(\G)$. For $r\ge 2$, a form $\alpha\in\Omega^{p,\bullet}(\G)$ belongs to $B_r^{p,\bullet}$ if and only if there exist $c_{p-r+1}\in\ker d_0\cap\Omega^{p-r+1,\bullet}(\G)$ as well as $\omega_{p-r+i}\in\ker\Box_0\cap\Omega^{p-r+i,\bullet}(\G)$ for $i=2,\ldots,r-1$ such that
    \begin{align*}
        d_c^i\Pi_0 c_{p-r+1}=\sum_{j=1}^{i-1}d_c^{i-j}\omega_{p-r+1+j}\ \text{ for each }i=1,\ldots,r-2
    \end{align*}
    and
    \begin{align*}
        \Pi_0\alpha=d_c^{r-1}\Pi_0c_{p-r+1}-\sum_{i=1}^{r-2}d_c^{r-1-i}\omega_{p-r+1+i}\,.
    \end{align*}
    Finally, for any $r\ge 1$ and $\alpha\in Z_r^{p,\bullet}$, one has
    \begin{align*}
        \Delta_r\big([\alpha]\big)=\bigg[d_c^r\Pi_0\alpha-\sum_{i=2}^{r-1}d_c^i\omega_{p+r-i}\bigg]\,,
    \end{align*}
    where the $\omega_{p+r-i}\in\ker\Box_0\cap\Omega^{p+r-i,\bullet}(\G)$ are chosen so that $d_c^j\Pi_0\alpha=\sum_{i=1}^{j-1}d_c^i\omega_{p+j-i}$ for $j=1,\ldots,r-1$.
\end{proposition}

As already proved explicitly in Lemma \ref{lem: expressing d_c on E_0}, in general, given an arbitrary choice of bidegree $(p,k)$, the Rumin differential $d_c$ applied to $E_0^k\cap\Omega^{p,k-p}(\G)$ will be given by a sum of several operators. Let us introduce the following notation: for a given choice of bidegree $(p,k)$, there exists $I_{p,k}=\{\nu_1,\ldots,\nu_N\}$ with each $\nu_i\in\N$ and $\nu_1<\cdots<\nu_N$ such that
\begin{align*}
    d_c=\sum_{j\in I_{p,k}}d_c^j\colon E_0^{k}\cap\Omega^{p,k-p}(\G)\longrightarrow\bigoplus_{j\in I_{p,k}}E_0^{k+1}\cap\Omega^{p+j,k+1-p-j}(\G)
\end{align*}
where each $d_c^j$ has bidegree $(j,1-j)$. Equivalently, this can also be expressed by saying that at each page $j\in I_{p,k}$, the differentials $\Delta_j$ are non-trivial maps
\begin{align*}
    \Delta_j\colon Z_j^{p,k-p}/B_j^{p,k-p}\longrightarrow Z_j^{p+j,k+1-p-j}/B_j^{p+j,k+1-p-j}\,.
\end{align*}

As shown explicitly in Lemma~\ref{lem: expressing d_c on E_0}, for a fixed bidegree $(p,k)$, the restriction of the Rumin differential $d_c$ to $E_0^k \cap \Omega^{p,k-p}(\mathbb{G})$
may decompose as a sum of several homogeneous components. More precisely, there exists a finite set
\[
I_{p,k}=\{\nu_1,\ldots,\nu_N\}\subset \mathbb{N},
\qquad \nu_1<\cdots<\nu_N,
\]
such that
\[
d_c
=
\sum_{j\in I_{p,k}} d_c^j
\colon
E_0^{k}\cap\Omega^{p,k-p}(\mathbb{G})
\longrightarrow
\bigoplus_{j\in I_{p,k}}
E_0^{k+1}\cap\Omega^{p+j,k+1-p-j}(\mathbb{G}),
\]
where each component $d_c^j$ has bidegree $(j,1-j)$.

In terms of the spectral sequence, these components correspond to the only possibly nonzero differentials
\[
\Delta_j
\colon
Z_j^{p,k-p}/B_j^{p,k-p}
\longrightarrow
Z_j^{p+j,k+1-p-j}/B_j^{p+j,k+1-p-j}\ ,
\quad j\in I_{p,k}.
\]
If there exists $\alpha \in Z_j^{p,\bullet}$ such that $d_c^j \Pi_0 \alpha \notin B_j^{p+j,\bullet}$ (modulo lower-order terms), then the differential $\Delta_j$ is nonzero.

\begin{proposition}\label{prop: Delta well defined}[Proposition 4.2 in \cite{tripaldi2026spectralcomplexestruncatedmulticomplexes}] Fix a bidegree $(p,k)$ and let $I_{p,k}=\{\nu_1,\ldots,\nu_N\}$. Assume that, for each $j\in I_{p,k}$, the component of bidegree $(j,1-j)$
\[
d_c^j \colon E_0^k \cap \Omega^{p,k-p}(\mathbb{G})
\longrightarrow
E_0^{k+1} \cap \Omega^{p+j,k+1-p-j}(\mathbb{G})
\]
is nonzero. 

Then, for each $j \in I_{p,k}$ and for any $m_1,m_2 \in \mathbb{N}$, the map
\begin{align}\label{the good differentials}
\begin{split}
\Delta_j \colon\;&
Z_j^{p,k-p}/B_{m_1}^{p,k-p}
\longrightarrow
Z_{m_2}^{p+j,k+1-p-j}/B_j^{p+j,k+1-p-j},\\
\Delta_j\big([\alpha]\big)
=\;&
\bigg[
d_j \alpha - \sum_{i=1}^{j-1} d_{j-i} z_{p+i}
\bigg]
=
\bigg[
d_c^j \Pi_0 \alpha - \sum_{i=2}^{j-1} d_c^i \omega_{p+j-i}
\bigg],
\end{split}
\end{align}
is well defined.

    Moreover, if for the same $(p,k)$ we have $j \in I_{p,k}$ and $l \in I_{p+j,k+1}$, i.e.
    \begin{align*}
        &d_c^j\colon E_0^k\cap\Omega^{p,k-p}(\G)\longrightarrow E_0^{k+1}\cap\Omega^{p+j,k+1-j-p}(\G)\text{ and }\\&d_c^l\colon E_0^{k+1}\cap\Omega^{p+j,k+1-p-j}(\G)\longrightarrow E_0^{k+2}\cap\Omega^{p+j+l,k+2-p-j-l}(\G)
    \end{align*}
    are non-trivial maps, then for any $m_1,m_2\in\N$, the composition
    \begin{align*}
        Z_j^{p,k-p}/B_{m_1}^{p,k-p}\xrightarrow[]{\Delta_j}Z_l^{p+j,k+1-p-j}/B_j^{p+j,k+1-p-j}\xrightarrow[]{\Delta_l} Z_{m_2}^{p+l+j,k+2-p-l-j}/B_l^{p+l+j,k+2-p-l-j}
    \end{align*}
    satisfies $\Delta_l\circ \Delta_j=0$.
\end{proposition}

    Notice how in Proposition \ref{prop: Delta well defined} each non-trivial map $\Delta_j$ of bidegree $(j,1-j)$ determines the page of the $Z$ module in the domain, and the $B$ module in the target. More precisely, the corresponding map has the form
    \begin{align*}
        \Delta_j\colon Z_j^{p,\bullet}/B_{m_1}^{p,\bullet}\longrightarrow Z_{m_2}^{p+j,\bullet}/B_j^{p+j,\bullet}\,.
    \end{align*}
    Here $m_1,m_2\in\N$ may be chosen arbitrarily. However, when constructing a complex, $m_1$ is determined by the bidegree $(m_1,1-m_1)$ of the incoming differential $\Delta_{m_1}$, while $m_2$ is determined by the bidegree $(m_2,1-m_2)$ of the outgoing differential $\Delta_{m_2}$. This construction can be carried out in every degree
$k=0,\ldots,\dim(\mathbb{G})$.

In this way, for each nontrivial space $E_1^{p,k-p}$, one obtains as many possible subcomplexes as there are elements $j\in I_{p,k}$.

Since we want to realise these complexes as subspaces of differential forms, rather than as quotient spaces as in \eqref{the good differentials}, we use the scalar product introduced in Subsection~\ref{subsection: introducing a scalar product}.

    \begin{definition}\label{def: spectral complexes}[Definition 4.3 in \cite{tripaldi2026spectralcomplexestruncatedmulticomplexes}]
        Given the de Rham complex $(\Omega^\bullet(\mathbb{G}), d = d_0 + d_{w_1} + \cdots + d_{w_s})$, assume that for some bidegree $(p,k)$ the space $E_1^{p,k-p}$ is nontrivial (equivalently, $\ker \Box_0 \cap \Omega^{p,k-p}(\mathbb{G}) \neq 0$).
        Then, for each $j \in I_{p,k}$ and $l \in I_{p-l,k-1}$, we define
        \begin{align}\label{eq: def E_{j,l}}
            E_{j,l}^{p,k-p}:=Z_j^{p,k-p}\cap\big(B_l^{p,k-p}\big)^\perp\subset \ker\Box_0\cap\Omega^{p,k-p}(\G)\,.
        \end{align}
        Moreover, for each $j\in I_{p,k}$, $l\in I_{p-l,k-1}$, and $i\in I_{p+j,k+1}$
        \begin{align*}
            E_{l,m_1}^{p-l,k-1-p+l}\xrightarrow[]{\Delta_l} E_{j,l}^{p,k-p}\xrightarrow[]{\Delta_j}E_{i,j}^{p+j,k+1-p-j}\xrightarrow[]{\Delta_i}E_{m_2,i}^{p+j+i,k+2-p-j-i}
        \end{align*}
        satisfies $\Delta_j\circ \Delta_l=\Delta_i\circ\Delta_j=0$ for any choice of $m_1,m_2\in\N$.

        Since this construction applies in every degree $k$, it yields a family of subspaces $E_{j,l}^{\bullet,\bullet}$. These subspaces are connected by the operators $\Delta_j$ of bidegree $(j,1-j)$, thus forming a collection of complexes. We refer to this collection as the \textit{spectral complexes} associated with the de Rham complex $(\Omega^\bullet(\G),d)$ of the positively gradable Lie group $\G$ and denote it by 
        \begin{align*}
            \big\{\big(E_{j,l}^{\bullet,\bullet},\Delta_j\big)\big\}_{j\in I_{\bullet,\bullet}}\,.
        \end{align*}
    \end{definition}
    \begin{remark}\label{remark: projection with exact forms} When we write
    \begin{align*}
        \Delta_j\colon E_{j,l}^{p,k-p}\longrightarrow E_{m,j}^{p+j,k+1-p-j}
    \end{align*}
we mean that the operator defined in \eqref{the good differentials} is composed with the projection onto the subspace $E_{m,j}^{p+j,k+1-p-j}$, i.e. onto the orthogonal complement of $B_j^{p+j,k+1-p-j}$.

Unlike in the case of the Rumin complex, where the operator $\Pi_0$ in $d_c=\Pi_0 d \Pi_E$ is an orthogonal projection onto $\ker \Box_0 \cap \Omega^{\bullet,\bullet}(\mathbb{G})$, here the projection is realised by modifying representatives through the addition or subtraction of exact terms. More precisely, given a representative of $\Delta_j([\alpha])$, one replaces it by an equivalent form modulo $B_j^{p+j,k+1-p-j}$ by adding terms of the form
$d(c_{p+1} + \cdots + c_{p+j})$,
where $c_{p+i} \in \Omega^{p+i,k-p-i}(\mathbb{G})$ are chosen as in Definition~\ref{Z and B defined}.
        
    \end{remark}
    
\begin{theorem}[Stokes' theorem for spectral complexes]\label{thm: stokes for spectral complexes} Fix a bidegree $(p,k-1)$ for which the operator $$\Delta_j\colon E_{j,l}^{p,k-1-p}\longrightarrow E_{m,j}^{p+j,k-p-j}$$
is nonzero, for some admissible choice of $l,m\in\mathbb{Z}^+$. Let $\Sigma\subset\mathbb{G}$ be an orientable $k$-dimensional $C^1$ manifold such that $\operatorname{deg}(\Sigma)=p+j$, and assume that $\partial\Sigma$ is also an orientable $C^1$ manifold with $\operatorname{deg}(\partial\Sigma)=p$. Then
\begin{align*}
    \int_{\partial\Sigma}\alpha=\int_\Sigma \Delta_j\alpha\ \text{ for every }\alpha\in E_{j,l}^{p,k-1-p}\,.
\end{align*}
\end{theorem}
\begin{proof}
Let $\alpha\in E_{j,l}^{p,k-1-p}\subset E_0^{k-1}\cap\Omega^{p,k-1-p}(\mathbb G)$.
By Definition~\ref{def: spectral complexes} and formula~\eqref{the good differentials}, we can represent $\Delta_j\alpha$ as
\[
\Delta_j\alpha
=
d(\alpha - z_{p+1} - \cdots - z_{p+j-1})
+
B_j^{p+j,k-p-j}
+
\Omega^{\ge p+j+1,\bullet}(\mathbb G),
\]
where $z_{p+i}\in \Omega^{p+i,k-1-p-i}(\mathbb G)$, $i=1,\ldots,j-1$,
are determined by the condition $\alpha\in Z_j^{p,k-1-p}$. Moreover, using the characterisation
\[
B_j^{p+j,k-p-j}
=
dZ_{j-1}^{p+1,k-p-2}
+
\Omega^{\ge p+j+1,\bullet}(\mathbb G),
\]
we obtain that $\Delta_j\alpha$ differs from $d(\alpha - z_{p+1} - \cdots - z_{p+j-1})$
by terms that either belong to $dZ_{j-1}^{p+1,k-p-2}$ or have weight at least $p+j+1$. Finally, since $\deg(\partial\Sigma)=p$, the forms $z_{p+i}$, as well as those in $Z^{p+1,k-2-p}_{j-1}$, have strictly larger weight and therefore do not contribute to the integral over $\partial\Sigma$. 

 This can be more compactly rephrased using a slight abuse of notation as follows
    \begin{align*}
        \int_{\partial\Sigma}\alpha=&\int_{\partial\Sigma}\big(\alpha-z_{p+1}-\cdots-z_{p+j-1}\big)+Z_{j-1}^{p+1,k-2-p}\\=&\int_\Sigma d\big(\alpha-z_{p+1}-\cdots-z_{p+j-1}\big)+dZ_{j-1}^{p+1,k-2-p}\\=&\int_\Sigma d\big(\alpha-z_{p+1}-\cdots-z_{p+j-1}\big)+dZ_{j-1}^{p+1,k-2-p}+\Omega^{\ge p+j+1,\bullet}(\G)=\int_\Sigma \Delta_j\alpha\,.
    \end{align*}
    and the claim follows directly from Corollary \ref{cor: integrating forms on smooth intrinsic graphs}: the additional terms in the expression of $\Delta_j\alpha$ do not contribute to the integral over $\Sigma$, since
\begin{itemize}
\item the terms in $dZ_{j-1}^{p+1,k-p-2}$ integrate to zero by Stokes' theorem and  $\operatorname{deg}(\partial\Sigma)=p$;
\item the terms in $\Omega^{\ge p+j+1,\bullet}(\mathbb G)$ vanish on $\Sigma$ because $\deg(\Sigma)=p+j$.
\end{itemize}

\end{proof}

\begin{corollary}[Stokes' theorem on locally smooth intrinsic graphs with boundary]
\label{cor: stokes thm for intrinsic graphs}
Let $\Sigma\subset\mathbb{G}$ be an orientable $k$-dimensional locally smooth intrinsic graph with boundary $\partial\Sigma$, then
\[
\int_{\partial\Sigma}\alpha
=
\int_\Sigma \Delta_j \alpha
\quad
\text{for every } \alpha\in E_{j,l}^{p,k-1-p}\,,
\]
where $p=\operatorname{deg}(\Sigma)$, $j=\operatorname{deg}(\Sigma)-\operatorname{deg}(\partial
\Sigma)$, and some $l\in\mathbb{N}$ such that $E_{j,l}^{p,k-1-p}\neq 0$.
\end{corollary}

\begin{proof}
It is enough to show that the assumptions that $\Sigma$ is a locally smooth intrinsic graph with boundary $\partial\Sigma$ ensures the existence of indices $j\in I_{p,k-1}$ with $p=\operatorname{deg}(\partial\Sigma)$ and $l\in\mathbb{Z}^+$ such that
\[
E_{j,l}^{p,k-1-p}\neq 0
\quad\text{and}\quad
\Delta_j \colon E_{j,l}^{p,k-1-p}\longrightarrow E_{m,j}^{p+j,k-p-j}
\]
is nonzero for some $m\in\mathbb{Z}^+$.

Since $\Sigma$ is a locally smooth intrinsic graph of degree $p+j$, there exists at least a point $p\in\Sigma$ such that a neighbourhood of $p$ in $\Sigma$ is a $(\mathbb{W},\mathbb{V})$-graph with $\operatorname{deg}(\mathbb{W})=p+j$. Therefore, by Lemma \ref{lem: subalgebras}, we have $E_0^k \cap \Omega^{p+j,k-p-j}(\mathbb{G}) \neq 0$. Similarly, the condition $\deg(\partial\Sigma)=p$ implies
$E_0^{k-1} \cap \Omega^{p,k-1-p}(\mathbb{G}) \neq 0$. 

Finally, by Lemma \ref{lem: deg boundary less}, we have that $j\ge 1$ necessarily.

By the decomposition of the Rumin differential, this implies that there exists at least one $j\in I_{p,k-1}$ such that the component
\[
d_c^j \colon E_0^{k-1} \cap \Omega^{p,k-1-p}(\mathbb{G})
\longrightarrow
E_0^{k} \cap \Omega^{p+j,k-p-j}(\mathbb{G})
\]
is nonzero. This then implies that there exists a corresponding spectral differential $\Delta_j$ which is nontrivial on a suitable subspace $E_{j,l}^{p,k-1-p}$.

The claim then follows from Theorem~\ref{thm: stokes for spectral complexes}.
\end{proof}

\begin{remark}[Geometric interpretation]
The previous corollary shows that the degree of the intrinsic graph determines which component of the spectral complex governs the Stokes' formula. 

More precisely, the decomposition of the Rumin differential
\[
d_c = \sum_{j\in I_{p,k-1}} d_c^j
\]
splits the Rumin differential into components of different homogeneous weights. The condition $\deg(\partial\Sigma)=p$ and $\deg(\Sigma)=p+j$ selects exactly the component $d_c^j$ that produces a nontrivial contribution when passing from $\partial\Sigma$ to $\Sigma$.

Equivalently, among all the spectral differentials $\{\Delta_j\}_{j\in I_{p,k-1}}$, only those compatible with the degree jump between $\partial\Sigma$ and $\Sigma$ can contribute to the Stokes' formula. In this sense, the geometry of the submanifold determines which subcomplex $(E_{j,l}^{\bullet,\bullet},\Delta_j)$ is relevant.
\end{remark}

\begin{example}[$\mathbb{H}^1\times\mathbb{R}$ and spectral Stokes phenomena]
\label{ex: H1xR spectral}
We consider the $4$-dimensional nilpotent Lie group $\mathbb{H}^1\times\mathbb{R}$. 
Using the explicit description of Rumin forms and complementary subgroups given in Subsection~\ref{subsection: rumin on HtimesR}, we analyse how the spectral Stokes theorem behaves for smooth intrinsic graphs.

The space of Rumin 1-forms coincides with the space of horizontal 1-forms, so every 1-dimensional intrinsic graph must have degree 1. However, the space of Rumin $2$-forms contains components in both
$\Omega^{2,0}(\mathbb{H}^1\times\mathbb{R})
$ and $\Omega^{3,-1}(\mathbb{H}^1\times\mathbb{R})$.
Consequently, the Rumin differential decomposes as $d_c = d_c^1 + d_c^2$,
and both components may contribute nontrivially.

As a result, for $2$-dimensional smooth intrinsic graphs $\Sigma$, the spectral Stokes' theorem may involve either $\Delta_1$ or $\Delta_2$, depending on the choice of complementary subgroup (equivalently, on the corresponding simple Rumin form). In particular, higher-order contributions can arise, reflecting the presence of non-horizontal directions.

This example shows that, even in a low-dimensional setting, different geometric configurations of intrinsic graphs can lead to different spectral complexes used for Stokes' formula. In particular, the degree of the submanifold determines which component $\Delta_j$ appears.
\end{example}
\begin{remark}\label{rmk: spectral stokes no on R manifolds}
    Stokes' theorem does not, in general, hold for the operators
\[
\Delta_j\colon E_{j,l}^{p,k-1-p}\longrightarrow E_{m,j}^{p+j,k-p-j}
\]
on arbitrary $R$-manifolds when $j\ge 2$.

In the case when $j=2$, given $\alpha\in Z_2^{p,k-1-p}\subset \ker d_0\cap\Omega^{p,k-1-p}(\mathbb G)$, there exists $z_{p+1}\in\Omega^{p+1,k-2-p}(\mathbb G)$ such that $d_1\alpha=d_0z_{p+1}$
and therefore
\[
\Delta_2\alpha
=
d(\alpha - z_{p+1})
+
B_2^{p+2,k-p-2}\,.
\]

The condition above determines uniquely the component of $z_{p+1}$ in $\operatorname{Im}\delta_0$, namely $\operatorname{pr}_{\operatorname{Im}\delta_0} z_{p+1} = d_0^{-1} d_1 \alpha$.
Hence one can choose $z_{p+1}$ so that its $\ker d_0$-component vanishes. In that case, if $\partial\Sigma$ is an $R$-manifold, one has
\[
\int_{\partial\Sigma} \alpha
=
\int_{\partial\Sigma} (\alpha - z_{p+1}).
\]
However, the operator $\Delta_2$ is only defined up to addition of elements in
$B_2^{p+2,k-p-2}
=
dZ_1^{p+1,k-2-p}
+
\Omega^{\ge p+3,\bullet}(\mathbb G)$,
and the forms in $Z_1^{p+1,k-2-p}=\ker d_0\cap\Omega^{p+1,k-2-p}(\mathbb G)$ are not, in general, annihilated by integration on $\partial\Sigma$, even if $\partial\Sigma$ is an $R$-manifold. Consequently, one cannot expect
\[
\int_{\partial\Sigma}\alpha = \int_\Sigma \Delta_2\alpha
\]
to hold in general.

In the case when $j\ge 3$, if $\alpha\in Z_j^{p,k-1-p}$ with $j\ge 3$, the characterisation in terms of the Rumin differential (see Proposition~\ref{prop: spectral things in terms of Rumin}) implies the existence of forms $\omega_{p+i}\in \ker\Box_0\cap\Omega^{p+i,k-1-p-i}(\mathbb G)$ with $i\ge 1$,
such that
\[
d_c^1\Pi_0\alpha = 0,
\qquad
d_c^2\Pi_0\alpha = d_c^1 \omega_{p+1}.
\]
In particular, unless $d_c^2\Pi_0\alpha=0$, the form $\omega_{p+1}$ is a nontrivial Rumin form of weight strictly larger than $p$.

As a consequence, even if $\partial\Sigma$ is an $R$-manifold, one cannot remove the contribution of such terms when passing from $\alpha$ to a representative of $\Delta_j\alpha$. Therefore, in general,
\[
\int_{\partial\Sigma}\alpha \neq \int_\Sigma \Delta_j\alpha,
\qquad j\ge 3.
\]
We further remark that the condition $\Sigma$ being an $R$-manifold is never needed when dealing with spectral complexes $\Delta_j\colon E_{j,l}^{p,\bullet}\longrightarrow E_{m,j}^{p+j,\bullet}$, as long as $\operatorname{deg}(\Sigma)=p+j$ since for any $j\in\N$
    \begin{align*}
        d Z_j^{p,\bullet}=B_{j+1}^{p+j,\bullet}\subset Z_1^{p+j,\bullet}=\ker d_0\cap\Omega^{p+j,\bullet}\,.
    \end{align*}

The assumptions $\deg(\Sigma)=p+j$ and $\deg(\partial\Sigma)=p$ in Theorem~\ref{thm: stokes for spectral complexes} are essential. It ensures that the additional terms appearing in the definition of $\Delta_j\alpha$ either vanish on $\Sigma$ (by weight considerations) or integrate to zero (being exact).

Finally, if either $\deg(\partial\Sigma)<p$ or $\deg(\Sigma)<p+j$, then the degree compatibility is lost, and Stokes' theorem for the operator $\Delta_j$ fails, similarly to the phenomenon observed in Proposition~\ref{prop: stokes intrinsic graphs H1timesR}.

\end{remark}

The validity of Stokes' theorem for spectral complexes being so inextricably linked to the degree of the submanifold naturally motivates the introduction of a distinguished class of submanifolds  for which the analogue of Stokes' theorem associated with the spectral complex holds. 
\begin{definition}[Spectral manifolds]\label{def: spectral manifold}
Let $\Sigma\subset\mathbb G$ be a $k$-dimensional orientable $C^1$ manifold and denote by $P_h$ the set of weights of nontrivial Rumin forms $E_0^h$ in degree $h$. We say that $\Sigma$ is a \emph{spectral manifold} if $\deg(\Sigma)\in P_k$.

 We say that $\Sigma$ is a \emph{spectral manifold with boundary} if $\Sigma\subset\mathbb G$ is a $k$-dimensional orientable $C^1$ manifold with boundary $\partial\Sigma$ and
\[
\deg(\Sigma)\in P_k
\qquad\text{and}\qquad
\deg(\partial\Sigma)\in P_{k-1}.
\]

\end{definition}

By Lemma \ref{lem: deg boundary less}, we know that necessarily
$
\deg(\partial\Sigma)<\deg(\Sigma)$, hence if $\Sigma$ is a spectral manifold, there exists a nontrivial spectral complex whose differential connects forms of weight $\deg(\partial\Sigma)$ in degree $k-1$ to forms of weight $\deg(\Sigma)$ in degree $k$. Consequently, Theorem \ref{thm: stokes for spectral complexes} applies to $\Sigma$ and $\partial\Sigma$ through the corresponding differential $\Delta_{j}$, where $j=\operatorname{deg}(\Sigma)-\operatorname{deg}(\partial\Sigma)$.

In this sense, spectral manifolds provide the natural geometric setting where the spectral complex behaves as a genuine boundary operator from the point of view of integration theory (see also Section \ref{section: currents}). This is because the defining condition for a spectral manifold guarantees that the degrees of the manifold and of its boundary are compatible with the weights appearing in the spectral complexes, so that Stokes' theorem can be realised through a suitable spectral differential.

    As already observed in Corollary \ref{cor: stokes thm for intrinsic graphs}, the class of spectral manifolds contains, in particular, locally smooth intrinsic graphs with boundary. However, Remark \ref{rmk: example Rumin forms nonsimple} suggests that this class might be strictly larger. Indeed, one can find subspaces of $\ker\Box_0\cap\Omega^\bullet(\G)$ spanned by non-simple Rumin forms, whereas smooth intrinsic graphs are in one-to-one correspondence with simple Rumin forms. Therefore, spectral manifolds capture geometric configurations that cannot be represented by intrinsic graphs alone.

In the cases where the Rumin differential coincides with some differential appearing in the spectral complexes, we get that Stokes' theorem holds for the Rumin complex by applying directly Theorem \ref{thm: stokes for spectral complexes}. Notice that this is the case whenever we consider the Rumin differential $d_c$ acting between Rumin $(k-1)$-forms of maximal weight and Rumin $k$-forms of minimal weight, i.e. if $N_{k-1}=\max\{w(\alpha)\in\Z^+\mid \alpha\in E_0^{k-1} \}$ and $n_k=\min\{w(\alpha)\in\Z^+\mid \alpha\in E_0^k\}$, then
\begin{align*}
    d_c\colon E_0^{k-1}\cap\Omega^{N_{k-1},k-1-N_{k-1}}(\G)\longrightarrow E_0^{k}\cap\Omega^{n_k,k-n_k}(\G)
\end{align*}
coincides with the differential
\begin{align*}
    \Delta_{n_k-N_{k-1}}\colon Z_1^{N_{k-1},k-1-N_{k-1}}\cap\big(B_m^{N_{k-1},k-1-N_{k-1}}\big)^\perp\longrightarrow Z_l^{n_k,k-n_k}\cap\big(B_{n_k-N_{k-1}}^{n_k,k-n_k}\big)^\perp\ \text{ for any }m,l\in\Z^+\,.
\end{align*}
In some special Carnot groups, such as Heisenberg groups and the 5-dimensional Cartan group, this happens in each degree since for each $k=0,\ldots,\operatorname{dim}(\G)$ there exists only one $p=p(k)$ such that $E_0^k\cap\Omega^{p,k-p}(\G)\neq 0$ (see Remark 4.4 in \cite{tripaldi2026spectralcomplexestruncatedmulticomplexes} for an explicit proof of this fact). In particular, this implies that Proposition \ref{prop: stokes on heisenberg groups} can be viewed as a direct consequence of Corollary \ref{cor: stokes thm for intrinsic graphs} in the case where $\G=\H^n$.

\begin{corollary}[Stokes' theorem for the Rumin complex on locally smooth intrinsic graphs in $\H^n$ and the Cartan group]\label{cor: stokes on Hn and cartan}
    Let $\G$ be either the $(2n+1)$-dimensional Heisenberg group or the 5-dimensional free nilpotent Cartan group. Let $\Sigma\subset\G$ be a $k$-dimensional locally smooth intrinsic graph with boundary $\partial\Sigma$, then
    \begin{align*}
        \int_{\partial \Sigma}\alpha=\int_\Sigma d_c\alpha\ \text{ for any }\alpha\in E_0^{k-1}\,.
    \end{align*}
    
\end{corollary}

\section{Re-expressing spectral complexes in view of Stokes' Theorem}\label{section: spectral Leibniz}
It is possible to reinterpret the spectral complexes $\{(E_{j,l}^{\bullet,\bullet},\Delta_j)\}_{j\in I_{\bullet,\bullet}}$, and in particular the action of the differentials
\[
\Delta_j\colon E_{j,l}^{p,\bullet}\longrightarrow E_{m,j}^{p+j,\bullet},
\]
in terms of Stokes' theorem. The aim of this discussion is to emphasise the naturality of the spectral complexes in the context of positively graded Lie groups. For this reason, unlike in Section~\ref{section: Rumin leibniz}, the considerations below are carried out in full generality and apply to spectral complexes defined on any positively gradable Lie group $\mathbb{G}$.

\medskip

To simplify the notation in the following computations, we introduce a decomposition adapted to the algebraic Hodge Laplacian $\Box_0=d_0\delta_0+\delta_0 d_0$, following the Hodge decomposition described in~\eqref{eq: hodge decompo Box_0}.

\begin{definition}
Let $\alpha\in\Omega^\bullet(\mathbb{G})$. We write
\[
\alpha=\check{\alpha}+\overline{\alpha}+\hat{\alpha},
\]
where
\begin{itemize}
    \item $\check{\alpha}=\operatorname{pr}_{\operatorname{Im}d_0}\alpha = d_0 d_0^{-1}\alpha \in \operatorname{Im} d_0$;
    \item $\hat{\alpha}=\operatorname{pr}_{\operatorname{Im}\delta_0}\alpha = d_0^{-1} d_0 \alpha \in \operatorname{Im} \delta_0$;
    \item $\overline{\alpha}=\Pi_0 \alpha = \alpha - d_0 d_0^{-1}\alpha - d_0^{-1} d_0 \alpha \in \ker \Box_0$.
\end{itemize}
In particular, $\overline{\alpha}$ is the Rumin component of $\alpha$.
When $\alpha\in\Omega^{p,\bullet}(\mathbb{G})$ and $\check{\alpha}=d_0\beta_p$, we will denote by $\beta_p$ a primitive of $\check{\alpha}$, in order to keep track of the weight of the corresponding term.
\end{definition}
To avoid excessively cumbersome notation, we do not carry out the computations in full generality for the differential $\Delta_j$ acting on a form $\alpha\in Z_j^{p,\bullet}$ for arbitrary $j\ge 1$. Instead, we focus on the cases $j=1,2,3$, which already illustrate the underlying mechanism and the role of the Leibniz rule in the general case.

Recall that, on a positively gradable Lie group $\mathbb{G}$, the exterior differential can be written as
\[
d = d_0 + d_1 + \cdots + d_{w_s},
\]
where $d_j$ is the zero operator unless $j=w_i$ for some $i=1,\ldots,s$ (see Lemma~\ref{lem: diving exterior derivative}).

\medskip
$j=1$. 
In this case, we consider $\alpha \in Z_1^{p,\bullet} = \ker d_0 \cap \Omega^{p,\bullet}(\mathbb{G})$. Using the Hodge decomposition, we write
\[
\alpha = \check{\alpha} + \overline{\alpha} = d_0 \beta_p + \overline{\alpha}\,,
\]
so that
\begin{align*}
    d\alpha=&d_0\alpha+d_1\alpha+\cdots+d_{w_s}\alpha=d_1(d_0\beta_p+\overline{\alpha})+\Omega^{\ge p+2,\bullet}(\G)=d_1\overline{\alpha}-d_0d_1\beta_p+\Omega^{\ge p+2,\bullet}\\=&
    d_1\overline{\alpha}-d_0d_0^{-1}d_1\overline{\alpha}+d_0\big(d_0^{-1}d_1\overline{\alpha}-d_1\beta_p\big)+\Omega^{\ge p+2,\bullet}(\G)=d_c^1\overline{\alpha}+\underbrace{d_0\big(d_0^{-1}d_1\overline{\alpha}-d_1\beta_p\big)}_{\in B_1^{p+1,\bullet}}+\Omega^{\ge p+2,\bullet}(\G)\\=&d_c^1\overline{\alpha}+d\big(d_0^{-1}d_1\overline{\alpha}-d_1\beta_p\big)-(d-d_0)\big(d_0^{-1}d_1\overline{\alpha}-d_1\beta_p\big)+\Omega^{\ge p+2,\bullet}(\G)\\=&d_c^1\overline{\alpha}+d\big(d_0^{-1}d_1\overline{\alpha}-d_1\beta_p\big)+\Omega^{\ge p+2,\bullet}(\G)\,.
\end{align*} 
Rearranging the formula, we obtain
\begin{align*}
    d\big(\alpha-d_0^{-1}d_1\overline{\alpha}+d_1\beta_p\big)=d_c^1\overline{\alpha}+\Omega^{\ge p+2,\bullet}(\G)
\end{align*}
In other words, the differential
\[
\Delta_1 \colon Z_1^{p,\bullet} \longrightarrow E_{m,1}^{p+1,\bullet}
\]
coincides with the exterior derivative $d$, up to an exact form and terms of weight $\ge p+2$.

Therefore, for any oriented $C^1$ manifold $\Sigma$ with $\deg(\Sigma)=p+1$ and boundary $\partial\Sigma$ with $\deg(\partial\Sigma)=p$, the spectral Stokes' theorem reduces to the classical Stokes' theorem for $d$.

\medskip
{$j=2$}.
In this case, we have
\begin{align*}
    \alpha\in Z_2^{p,\bullet}\  \longleftrightarrow\ d_0\alpha=0\ \text{ and there exists a }z_{p+1}\in\Omega^{p+1,\bullet}(\G)\text{ such that }d_1\alpha=d_0z_{p+1}\,.
\end{align*}
Notice that the second condition, univocally determines $\hat{z}_{p+1}=d_0^{-1}d_0z_{p+1}$, but we have no control over $\operatorname{pr}_{\ker d_0}z_{p+1}$. This is because
\begin{align*}
    d_1\alpha=d_1(d_0\beta_p+\overline{\alpha})=-d_0d_1\beta_p+d_1\overline{\alpha}-d_0d_0^{-1}d_1\overline{\alpha}+d_0d_0^{-1}d_1\overline{\alpha}=d_c^1\overline{\alpha}+d_0d_0^{-1}d_1\overline{\alpha}-d_0d_1\beta_p=d_0z_{p+1}
\end{align*}
so on one hand this implies that $d_c^1\overline{\alpha}=0$, but also that, $\hat{z}_{p+1}=d_0^{-1}d_1\overline{\alpha}-d_1\beta_p$. Therefore $d_1\alpha=d_0z_{p+1}$ as long as $\hat{z}_{p+1}=d_0^{-1}d_1\overline{\alpha}-d_1\beta_p$, for any $\check{z}_{p+1}+\overline{z}_{p+1}\in\ker d_0\cap\Omega^{p+1,\bullet}(\G)$. Then
\begin{align*}
    d\alpha=&d_0\alpha+d_1\alpha+d_2\alpha+\Omega^{\ge p+3,\bullet}(\G)=d_0\hat{z}_{p+1}+d_2\alpha+\Omega^{\ge p+3,\bullet}(\G)\\=&d_0(\check{z}_{p+1}+\overline{z}_{p+1}+\hat{z}_{p+1})+d_2\alpha+\Omega^{\ge p+3,\bullet}(\G)\\=&dz_{p+1}-(d-d_0)(\check{z}_{p+1}+\overline{z}_{p+1}+\hat{z}_{p+1})+d_2\alpha+\Omega^{\ge p+3,\bullet}(\G)\\
    =&dz_{p+1}-d_1(\check{z}_{p+1}+\overline{z}_{p+1})-d_1(d_0^{-1}d_1\overline{\alpha}-d_1\beta_p)+d_2d_0\beta_p+d_2\overline{\alpha}+\Omega^{\ge p+3,\bullet}(\G\\
    =&dz_{p+1}+d_c^2\overline{\alpha}-\underbrace{d_1(\check{z}_{p+1}+\overline{z}_{p+1})+d_0\big(d_0^{-1}\partial_2\overline{\alpha}-d_2\beta_p\big)}_{\in B_2^{p+2,\bullet}}+\Omega^{\ge p+3,\bullet}(\G)\\=&dz_{p+1}+d_c^2\overline{\alpha}-d\big[\check{z}_{p+1}+\overline{z}_{p+1}-d_0^{-1}\partial_2\overline{\alpha}+d_2\beta_p\big]+(d-d_1-d_0)(\check{z}_{p+1}+\overline{z}_{p+1})+\\&+(d-d_0)(-d_0^{-1}\partial_2\overline{\alpha}+d_2\beta_p)+\Omega^{\ge p+3,\bullet}(\G)\\=&dz_{p+1}+d_c^2\overline{\alpha}-d\big[\check{z}_{p+1}+\overline{z}_{p+1}-d_0^{-1}\partial_2\overline{\alpha}+d_2\beta_p\big]+(d-d_1-d_0)(\check{z}_{p+1}+\overline{z}_{p+1})+\Omega^{\ge p+3,\bullet}(\G)
\end{align*}
and so by rearranging we get
\begin{align*}
    d(\alpha-\hat{z}_{p+1}+d_0^{-1}\partial_2\overline{\alpha}-d_2\beta_p)=d_c^2\overline{\alpha}+\Omega^{\ge p+3,\bullet}(\G)\,.
\end{align*}
Note that in this case $d_1(\check{z}_{p+1}+\overline{z}_{p+1})+d_0\big(d_0^{-1}\partial_2\overline{\alpha}-d_2\beta_p\big)$ belongs to $B_2^{p+2,\bullet}=dZ_1^{p+1,\bullet}+\Omega^{\ge p+3,\bullet}(\G)$ simply because $\check{z}_{p+1}+\overline{z}_{p+1}\in\ker d_0\cap\Omega^{p+1,\bullet}(\G)=Z_1^{p+1,\bullet}$.

Furthermore, let us highlight the fact that in general for a given $\alpha\in Z_2^{p,\bullet}$, we won't have that $d_c^2\overline{\alpha}$ belongs to $B_3^{p+2,\bullet}\cap (B_2^{p+2,\bullet})^\perp$, that is
\begin{align*}
    \Delta_2\alpha=d_c^2\overline{\alpha}+\xi_{p+2}\ \text{ for some }\xi_{p+2}\in B_2^{p+2,\bullet}=dZ_1^{p+1,\bullet}+\Omega^{\ge p+3,\bullet}(\G)\,,
\end{align*}
so one can simply apply again the Leibniz rule just as we did above, since $\xi_{p+2}=d\theta_{p+1}+\Omega^{\ge p+3,\bullet}(\G)$ for some $\theta_{p+1}\in Z_1^{p+1,\bullet}$. We then finally obtain the following simplified expression
\begin{align*}
     d(\alpha-\hat{z}_{p+1}+d_0^{-1}\partial_2\overline{\alpha}-d_2\beta_p-\theta_{p+1})=\Delta_2\alpha+\Omega^{\ge p+3,\bullet}(\G)\,.
\end{align*}

In other words, the differential $\Delta_2\colon Z_2^{p,\bullet}\longrightarrow E_{m,2}^{p+2,\bullet}$ is the same as the exterior derivative $d$, up to an exact form exact form $d\eta$ with $\eta\in\Omega^{p+1,\bullet}(\G)\oplus\Omega^{p+2,\bullet}(\G)$ and forms of weight $\ge p+3$. When taking the integral over a $C^1$ manifold $\Sigma$ for which $\operatorname{deg}(\Sigma)=p+2$ whose boundary is a $C^1$ manifold $\partial\Sigma$ for which $\operatorname{deg}(\partial\Sigma)=p$, we get exactly the Stokes' theorem of Theorem \ref{thm: stokes for spectral complexes} directly from the classical Stokes' theorem together with the Leibniz rule.

\medskip
{$j=3$}. In this case, we have
\begin{align*}
    \alpha\in Z_3^{p,\bullet}\longleftrightarrow &d_0\alpha=0\text{ and there exist }z_{p+i}\in\Omega^{p+i,\bullet}(\G)\text{ with }i=1,2\text{ such that }\\
    &d_1\alpha=d_0z_{p+1}\text{ and }d_2\alpha=d_1z_{p+1}+d_0z_{p+2}\,.
\end{align*}
From the previous steps, we already know that
\begin{itemize}
    \item $\alpha=d_0\beta_p+\overline{\alpha}$;
    \item $d_c^1\overline{\alpha}=0$;
    \item $z_{p+1}=\check{z}_{p+1}+\overline{z}_{p+1}+\hat{z}_{p+1}$ with $\hat{z}_{p+1}=d_0^{-1}d_1\overline{\alpha}-d_1\beta_p$ and $\check{z}_{p+1}=d_0\beta_{p+1}$.
\end{itemize}
The additional condition then reads as
\begin{align*}
    d_2\alpha=d_2(d_0\beta_p+\overline{\alpha})=-d_1^2\beta_p-d_0d_2\beta_p+d_2\overline{\alpha}=d_1(d_0\beta_{p+1}+\overline{z}_{p+1}+d_0^{-1}d_1\overline{\alpha}-d_1\beta_p)+d_0z_{p+2}\,.
\end{align*}
By re-arranging, this gives
\begin{align*}
    &\partial_2\overline{\alpha}-\partial_1\overline{z}_{p+1}=d_0\big(z_{p+2}+d_2\beta_p-d_1\beta_{p+1}\big)\\&\longleftrightarrow d_c^2\overline{\alpha}-d_c^1\overline{z}_{p+1}=d_0(z_{p+2}+d_2\beta_p-d_1\beta_{p+1}-d_0^{-1}\partial_2\overline{\alpha}+d_0^{-1}\partial_1\overline{z}_{p+1})
\end{align*}
which implies that $d_c^2\overline{\alpha}=d_c^1\overline{z}_{p+1}$ but also that $\hat{z}_{p+2}=d_1\beta_{p+1}-d_2\beta_p+d_0^{-1}\partial_2\overline{\alpha}-d_0^{-1}\partial_1\overline{z}_{p+1}$. This also means that for any other choices of $\check{\zeta}_{p+1}+\overline{\zeta}_{p+1}\in\ker d_0\cap\ker d_c^1\cap\Omega^{p+1,\bullet}(\G)$, we have that ${y}_{p+1}:=\check{z}_{p+1}+\check{\zeta}_{p+1}+\overline{z}_{p+1}+\overline{\zeta}_{p+1}+\hat{z}_{p+1}$ also satisfies the equations $d_1\alpha=d_0{y}_{p+1}$ and $d_2\alpha=d_1y_{p+1}+d_0y_{p+2}$ with $\hat{y}_{p+2}=\hat{z}_{p+2}+d_1\gamma_{p+1}-d_0^{-1}\partial_1\overline{\zeta}_{p+1}$, where $\check{\zeta}_{p+1}=d_0\gamma_{p+1}$. Moreover, one can also consider $z_{p+2}=\check{z}_{p+2}+\overline{z}_{p+2}+\hat{z}_{p+2}$ for any choice of $\check{z}_{p+2}+\overline{z}_{p+2}\in\ker d_0\cap\Omega^{p+2,\bullet}(\G)$. Therefore, 
\begin{align*}
    d\alpha=&d_0\alpha+d_1\alpha+d_2\alpha+d_3\alpha+\Omega^{\ge p+4,\bullet}(\G)\\=&d_0y_{p+1}+d_1y_{p+1}+d_0\hat{y}_{p+2}+d_3\alpha+\Omega^{\ge p+4,\bullet}(\G)\\=&d(y_{p+1})-d_2y_{p+1}+dy_{p+2}-d_1y_{p+2}+d_3\alpha+\Omega^{\ge p+4,\bullet}(\G)\\=&d(y_{p+1}+y_{p+2})+\partial_3\overline{\alpha}-\partial_2(\overline{z}_{p+1}+\overline{\zeta}_{p+1})-d_1\overline{z}_{p+2}+d_0\big[d_1\beta_{p+2}+d_2(\beta_{p+1}+\gamma_{p+1})-d_3\beta_p\big]+\Omega^{\ge p+4,\bullet}(\G)\\=&d(y_{p+1}+y_{p+2})+\partial_3\overline{\alpha}-\partial_2\overline{z}_{p+1}-\underbrace{\partial_2\overline{\zeta}_{p+1}-d_1\overline{z}_{p+2}+\operatorname{Im}d_0\cap\Omega^{p+3,\bullet}(\G)}_{\in B_3^{p+3,\bullet}}+\Omega^{\ge p+4,\bullet}(\G)\\
    =&d(y_{p+1}+y_{p+2})+\Delta_3\alpha+B_3^{p+3,\bullet}+\Omega^{\ge p+4,\bullet}(\G)\,.
\end{align*}
Using the fact that $B_3^{p+3,\bullet}=dZ_2^{p+1}+\Omega^{\ge p+4,\bullet}(\G)$, one can use the Leibniz rule again to show that
\begin{align*}
    d(\alpha-y_{p+1}-y_{p+2}-y_{p+3})=\Delta_3\alpha+\Omega^{\ge p+4,\bullet}(\G)\,,
\end{align*}
that is, the differential $\Delta_3\colon Z_3^{p,\bullet}\longrightarrow E_{m,3}^{p+3,\bullet}$ is the same as the exterior derivative $d$, up to an exact form $d\eta$ with $\eta\in\Omega^{p+1,\bullet}(\G)\oplus\Omega^{p+2,\bullet}(\G)\oplus\Omega^{p+3,\bullet}(\G)$ and forms of weight $\ge p+4$. When taking the integral over a $C^1$ manifold $\Sigma$ for which $\operatorname{deg}(\Sigma)=p+3$ whose boundary is a $C^1$ manifold $\partial\Sigma$ for which $\operatorname{deg}(\partial\Sigma)=p$, we get exactly the Stokes' theorem of Theorem \ref{thm: stokes for spectral complexes} directly from the classical Stokes' theorem together with the Leibniz rule.

\section{Currents in positively graded groups}\label{section: currents}

The purpose of this section is to show how Stokes' theorem for spectral complexes automatically introduces the natural setting for a general theory of currents in positively graded groups. 
It is, in fact, quite natural to replace a spectral manifold with a continuous operator acting on the smooth compactly supported forms of the spectral complexes $\{(E_{j,l}^{\bullet,\bullet},\Delta_j)\}_{j\in I_{\bullet,\bullet}}$. 

In this section, $\G$ will denote a positively graded group and $U\subset\G$ an open subset. We also introduce the spaces ${\mathscr D}^{p,k-p}_{j,l}(U)$ of compactly supported smooth forms of $E^{p,k-p}_{j,l}(U)$, viewed as quotient spaces. Before introducing the precise definition of such quotient spaces, the following remark is needed.
\begin{remark}[Viewing $Z_r^{\bullet,\bullet}$ and $B_r^{\bullet,\bullet}$ as quotients]\label{remark:viewZandBas quotients}
    Let us consider the spaces of $r$-cocycles and $r$-boundaries that are traditionally used to define the quotients $E_r^{\bullet,\bullet}$ of spectral sequences (see e.g. \cite{mccleary2001user,livernet2020spectral}). Given $U\subseteq\G$ an open subset, for $r\ge 1$ we have
    \begin{align*}
        \mathcal{Z}_r^{p,\bullet}(U):=&\{\alpha\in\Omega^{\ge p,\bullet}(U)\mid d\alpha\in\Omega^{\ge p+r,\bullet}(U)\}\ \text{ and }\ 
        \mathcal{B}_r^{p,\bullet}(U):=\mathcal{Z}_{r-1}^{p+1,\bullet}(U)+d\mathcal{Z}_{r-1}^{p-r+1,\bullet}(U)\,.
    \end{align*}
    Then it is straightforward to show the following isomorphisms: for each $r\ge 1$, 
    \begin{align}\label{eq: Z and B as quotients}
        Z_r^{p,\bullet}(U)\cong \mathcal{Z}_r^{p,\bullet}(U)/\mathcal{Z}_{r-1}^{p+1,\bullet}(U)\ \text{ and }\ B_r^{p,\bullet}(U)\cong\mathcal{B}_r^{p,\bullet}(U)/\mathcal{Z}_{r-1}^{p+1,\bullet}(U)\,,
    \end{align}
    where $Z_r^{p,\bullet}(U)$ and $B_r^{p,\bullet}(U)$ are the $\R$-submodules of $\Omega^\bullet(U)$ constructed according to Definition \ref{Z and B defined}.

    For a thorough exposition of these properties, we refer to \cite{FILIPPA}.
\end{remark}
\begin{definition}Given $U\subseteq\G$ an open subset, for any $j,l\ge 1$ and $p=0,\ldots,Q$, we define
\begin{align*}
    \mathscr D_{j,l}^{p,\bullet}(U):={\{\alpha\in Z_j^{p,\bullet}(U)\cap \Omega_c^\bullet(U)\}}\ /\ \overline{\{\beta\in B_l^{p,\bullet}(U)\cap\Omega_c^\bullet(U)\}}\,.
\end{align*}
\end{definition}
 The closure here is taken in order to have the isomorphism between these quotients $\mathscr D_{j,l}^{\bullet,\bullet}(U)$ and the subspaces of compactly supported forms in $E_{j,l}^{\bullet,\bullet}(U)$, that is 
\begin{align*}
    \mathscr{D}_{j,l}^{p,\bullet}(U)\cong \{\alpha\in Z_j^{p,\bullet}(U)\cap \Omega_c^\bullet(U)\}\cap\big(\{\beta\in B_l^{p,\bullet}(U)\cap\Omega_c^\bullet(U)\}\big)^\perp\,.
\end{align*}
We remark that the closure is taken with respect to the $L^2$-norm, but the $\R$-modules $Z_j^{p,\bullet}(U)$ and $B_l^{p,\bullet}(U)$ are viewed via the isomorphism \eqref{eq: Z and B as quotients}. In other words, we do not view a form in either $Z_j^{p,\bullet}(U)$ or $B_l^{p,\bullet}(U)$ as purely its representative of weight $p$, as this would also prevent the subspaces $Z_j^{p,\bullet}(U)$ from being closed.

Using this characterisation of $\mathscr{D}_{j,l}^{\bullet,\bullet}(U)$, as well as the submodules $Z_r^{\bullet,\bullet}(U)$ and $B_r^{\bullet,\bullet}(U)$, as quotients, there is a natural way of extending norms from the space of compactly supported forms. Indeed, given a norm $\Vert\cdot\Vert_{X}$ on a Banach space $X$, we define the norm $\Vert\cdot\Vert_{X/Y}$ on the quotient $X/Y$ as follows
\begin{align}\label{eq: norm on quotient}
    \Vert [x]\Vert_{X/Y}:=\inf_{y\in Y}\Vert x-y\Vert_X\ \text{ for any }x\in X\,.
\end{align}
\begin{definition}[Seminorms on $\mathscr D_{j,l}^{p,\bullet}$]

For any compactly supported $\alpha\in Z_r^{p,\bullet}(U)\cap\Omega_c^\bullet(U)$ and $\beta\in B_r^{p,\bullet}(U)\cap\Omega_c^\bullet(U)$, we can use the isomorphisms \eqref{eq: Z and B as quotients} and formula \eqref{eq: norm on quotient} to define the seminorms 
\begin{align*}
    &\mathfrak{p}_{K,i}([\alpha]):=\inf\{p_{K,i}(\alpha+\Tilde{\alpha})\mid \Tilde{\alpha}\in \mathcal{Z}_{r-1}^{p+1,\bullet}(U)\}\ \text{ for any }[\alpha]\in\mathcal{Z}_r^{p,\bullet}(U)/\mathcal{Z}_{r-1}^{p+1,\bullet}(U)\\&\mathfrak{p}_{K,i}([\beta]):=\inf\{p_{K,i}(\beta+\Tilde{\beta})\mid \Tilde{\beta}\in\mathcal{Z}_{r-1}^{p+1,\bullet}(U)\}\ \text{ for any }[\beta]\in\mathcal{B}_r^{p,\bullet}(U)/\mathcal{Z}_{r-1}^{p+1,\bullet}(U)\,,
\end{align*}
where $p_{K,i}$ denotes the seminorm on compactly supported forms $\omega\in\Omega^k_c(U)$
\[
p_{K,i}(\omega)=\max\big\{\|\partial^\alpha \omega_I\|_{C^0(K)}\mid \omega=\sum_I \omega_I \,\theta_I, \;I=(i_1,\ldots,i_k)\in I(k,n),\, |\alpha|\le i  \big\}
\]
and  $\theta_I=\theta_{i_1}\wedge\cdots\wedge \theta_{i_k}\in\bigwedge^k\mathfrak{g}^\ast$.

Finally, for any $\omega\in\mathscr D_{j,l}^{p,\bullet}(U)=Z_j^{p,\bullet}(U)\cap\Omega_c^\bullet(U)/ \,\overline{B_l^{p,\bullet}(U)\cap\Omega_c^\bullet(U)}$, we have
\begin{align*}
    \Tilde{p}_{K,i}(\omega):=\inf\{\mathfrak{p}_{K,i}(\omega+\Tilde{\omega})\mid \Tilde{\omega}\in\overline{B_{l}^{p,\bullet}(U)\cap\Omega^\bullet_c(U)}\}\,,
\end{align*}
hence allowing to construct the associated {\em strict inductive limit topology} on $\mathscr D_{j,l}^{p,\bullet}(U)$.
\end{definition}

\begin{definition}[Spectral current and its boundary]
We say that $T:\mathscr D^{p,k-p}_{j,l}(U)\to\R$ is a $(p,k,j,l)$-{\em spectral current in $U$}, or simply a {\em spectral current in $U$}, if it is linear and continuous with respect to
the strict inductive limit topology on $\mathscr D^{p,k-p}_{j,l}(U)$.
The {\em boundary} of $T$ is naturally defined as 
\[
\partial^l T(\omega)=T(\Delta_l \omega) \quad \text{for each}\; \omega\in\mathscr{ D}^{p-l,k-1-p+l}_{l,m}(U)
\]
where we have set the spectral differential 
\[
\Delta_l\colon E_{l,m}^{p-l,k-1-p+l}\longrightarrow E_{j,l}^{p,k-p}\,,
\]
for an appropriate choice of $m\in\Z^+$.

The continuity of the boundary of $T$ follows directly from the characterisation of the differential $\Delta_j$ via the exterior derivative (see Proposition \ref{prop: expression Delta spectral}) and the well-posedness of the seminorms $\Tilde{p}_{K,i}$ defined on $\mathscr D_{j,l}^{p,\bullet}(U)$. In other words, the properties of the boundary operator $\partial^j$ follow directly from the properties of the boundary operator $\partial$ of a classical current, given by the exterior derivative. 

Therefore, given $T$ a $(p+j,k,m,j)$-spectral current in $U$, we have that $\partial^jT$ is a $(p,k-1,j,l)$-spectral current. Again, the choices of the indices $m,l\in\Z^+$ are not uniquely defined, and are chosen depending on the spectral complex so that both $E_{j,l}^{p,k-1-p}$ and $E_{m,j}^{p+j,k-p-j}$ are nontrivial.

If $\Sigma$ is an oriented $k$-dimensional {\em spectral manifold} of class $C^1$, then we set
\[
[[\Sigma]]_p(\omega):=\int_\Sigma\omega\ \text{ for any }\omega\in\mathscr D_{j,l}^{p,k-p}
\]
to denote the $(p,k,j,l)$-spectral current associated with the spectral manifold $\Sigma$ of degree $p$. Here we use the notation $[[\Sigma]]_p$ is to highlight the fact that $\operatorname{deg}(\Sigma)=p$.

In the case of an oriented $k$-dimensional $C^1$ spectral manifold $\Sigma$ with boundary $\partial\Sigma$, then by Stokes' theorem for spectral complexes (Theorem \ref{thm: stokes for spectral complexes}), we get
\begin{align*}
    [[\partial^j\Sigma]]_p(\omega)=\int_{\partial\Sigma}\omega=\int_\Sigma\Delta_j\omega=[[\Sigma]]_{p+j}(\Delta_j\omega)\ \text{ for any }\omega\in \mathscr D_{j,l}^{p,k-1-p}\,.
\end{align*}
\end{definition}

Inspired by the classical definition of comass in \cite{Federer69}, we define a modified notion of comass adapted to the bigraded structure of $\Omega^\bullet_c(\G)$ (see Definition~\ref{def: weights of forms}). This modification is designed to capture the additional geometric information encoded by the homogeneous weight decomposition.

\begin{definition}[Weighted comass of a differential form]\label{def:weightedcomass}   Let $\theta\in\bigwedge^k\mathfrak{g}^\ast$ and let $W_k$ denote all the possible weights of covectors in degree $k$, then we can introduce a weighted-comass as
\begin{align*}
    \Vert\theta\Vert_p:=\sup\{\theta(v)\mid v\in\bigwedge\nolimits^k\mathfrak{g},\,v\text{ is simple and of weight/degree }p,\text{ and }\vert v\vert\le 1\}\,.
\end{align*}
For each $p\in W_k$, we will refer to $\Vert\theta\Vert_p$ as the $p$-comass of $\theta\in\bigwedge^k\mathfrak{g}^\ast$.

We note that in the case where $\theta\in\bigwedge^k\mathfrak{g}^\ast$, is a homogenous $k$-covector of weight $p\in W_k$, its comass $\Vert\theta\Vert$ in the classical sense (see e.g. \cite[Subection 1.8.1]{Federer69}) coincides with $\Vert\theta\Vert_p$ by the orthogonality of covectors of different weight \eqref{eq: left diff weight implies orthogonal}.

This definition readily extends to arbitrary compactly supported forms, for that, given $\omega\in\Omega^k_c(U)$, we define its $p$-comass as
\begin{align*}
    \Vert\omega\Vert_p :=\sup\{\Vert\omega(x)\Vert_p\mid x\in U\}\ \text{ for a given }p\in W_k\,.
\end{align*}

\end{definition}

\begin{definition}[Comass of a form in $Z_j^{p,k}$]
    Let $\omega\in Z_j^{p,k}(U)\cap\Omega^{\bullet}_c(U)$. Then by the isomorphism \ref{eq: Z and B as quotients}, we can view the form $\omega$ as a linear combination $\omega-z_{p+1}-\cdots-z_{p+j-1}\in\oplus_{i=0}^{j-1}\Omega^{p+i,k-p-i}_c(U)$, where forms $z_{p+i}\in\Omega^{p+i,k-p-i}_c(U)$ are uniquely determined up to elements in $\mathcal{Z}_{j-1}^{p+1}(U)\cap\Omega_c^{\bullet}(G)$. We then define its comass as the $p$-comass of the form $\omega-z_{p+1}-\cdots-z_{p+j-1}$, that is
    \begin{align*}
        \Vert\omega\Vert=\Vert\omega\Vert_p=\Vert \omega-z_{p+1}-\cdots-z_{p+j-1}\Vert_p=\Vert \omega-z_{p+1}-\cdots-z_{p+j-1}+\mathcal{Z}_{j-1}^{p+1}(U)\cap\Omega_c^{\bullet}(G)\Vert_p\,.
    \end{align*}
    The first equality follows from the fact that $\omega$ is a homogeneous form of weight $p$, while the others are due to the fact that the $p$-comass of forms of weight strictly greater than $p$ vanish identically.
\end{definition}
\begin{remark}
    Note that this definition mimics exactly the simplifications needed in Theorem \ref{thm: stokes for spectral complexes} in order to obtain the validity of Stokes' theorem when integrating over a manifold of a given fixed degree $p$ equal to the weight of the form $\omega\in Z_j^{p,\bullet}$ considered.
\end{remark}

\begin{definition}[Mass of a spectral current]
    Let $T\colon \mathscr D_{j,l}^{p,k-p}(U)\longrightarrow\R$ be a $(p,k,j,l)$-spectral current in $U$, we define its $p$-mass as
    \begin{align*}
        \mathbf{M}(T):=\sup\{T(\omega)\mid \omega\in\mathscr D_{j,l}^{p,k-p}\ \text{ and }\Vert\omega\Vert\le 1\}\,.
    \end{align*}
\end{definition}

Since we have shown that, given $T$ a $(p+j,k,m,j)$-current in $U$, its boundary $\partial^jT$ is a $(p,k-1,j,l)$-current in $U$, and so it is possible to compute its mass to. Similarly to what has been done in the literature, we propose the definition of \textit{normal spectral currents} as those for which both $T$ and $\partial^j T$ have finite mass, so that $\mathbf{N}(T)=\mathbf{M}(T)+\mathbf{M}(\partial^j T)$.

\bigskip
\bibliographystyle{amsplain}
\bibliography{bibliography}

@article {JulNGolVit2023,
    AUTHOR = {Julia, Antoine and Nicolussi Golo, Sebastiano and Vittone,
              Davide},
     TITLE = {Lipschitz functions on submanifolds of {H}eisenberg groups},
   JOURNAL = {Int. Math. Res. Not. IMRN},
  FJOURNAL = {International Mathematics Research Notices. IMRN},
      YEAR = {2023},
    NUMBER = {9},
     PAGES = {7399--7422},
      ISSN = {1073-7928,1687-0247},
   MRCLASS = {22E25 (49Q15 49Q20 53C17)},
  MRNUMBER = {4584704},
MRREVIEWER = {Tijana\ \v Sukilovi\'c},
       DOI = {10.1093/imrn/rnac066},
       URL = {https://doi.org/10.1093/imrn/rnac066},
}

@book{mccleary2001user,
  title={A user's guide to spectral sequences},
  author={McCleary, John},
  number={58},
  year={2001},
  publisher={Cambridge University Press}
}

@article{franchi2025currents,
  title={Currents in {H}eisenberg groups},
  author={Franchi, Bruno and Pansu, Pierre},
  journal={arXiv preprint arXiv:2511.18895},
  year={2025}
}

@book {Morgan2016-bk,
    AUTHOR = {Morgan, Frank},
     TITLE = {Geometric measure theory},
   EDITION = {Fifth},
      NOTE = {A beginner's guide,
              Illustrated by James F. Bredt},
 PUBLISHER = {Elsevier/Academic Press, Amsterdam},
      YEAR = {2016},
     PAGES = {viii+263},
      ISBN = {978-0-12-804489-6},
   MRCLASS = {49-01 (26-01 28-01 28A75 49Q20 53C23 58E12)},
  MRNUMBER = {3497381},
}

@book {Federer69,
    AUTHOR = {Federer, Herbert},
     TITLE = {Geometric measure theory},
    SERIES = {Die Grundlehren der mathematischen Wissenschaften},
    VOLUME = {Band 153},
 PUBLISHER = {Springer-Verlag New York, Inc., New York},
      YEAR = {1969},
     PAGES = {xiv+676},
   MRCLASS = {28.80 (26.00)},
  MRNUMBER = {257325},
MRREVIEWER = {J.\ E.\ Brothers},
}

@article {Cygan81,
    AUTHOR = {Cygan, Jacek},
     TITLE = {Subadditivity of homogeneous norms on certain nilpotent {L}ie
              groups},
   JOURNAL = {Proc. Amer. Math. Soc.},
  FJOURNAL = {Proceedings of the American Mathematical Society},
    VOLUME = {83},
      YEAR = {1981},
    NUMBER = {1},
     PAGES = {69--70},
      ISSN = {0002-9939},
   MRCLASS = {22E30 (35H05 43A80)},
  MRNUMBER = {619983},
       DOI = {10.2307/2043893},
       URL = {http://dx.doi.org/10.2307/2043893},
}

@article {DiD21,
    AUTHOR = {Di Donato, Daniela},
     TITLE = {Intrinsic differentiability and intrinsic regular surfaces in
              {C}arnot groups},
   JOURNAL = {Potential Anal.},
  FJOURNAL = {Potential Analysis. An International Journal Devoted to the
              Interactions between Potential Theory, Probability Theory,
              Geometry and Functional Analysis},
    VOLUME = {54},
      YEAR = {2021},
    NUMBER = {1},
     PAGES = {1--39},
      ISSN = {0926-2601},
   MRCLASS = {35R03 (22E25 28A75 53C17)},
  MRNUMBER = {4194533},
       DOI = {10.1007/s11118-019-09817-4},
       URL = {https://doi.org/10.1007/s11118-019-09817-4},
}

@PhDThesis{CorniPhD,
title = {Low codimensional intrinsic regular submanifolds in the Heisenberg group $\mathbb H^n$},
author = {Francesca Corni},
school = {Universit{\`a} di Bologna},
year= {2021},
}

@article {ADDDLD24,
    AUTHOR = {Antonelli, Gioacchino and Di Donato, Daniela and Don,
              Sebastiano and Le Donne, Enrico},
     TITLE = {Characterizations of uniformly differentiable co-horizontal
              intrinsic graphs in {C}arnot groups},
   JOURNAL = {Ann. Inst. Fourier (Grenoble)},
  FJOURNAL = {Universit\'e{} de Grenoble. Annales de l'Institut Fourier},
    VOLUME = {74},
      YEAR = {2024},
    NUMBER = {6},
     PAGES = {2523--2621},
      ISSN = {0373-0956,1777-5310},
   MRCLASS = {53C17 (22E25 26A16 28A75 49Q15)},
  MRNUMBER = {4810245},
       DOI = {10.5802/aif.3660},
       URL = {https://doi.org/10.5802/aif.3660},
}

@article {HebSik90,
    AUTHOR = {Hebisch, Waldemar and Sikora, Adam},
     TITLE = {A smooth subadditive homogeneous norm on a homogeneous group},
   JOURNAL = {Studia Math.},
  FJOURNAL = {Polska Akademia Nauk. Instytut Matematyczny. Studia
              Mathematica},
    VOLUME = {96},
      YEAR = {1990},
    NUMBER = {3},
     PAGES = {231--236},
      ISSN = {0039-3223},
   MRCLASS = {22E25 (17B30 22E15 43A85)},
  MRNUMBER = {1067309},
MRREVIEWER = {Hidenori Fujiwara},
       DOI = {10.4064/sm-96-3-231-236},
       URL = {https://doi.org/10.4064/sm-96-3-231-236},
}

@article{FILIPPA,
    author = {Lo Biundo, Filippa},
    title = {Distributional spectral complexes and bi{L}ipschitz invariance},
    journal = {in preparation},
    year = {}
}

@article {FSSC01,
    AUTHOR = {Franchi, Bruno and Serapioni, Raul and Serra Cassano, Francesco},
     TITLE = {Rectifiability and perimeter in the {H}eisenberg group},
   JOURNAL = {Math. Ann.},
  FJOURNAL = {Mathematische Annalen},
    VOLUME = {321},
      YEAR = {2001},
    NUMBER = {3},
     PAGES = {479--531},
      ISSN = {0025-5831},
     CODEN = {MAANA},
   MRCLASS = {49Q15 (22E25 46E35)},
  MRNUMBER = {1871966 (2003g:49062)},
MRREVIEWER = {Piotr Haj{\l}asz},
       DOI = {10.1007/s002080100228},
       URL = {http://dx.doi.org/10.1007/s002080100228},
}

@article {Vit22,
    AUTHOR = {Vittone, Davide},
     TITLE = {Lipschitz graphs and currents in {H}eisenberg groups},
   JOURNAL = {Forum Math. Sigma},
  FJOURNAL = {Forum of Mathematics. Sigma},
    VOLUME = {10},
      YEAR = {2022},
     PAGES = {Paper No. e6, 104},
   MRCLASS = {49Q15 (22E30 26A16 53C17)},
  MRNUMBER = {4377000},
MRREVIEWER = {Yongsheng Zhang},
       DOI = {10.1017/fms.2021.84},
       URL = {https://doi.org/10.1017/fms.2021.84},
}

@article {Magnani2019Area,
    AUTHOR = {Magnani, Valentino},
     TITLE = {Towards a theory of area in homogeneous groups},
   JOURNAL = {Calc. Var. Partial Differential Equations},
  FJOURNAL = {Calculus of Variations and Partial Differential Equations},
    VOLUME = {58},
      YEAR = {2019},
    NUMBER = {3},
     PAGES = {58:91},
      ISSN = {0944-2669},
   MRCLASS = {28A75 (22E30 53C17)},
  MRNUMBER = {3947860},
       DOI = {10.1007/s00526-019-1539-7},
       URL = {https://doi.org/10.1007/s00526-019-1539-7},
}

@article {Mag12A,
    AUTHOR = {Magnani, Valentino},
     TITLE = {Non-horizontal submanifolds and coarea formula},
   JOURNAL = {J. Anal. Math.},
  FJOURNAL = {Journal d'Analyse Math\'ematique},
    VOLUME = {106},
      YEAR = {2008},
     PAGES = {95--127},
      ISSN = {0021-7670},
   MRCLASS = {53C17 (22E25 49Q15)},
  MRNUMBER = {2448983},
MRREVIEWER = {C{\'e}sar Rosales},
       DOI = {10.1007/s11854-008-0045-1},
       URL = {http://dx.doi.org/10.1007/s11854-008-0045-1},
}

@article {Mag06,
    AUTHOR = {Magnani, Valentino},
     TITLE = {Blow-up of regular submanifolds in {H}eisenberg groups and applications},
   JOURNAL = {Cent. Eur. J. Math.},
  FJOURNAL = {Central European Journal of Mathematics},
    VOLUME = {4},
      YEAR = {2006},
    NUMBER = {1},
     PAGES = {82--109},
      ISSN = {1895-1074,1644-3616},
   MRCLASS = {53C17 (22E25 28A75 49Q15)},
  MRNUMBER = {2213028},
MRREVIEWER = {Lars\ Olsen},
       DOI = {10.1007/s11533-005-0006-1},
       URL = {https://doi.org/10.1007/s11533-005-0006-1},
}

@article {FSSC6,
    AUTHOR = {Franchi, Bruno and Serapioni, Raul and Serra Cassano,
              Francesco},
     TITLE = {Regular submanifolds, graphs and area formula in {H}eisenberg
              groups},
   JOURNAL = {Adv. Math.},
  FJOURNAL = {Advances in Mathematics},
    VOLUME = {211},
      YEAR = {2007},
    NUMBER = {1},
     PAGES = {152--203},
      ISSN = {0001-8708},
   MRCLASS = {49Q15 (22E25 28A75 28A78)},
  MRNUMBER = {2313532},
MRREVIEWER = {Oleg G. Okunev},
       DOI = {10.1016/j.aim.2006.07.015},
       URL = {https://doi.org/10.1016/j.aim.2006.07.015},
}

@article {Can21jga,
    AUTHOR = {Canarecci, Giovanni},
     TITLE = {Sub-{R}iemannian currents and slicing of currents in the
              {H}eisenberg group {$\Bbb{H}^n$}},
   JOURNAL = {J. Geom. Anal.},
  FJOURNAL = {Journal of Geometric Analysis},
    VOLUME = {31},
      YEAR = {2021},
    NUMBER = {5},
     PAGES = {5166--5200},
      ISSN = {1050-6926,1559-002X},
   MRCLASS = {53C17 (53A35)},
  MRNUMBER = {4244901},
MRREVIEWER = {Francescopaolo\ Montefalcone},
       DOI = {10.1007/s12220-020-00474-3},
       URL = {https://doi.org/10.1007/s12220-020-00474-3},
}

@article{julia2023flat,
  title={Flat compactness of normal currents, and charges in {C}arnot groups},
  author={Julia, Antoine and Pansu, Pierre},
  journal={arXiv preprint arXiv:2303.02012},
  year={2023}
}

@article{corni2026minimal,
  title={A minimal regularity for the area formula in the {E}ngel group},
  author={Corni, Francesca and Essebei, Fares and Magnani, Valentino},
  journal={arXiv preprint arXiv:2602.00346},
  year={2026}
}

@article{tripaldi2026spectralcomplexestruncatedmulticomplexes,
      title={Spectral complexes from truncated multicomplexes}, 
      author={Francesca Tripaldi},
      journal={arXiv preprint, arXiv:2602.01214},
      year={2026},
}

@article{hakavuori2022gradings,
  title={Gradings for Nilpotent {L}ie Algebras},
  author={Hakavuori, Eero and Kivioja, Ville and Moisala, Terhi and Tripaldi, Francesca},
  journal={Journal of Lie Theory},
  volume={32},
  number={2},
  pages={383--412},
  year={2022}
}

@article{livernet2020spectral,
  title={On the spectral sequence associated to a multicomplex},
  author={Livernet, Muriel and Whitehouse, Sarah and Ziegenhagen, Stephanie},
  journal={Journal of Pure and Applied Algebra},
  volume={224},
  number={2},
  pages={528--535},
  year={2020},
  publisher={Elsevier}
}

@incollection {rumin_grenoble,
    AUTHOR = {Rumin, Michel},
     TITLE = {Around heat decay on forms and relations of nilpotent {L}ie
              groups},
 BOOKTITLE = {S\'eminaire de Th\'eorie Spectrale et G\'eom\'etrie, Vol. 19,
              Ann\'ee 2000--2001},
    SERIES = {S\'emin. Th\'eor. Spectr. G\'eom.},
    VOLUME = {19},
     PAGES = {123--164},
 PUBLISHER = {Univ. Grenoble I},
      YEAR = {2001},
   MRCLASS = {58J50 (53C17 58A10 58J10)},
  MRNUMBER = {MR1909080 (2003f:58062)},
MRREVIEWER = {Emmanuel Russ},
}

@article{LeDonneTripaldi2021,
  title={A {C}ornucopia of {C}arnot groups in low dimensions},
  author={Le Donne, Enrico and Tripaldi, Francesca},
pages = {155--289},
volume = {10},
number = {1},
journal = {Analysis and Geometry in Metric Spaces},
doi = {doi:10.1515/agms-2022-0138},
year = {2022},
lastchecked = {2026-03-01}
}

@book{Fischer2016,
author="Fischer, Veronique
and Ruzhansky, Michael",
title="Quantization on Nilpotent Lie Groups",
year="2016",
publisher="Springer International Publishing",
address="Cham",
pages="15--56",
isbn="978-3-319-29558-9",
doi="10.1007/978-3-319-29558-9_1",
url="https://doi.org/10.1007/978-3-319-29558-9_1"
}

@book{montgomery2002tour,
  title={A {T}our of {S}ubriemannian {G}eometries, {T}heir {G}eodesics and {A}pplications},
  author={Montgomery, Richard},
  number={91},
  year={2002},
  publisher={American Mathematical Soc.}
}

@article{FT,
  title={Differential {F}orms in {C}arnot {G}roups {A}fter {M}. {R}umin: an {I}ntroduction},
  author={Franchi, Bruno and Tripaldi, Francesca},
  journal={Quaderni dell’Unione Matematica Italiana, Topics in mathematics, Bologna, Pitagora},
  pages={75--122},
  year={2015}
}

@article{rumin1999differential,
  title={Differential geometry on {CC} spaces and application to the {N}ovikov-{S}hubin numbers of nilpotent {L}ie groups},
  author={Rumin, Michel},
  journal={Comptes Rendus de l'Acad{\'e}mie des Sciences-Series I-Mathematics},
  volume={329},
  number={11},
  pages={985--990},
  year={1999},
  publisher={Elsevier}
}

@article {CompensatedCompactness,
    AUTHOR = {Baldi, Annalisa and Franchi, Bruno and Tchou, Nicoletta and
              Tesi, Maria Carla},
     TITLE = {Compensated compactness for differential forms in {C}arnot
              groups and applications},
   JOURNAL = {Adv. Math.},
  FJOURNAL = {Advances in Mathematics},
    VOLUME = {223},
      YEAR = {2010},
    NUMBER = {5},
     PAGES = {1555--1607},
      ISSN = {0001-8708,1090-2082},
   MRCLASS = {43A80 (35B27 58A10)},
  MRNUMBER = {2592503},
       DOI = {10.1016/j.aim.2009.09.020},
       URL = {https://doi.org/10.1016/j.aim.2009.09.020},
}

@article {CorMag25,
    AUTHOR = {Corni, Francesca and Magnani, Valentino},
     TITLE = {Symmetry results for the area formula in homogeneous groups},
   JOURNAL = {J. Math. Anal. Appl.},
  FJOURNAL = {Journal of Mathematical Analysis and Applications},
    VOLUME = {546},
      YEAR = {2025},
    NUMBER = {2},
     PAGES = {Paper No. 129238, 17},
      ISSN = {0022-247X,1096-0813},
   MRCLASS = {53C17 (22E30 28A75)},
  MRNUMBER = {4853462},
       DOI = {10.1016/j.jmaa.2025.129238},
       URL = {https://doi.org/10.1016/j.jmaa.2025.129238},
}

@article {Mag22RS,
    AUTHOR = {Magnani, Valentino},
     TITLE = {Rotational symmetries and spherical measure in homogeneous
              groups},
   JOURNAL = {J. Geom. Anal.},
  FJOURNAL = {Journal of Geometric Analysis},
    VOLUME = {32},
      YEAR = {2022},
    NUMBER = {4},
     PAGES = {Paper No. 119, 31},
      ISSN = {1050-6926},
   MRCLASS = {43A80 (22E25 28A75 58C35)},
  MRNUMBER = {4375063},
MRREVIEWER = {Sergei S. Platonov},
       DOI = {10.1007/s12220-022-00874-7},
       URL = {https://doi.org/10.1007/s12220-022-00874-7},
}

@article{Mag14,
author={Magnani, Valentino},
title={Towards Differential Calculus in stratified groups},
journal = {J. Aust. Math. Soc.},
volume={95},
number = "1",
year="2013",
pages = "76-128"
}

@article{Mag13Vit,
author={Magnani, Valentino and Vittone, Davide},
title={An intrinsic measure for submanifolds in stratified groups},
journal = {J. Reine Angew. Math.},
volume={619},
pages={203--232},
year={2008},
language={English}
}

@book{donne2024metric,
author="Le Donne, Enrico",
title="Metric {L}ie Groups: Carnot-{C}arath{\'e}odory {S}paces from the {H}omogeneous {V}iewpoint",
year="2025",
publisher="Springer Nature Switzerland",
address="Cham",
pages="1--21",
isbn="978-3-031-98832-5",
doi="10.1007/978-3-031-98832-5_1",
url="https://doi.org/10.1007/978-3-031-98832-5_1"
}

@article {FMS14,
    AUTHOR = {Franchi, Bruno and Marchi, Marco and Serapioni, Raul Paolo},
     TITLE = {Differentiability and approximate differentiability for
              intrinsic {L}ipschitz functions in {C}arnot groups and a
              {R}ademacher theorem},
   JOURNAL = {Anal. Geom. Metr. Spaces},
  FJOURNAL = {Analysis and Geometry in Metric Spaces},
    VOLUME = {2},
      YEAR = {2014},
     PAGES = {258--281},
      ISSN = {2299-3274},
   MRCLASS = {49Q15 (53C17 58C20)},
  MRNUMBER = {3290378},
MRREVIEWER = {Laurent Moonens},
       DOI = {10.2478/agms-2014-0010},
       URL = {https://doi.org/10.2478/agms-2014-0010},
}

@article {FSSC5,
    AUTHOR = {Franchi, Bruno and Serapioni, Raul and Serra Cassano, Francesco},
     TITLE = {On the structure of finite perimeter sets in step 2 {C}arnot
              groups},
   JOURNAL = {J. Geom. Anal.},
  FJOURNAL = {The Journal of Geometric Analysis},
    VOLUME = {13},
      YEAR = {2003},
    NUMBER = {3},
     PAGES = {421--466},
      ISSN = {1050-6926},
   MRCLASS = {49Q15 (53C17)},
  MRNUMBER = {1984849 (2004i:49085)},
       DOI = {10.1007/BF02922053},
       URL = {http://dx.doi.org/10.1007/BF02922053},
}

@article{arena2009intrinsic,
  title={Intrinsic regular submanifolds in {H}eisenberg groups are differentiable graphs},
  author={Arena, Gabriella and Serapioni, Raul and others},
  journal={Calculus of Variations and Partial Differential Equations},
  volume={35},
  number={4},
  pages={517--536},
  year={2009},
  publisher={Springer-Verlag(Heidelberg), Tiergartenstrasse 17}
}

@unpublished{CorMag23pr,
	author = {Corni, F. and Magnani V.},
	title = {Area of intrinsic graphs in homogeneous groups},
	year={November 11, 2023},
	note={arXiv:2311.06638},
}

@article {rumin_palermo,
    AUTHOR = {Rumin, Michel},
     TITLE = {An introduction to spectral and differential geometry in
              {C}arnot-{C}arath\'eodory spaces},
   JOURNAL = {Rend. Circ. Mat. Palermo (2) Suppl.},
  FJOURNAL = {Rendiconti del Circolo Matematico di Palermo. Serie II.
              Supplemento},
    VOLUME = {75},
      YEAR = {2005},
     PAGES = {139--196},
   MRCLASS = {58J35 (53C23)},
  MRNUMBER = {MR2152359 (2006g:58053)},
MRREVIEWER = {Thierry Coulhon},
}

@article {FranchiSerapioni,
    AUTHOR = {Bruno Franchi  and Raul Paolo Serapioni},
     TITLE = {Intrinsic {L}ipschitz graphs within {C}arnot groups},
   JOURNAL = {J. Geom. Anal.},
  FJOURNAL = {Journal of Geometric Analysis},
    VOLUME = {26},
      YEAR = {2016},
    NUMBER = {3},
     PAGES = {1946--1994}
}

@article {FranchiTesi,
    AUTHOR = {Bruno Franchi and Maria Carla Tesi},
     TITLE = {Wave and {M}axwell's equations in {C}arnot groups},
   JOURNAL = {Commun. Contemp. Math.},
  FJOURNAL = {Communications in Contemporary Mathematics},
    VOLUME = {14},
      YEAR = {2012},
    NUMBER = {5},
     PAGES = {1250032, 62}
}

@article {F+T1,
    AUTHOR = {Fischer, V\'{e}ronique and Tripaldi, Francesca},
     TITLE = {An alternative construction of the {R}umin complex on
              homogeneous nilpotent {L}ie groups},
   JOURNAL = {Adv. Math.},
  FJOURNAL = {Advances in Mathematics},
    VOLUME = {429},
      YEAR = {2023},
     PAGES = {Paper No. 109192, 39},
      ISSN = {0001-8708,1090-2082},
   MRCLASS = {58J10 (22E25 53C17)},
  MRNUMBER = {4613357},
       DOI = {10.1016/j.aim.2023.109192},
       URL = {https://doi.org/10.1016/j.aim.2023.109192},
}

@article{StokesFranchi,
author = {Franchi, Bruno and Tchou, Nicoletta and Tesi, Maria Carla},
title = {DIV–CURL TYPE THEOREM, H-CONVERGENCE AND STOKES FORMULA IN THE {H}EISENBERG GROUP},
journal = {Communications in Contemporary Mathematics},
volume = {08},
number = {01},
pages = {67-99},
year = {2006},
doi = {10.1142/S0219199706002039},
URL = {   https://doi.org/10.1142/S0219199706002039

},
eprint = {         https://doi.org/10.1142/S0219199706002039
 }
}

@article {DiMJulNGoloVit25,
    AUTHOR = {Di Marco, Marco and Julia, Antoine and Nicolussi Golo,
              Sebastiano and Vittone, Davide},
     TITLE = {Submanifolds with boundary and {S}tokes' theorem in
              {H}eisenberg groups},
   JOURNAL = {Trans. Amer. Math. Soc.},
  FJOURNAL = {Transactions of the American Mathematical Society},
    VOLUME = {378},
      YEAR = {2025},
    NUMBER = {7},
     PAGES = {4955--4990},
      ISSN = {0002-9947,1088-6850},
   MRCLASS = {53C17 (26B20 53C65 58C35)},
  MRNUMBER = {4919585},
MRREVIEWER = {Christiam\ B.\ Figueroa},
       DOI = {10.1090/tran/9410},
       URL = {https://doi.org/10.1090/tran/9410},
}

@article{dimarco2025submanifoldsboundarysubriemannianheisenberg,
      title={Submanifolds with boundary in sub-{R}iemannian {H}eisenberg Groups}, 
      author={Marco Di Marco and Davide Vittone},
      year={2025},
      eprint={2508.15634},
      archivePrefix={arXiv},
      primaryClass={math.DG},
      url={https://arxiv.org/abs/2508.15634}, 
}
\end{document}